\documentclass[final,1p]{elsarticle}

\usepackage{amssymb}
\usepackage{amsmath}
\usepackage{amsthm}
\usepackage{latexsym}
\usepackage[ruled,lined]{algorithm2e}
\usepackage{algorithmic}
\usepackage{array}
\usepackage{hyperref} 
\usepackage{color,tabularx}
\usepackage{booktabs}
\usepackage{subfigure}
\usepackage{multirow}
\usepackage{stmaryrd}
\usepackage[T1]{fontenc}
\biboptions{numbers,sort&compress}
\usepackage{graphicx} 
\usepackage{adjustbox}
\usepackage[mathlines, pagewise]{lineno} 

\newcommand{\f}{\frac}
\newcommand{\p}{\partial}

\newcommand{\me}{\mathcal{ E }}

\newcommand{\mt}{\mathcal{ T }}
\newcommand{\mg}{\mathcal{G}}
\newcommand{\mm}{\mathcal{M}}

\numberwithin{equation}{section}
\allowdisplaybreaks
\newtheorem{theorem}{Theorem}[section]
\newtheorem{lemma}[theorem]{Lemma}

\theoremstyle{definition}

\theoremstyle{remark}
\newtheorem{remark}{Remark}

\graphicspath{{Figures/}}
\journal{Journal of Computational Physics}
\begin{document}
	

\begin{frontmatter}


\title{High-order nonuniform time-stepping and MBP-preserving linear schemes for the time-fractional Allen--Cahn equation }

\author[ouc]{Bingyin Zhang} \ead{zhangbingyin@stu.ouc.edu.cn}
\author[ouc,lab]{Hongfei Fu\corref{cor1}} \ead{fhf@ouc.edu.cn}

\address[ouc]{School of Mathematical Sciences, Ocean University of China, Qingdao, Shandong 266100, China}
\address[lab]{Laboratory of Marine Mathematics, Ocean University of China, Qingdao, Shandong 266100, China}
\cortext[cor1]{Corresponding author.}

\begin{abstract}
In this paper, we present a class of nonuniform time-stepping, high-order linear stabilized schemes that can preserve both the discrete energy stability and maximum-bound principle (MBP) for the time-fractional Allen--Cahn equation. To this end, we develop a new prediction strategy to obtain a second-order and MBP-preserving predicted solution, which is then used to handle the nonlinear potential explicitly. Additionally, we introduce an essential nonnegative auxiliary functional that enables the design of an appropriate stabilization term to dominate the predicted nonlinear potential, and thus to preserve the discrete MBP. Combining the newly developed prediction strategy and auxiliary functional, we propose two unconditionally energy-stable linear stabilized schemes, $L1$ and $L2$-$1_{\sigma}$ schemes. We show that the $L1$ scheme unconditionally preserves the discrete MBP, whereas the $L2$-$1_{\sigma}$ scheme requires a mild time-step restriction. Furthermore, we develop an improved $L2$-$1_{\sigma}$ scheme with enhanced MBP preservation for large time steps, achieved through a novel unbalanced stabilization term that leverages the boundedness and monotonicity of the auxiliary functional. Representative numerical examples validate the accuracy, effectiveness, and physics-preserving of the proposed methods.
\end{abstract}


\end{frontmatter}

\section{Introduction}
As a diffuse interface model, the classic phase-field model has been widely applied across various research areas, including material sciences \cite{Acta_Allen_1979,JCP_CH_1958}, hydrodynamics \cite{ARFM_Anderson_1998,PRE_Qian_2003}, biology and tumor growth \cite{JCP_Du_2004,IJNMBE_Oden_2012,JTB_Wise_2008,CMAME_Huang_2024}. More recently, there has been growing interest in nonlocal phase-field models \cite{SISC_Tang_2019,JCP_Wang_2017,JDE_Akagi_2016,SINUM_Ainthworth_2017,SINUM_Du_2019,PASMA_Inc_2018,CMA_Wang_2019}. The incorporation of nonlocal operators into phase-field equations has been shown to substantially alter the diffusive dynamics. For example, in space-fractional Allen--Cahn phase field model, the fractional order can regulate the sharpness of the interface \cite{CMAME_Xu_2016}; and in time-fractional phase-field models, the fractional order significantly influences the coarsening behavior \cite{JCP_Wang_2017,SISC_Tang_2019,CPC_Chen_2019}.
Moreover, fractional models have been shown to provide more accurate descriptions of complex phenomena involving anomalous diffusion and memory effects than traditional integer-order models; see, e.g., \cite{JCP_Wang_2017,SISC_Hou_2021,PRL_Golding_2006,Nature_Kirchner_2000,PSS_Nigmatullin_1986}.
In heterogeneous porous media, diffusive transport of solute particles is influenced by the strong interactions between fluids and solids, where a large quantity of solute particles may get absorbed to the solid formation \cite{MS_Sharma_2015}. As a result, the travel times of these adsorbed particles can differ substantially from those of particles that move freely within the bulk phase \cite{JCP_Zhokn_2017}, giving rise to anomalous subdiffusive transport. This regime is characterized by a sublinear growth of the particle mean square displacement $ < r^2 > $  
with respect to the time $t$ \cite{PRL_Chepizhko_2013}. Such anomalous transport behavior can be naturally modeled using time-fractional diffusion models \cite{PR_Metzler_2000,JMAA_Wang_2019}, where the fractional-order parameter quantifies the strength of subdiffusive behavior.

In this paper, we focus on the following phase-field model involving special time nonlocaity, namely the time-fractional Allen--Cahn (tFAC) equation
\begin{equation}\label{Model:tAC}
		{}_0^C D^{\alpha}_t \phi = \mm  ( \varepsilon^2 \Delta \phi +  f(\phi) ) , \quad t >0, \  \mathbf{x} \in \Omega;~~
        \phi(\mathbf{x}, 0)=\phi_{\text {init}}(\mathbf{x}), \  \mathbf{x} \in \bar{\Omega},
\end{equation}
where the spatial domain is denoted by $ \Omega \subset \mathbb{R}^d $ with $ d \le 3$, and the Caputo fractional-order derivative operator $ {}_0^C D^{\alpha}_t $ of order $ \alpha \in (0,1) $ is defined as
\begin{equation*}
		{}_{0}^{C}D^{\alpha}_{t} v = \int^{t}_{0} \omega_{1-\alpha} (t-s)\partial_s v (s) ds,~~
        \omega_{\mu}(t) := \frac{t^{\mu-1}}{\Gamma(\mu)}.
\end{equation*}
In \eqref{Model:tAC}, $ \phi( \mathbf{x}, t ) $ is the unknown phase-field function, $ \mm $  is a  positive constant mobility,  $ \varepsilon > 0 $ is the interaction length that describes the thickness of the transition boundary between materials, and $ f(\phi)= -F'(\phi) $ is a continuously differentiable nonlinear potential function.

It is well known that when $\alpha \rightarrow 1 $, the tFAC equation recovers the classical Allen--Cahn equation, which satisfies the so-called \textit{energy dissipation law}, i.e.,
\begin{equation}\label{AC:energy}
E[ \phi ]( t ) \leq E[ \phi ]( s ),  \quad  \forall t > s \quad \text{with} \quad E[\phi](t) := \int_{\Omega} \Big( \frac{\varepsilon^2}{2} \vert \nabla \phi(\mathbf{x}) \vert^2 
+ F( \phi(\mathbf{x}) ) \Big) d \mathbf{x},
\end{equation}
and the \textit{maximum-bound principle} (MBP), i.e.,
\begin{equation}\label{AC:MBP}
	\max_{\mathbf{x} \in \bar{\Omega}} |\phi_{\text {init }}(\mathbf{x})| \leq \beta \quad \Longrightarrow 
	\quad \max_{\mathbf{x} \in \bar{\Omega}}|\phi(\mathbf{x},t)| \leq \beta, \quad \forall t>0,
\end{equation}
see \cite{SIREV_Du_2021} for details. Regarding the energy stability of the tFAC model,
the authors \cite{SISC_Tang_2019} provided the first theoretical conclusion, i.e.,
\begin{equation}\label{tFAC:energy}
	E[ \phi ]( t ) \leq E[ \phi ]( 0 ), \quad  \forall t > 0,
\end{equation}
which differs from the classical energy dissipation law \eqref{AC:energy}. 
Moreover, it was demonstrated in \cite{JSC_Du_2020} that the double-well tFAC equation also admits the MBP \eqref{AC:MBP}.

Another key feature of the tFAC equation is that its evolution process often requires a long time to reach a steady state and typically involves multiple time scales (see \cite{JCP_Liao_2020,SISC_Liao_2021,JSC_Liao_2024}). 
Therefore, an adaptive time-stepping strategy \cite{SISC_2011_Qiao,JCP_Gomez_2011,JCP_Wodo_2011} serves as a heuristic and available approach to enhance computational efficiency without compromising accuracy. 
Moreover, it is well known that solutions to subdiffusion problems, including the tFAC equation, exhibit weak singularities near the initial time, although they would be smooth away from $ t = 0 $, see \cite{IMA_Jin_2016,SINUM_Martin_2017}.
The initial weak singularity may cause loss of convergence order in numerical simulations based upon a uniform temporal mesh. However, this issue can be alleviated by employing a special nonuniform temporal mesh, e.g., the graded mesh \cite{SINUM_Martin_2017,SINUM_Liao_2018}. These two characteristics of the tFAC equation jointly motivate the development of effective and accurate \textit{nonuniform time-stepping} numerical methods.

To achieve stable numerical simulations and avoid nonphysical solutions, it is hence necessary to develop numerical schemes that can preserve both the discrete energy stability and MBP for the tFAC model \eqref{Model:tAC}. 
In \cite{SISC_Tang_2019}, Tang et al. presented a first-order method combining the uniform $L1$ formula \cite{JCP_Xu_2007} and the stabilization technique \cite{SINUM_Xu_2006}, which preserves the energy stability \eqref{tFAC:energy} and MBP \eqref{AC:MBP} at the discrete level. 
By discretizing the fractional derivative using backward Euler convolution quadrature, three energy-stable and MBP-preserving schemes of order $ O(\tau^{\alpha}) $ were proposed and analyzed in \cite{JSC_Du_2020}.
To achieve high-order (more than first-order) and MBP-preserving numerical methods, the $L1$ scheme \cite{ACM_Ji_2020}, the $L1_{R}$ scheme \cite{SISC_Liao_2021}, the $L2$-$1_{\sigma}$ scheme \cite{JCP_Liao_2020,JSC_Liao_2024}, and the SFTR (shifted fractional trapezoidal rule) scheme \cite{JSC_Huang_2023} were developed. 
However, there are basically two limitations to the above mentioned high-order numerical schemes: 
(i) they treat the nonlinear term either fully or partially implicit, which implies that their unique solvability is uncertain, especially for Flory--Huggins potential case, and a nonlinear iteration must be implemented at each time step; 
(ii) these schemes always require certain restrictions on the time-step to preserve the discrete MBP. 
These motivate us to develop a novel class of high-order linear schemes that are more efficient and can preserve the discrete MBP unconditionally. Although inherently nonlocal, the scalar auxiliary variable (SAV) method \cite{JCP_Shen_2018,SINUM_Shen_2018} has been widely applied to phase-field models \cite{SISC_Hou_2021,SINUM_Qiao_2022,SISC_Zhao_2024,JCP_Quan_2022,SISC_Ji_2020}, owing to its ability to facilitate the design of linear and unconditionally 'modified' energy-stable schemes. To make the modified energy closer to the original one, several improved SAV-type techniques have been proposed recently, including the relaxation methods \cite{JCP_Shen_2022,JCP_Jiang_2022} and energy-optimized (EOP) strategies \cite{JSC_Liu_2024,ANM_2025_Zhang}. Moreover, due to the uncertainty in the signs of the nonlinear term coefficients (typically functions with respect to the SAV), the construction  of MBP-preserving SAV-type schemes are nontrivial and compelling, particularly for high-order time discretization methods or time-fractional model problems \cite{SISC_Hou_2021,JCP_Quan_2022}. 

An inspiring approach is the combinations of prediction-correction method and stabilization technique proposed in \cite{JSC_Qiao_2022,SINUM_Qiao_2022,MOC_Hou_2023} for classical Allen--Cahn type equations. 
They first adopted a first-order scheme to provide a predicted solution with local second-order accuracy and preservation of the MBP, and then presented their second-order schemes by controlling the nonlinear term using an artificial stabilization term. 
Through this strategy, second-order linear schemes with conditional \cite{JSC_Qiao_2022,MOC_Hou_2023} or unconditional \cite{SINUM_Qiao_2022} preservation of the MBP were successfully established for the classical Allen--Cahn equation. 
However, it should be noted that this approach is difficult to directly apply to the tFAC model \eqref{Model:tAC}, as a first-order method no longer provides sufficiently accurate predicted solutions, although this is feasible in the integer-order case. Thus, in the current work, we shall propose a general prediction strategy capable of providing satisfactory predicted solutions. We argue that it can also serve as a more efficient alternative approach for the prediction step in the numerical methods developed in \cite{JSC_Qiao_2022,SINUM_Qiao_2022,MOC_Hou_2023}.

In this paper, we aim to design a class of linear, energy-stable, and MBP-preserving high-order nonuniform time-stepping schemes for the tFAC equation \eqref{Model:tAC}, by leveraging the stabilized exponential SAV approach, namely the sESAV.
To achieve this, we introduce a novel nonnegative auxiliary functional in Section \ref{Sec:L1}, which serves twofold functions: (i) it ensures that the first-order approximation of the SAV does not compromise the temporal accuracy of $\phi$; 
and (ii) it enables the development of an effective stabilization term, allowing us to establish an MBP-preserving scheme.
In combination with the newly developed prediction strategy and auxiliary functional, we first propose a linear stabilized $L1$-sESAV nonuniform time-stepping scheme, which possesses the following remarkable advantages:
\begin{itemize}
	\item it is unconditionally energy-stable;
	
	\item it can unconditionally preserve the discrete MBP, which seems to be the first time such a linear, MBP-preserving, and $(2-\alpha)$th order $L1$ scheme is established;
	
	\item it is computationally efficient, in fact, only systems of linear algebraic equations require to be solved for $ n \geq 2 $.
\end{itemize}
Then, a linear stabilized $L2$-$1_{\sigma}$-sESAV scheme is also investigated and proven to be unconditionally energy-stable, but conditionally MBP-preserving. Moreover, to further improve the $L2$-$1_{\sigma}$-sESAV scheme for large time-step computation, we present an unbalanced $L2$-$1_{\sigma}$-sESAV scheme by introducing a novel stabilization term based on the boundedness and monotonicity of the novel auxiliary functional. To the best of our knowledge, so far no such linear second-order scheme with provable energy stability and MBP-preserving in the discrete setting has been developed. Furthermore, we demonstrate that the proposed method can be coupled with the EOP technique to enhance the modified energy, while simultaneously retaining the discrete energy stability and the discrete MBP; see Remark \ref{rem:eop} and Example \ref{Ex4_2}.

The remainder of the paper is organized as follows. 
In Section \ref{Sec:L1}, we first prove the MBP for the tFAC equation under either double-well or Flory--Huggins potential. Then, we propose and analyze the fully-discrete linear stabilized $L1$ type scheme, including the unique solvability, unconditional MBP-preservation and energy-stability. Numerical experiments are conducted in Section \ref{Sec:L1_Numer} to illustrate the performance of the proposed method. In Section \ref{Sec:L21}, we develop two linear stabilized $L2$-$1_{\sigma}$-ESAV schemes, incorporating two different stabilization terms: one balanced and one unbalanced, and unconditional energy-stability and conditional MBP-preservation are also discussed and tested. Some concluding remarks are finally drawn in the last section.

\section{Maximum-bound principle and stabilized $L1$ scheme}\label{Sec:L1}
For simplicity, the whole discussion below is confined to the two-dimensional (2D) square domain $ \Omega = (0,L)^2 $ with periodic boundary conditions. It is worth noting that the extensions to the three-dimensional (3D) case and/or homogeneous Neumann boundary conditions present no significant difficulties.
In this section, we first discuss the MBP for the tFAC equation \eqref{Model:tAC}. Assume that there exists a positive constant $ \beta $ such that $f(\beta) = f(-\beta) = 0$ and $ F(\phi) $ satisfies the monotone conditions away from $ (-\beta, \beta) $, i.e.,
\begin{equation}\label{Condi2:MBP}
	\begin{aligned}
		f(\phi) > 0, \  \forall \phi \in ( -\infty, -\beta ) \cap \text{Dom}( f ); \quad f(\phi) < 0, \  \forall \phi \in ( \beta, \infty ) \cap \text{Dom}( f ).
	\end{aligned}
\end{equation}
Two such typical potentials are widely used in the classical Allen--Cahn equation: one is the \textit{double-well} potential
\begin{equation}\label{poten:dw}
	F( \phi ) = \frac{1}{4} ( 1 - \phi^2 )^2, \quad f(\phi) = -F'( \phi ) = \phi - \phi^3 ,
\end{equation}
where $ \text{Dom}( f ) = \mathbb{R} $ and $ \beta = 1 $, and the other is the \textit{Flory--Huggins} potential
\begin{equation}\label{poten:fh}
	F( \phi ) = \frac{\theta}{2} [ ( 1 + \phi ) \ln ( 1 + \phi ) + ( 1 - \phi ) \ln ( 1 - \phi ) ] - \frac{ \theta_{c} }{2} \phi^2, \quad f(\phi)   = \frac{\theta}{2} \ln \frac{ 1 - \phi }{ 1 + \phi } + \theta_{c} \phi,
\end{equation}
with $ \theta_{c} > \theta > 0  $, where $ \text{Dom}( f ) = ( -1, 1 ) $ and $ \beta $ is the positive root of $ f( \rho ) = 0 $. For instance, if we set $\theta =0.8$ and $ \theta_{c} = 1.6 $, we have $ \beta \approx 0.9575 $.

The following lemma is originally presented in Lemma 2.3 of Ref. \cite{JSC_Du_2020}, we provide a detailed proof  in \ref{App:A} for completeness.
\begin{lemma}\label{lem:frac}
	Assume that $ v \in C[0,T] \cap C^{1} ( 0, T ]  $ and attains its minimum (maximum) at $ t_{*} \in ( 0, T ] $. Then there holds
	$$
	{}_{0}^{C}D^{\alpha}_{t} v( t_{*} ) \le ( \ge )~ 0.
	$$
\end{lemma}

Then, based on Lemma \ref{lem:frac}, we now proceed to establish the main MBP conclusion for the tFAC equation \eqref{Model:tAC}.
\begin{theorem}\label{thm:mbp_GtFAC}  
Suppose that $ \phi \in C( [0,T]; C^{2} ( \bar{ \Omega } ) ) \cap C^{1} ( ( 0, T ], C ( \bar{ \Omega } ) ) $ is a solution to the model problem \eqref{Model:tAC} with double-well or Flory--Huggins potential. If the  initial value satisfies $ \max_{\mathbf{x} \in \bar{\Omega}} |\phi_{\text {init }}(\mathbf{x})| \leq \beta $, there holds
\begin{equation}\label{mbp:GtFAC}
	\max_{\mathbf{x} \in \bar{\Omega}}|\phi(\mathbf{x},t)| \leq \beta, \quad \forall t>0.
\end{equation}
\end{theorem}
\begin{proof}
The conclusion for the double-well potential with $\beta=1$ can be found in \cite[Theorem 2.3]{JSC_Du_2020}. For completeness, we briefly revisit its proof here. Assume that the maximum of $ \phi $ is smaller than $-1$ and is achieved at $ ( \mathbf{x}_{*}, t_{*} ) \in \Omega \times (0,T] $. Then, by the regularity of $ \phi $ and Lemma \ref{lem:frac}, we have 
\begin{equation}\label{mbp1:GtFAC_dw}
	{}_0^C D^{\alpha}_t \phi( \mathbf{x}_{*}, t_{*} ) \leq  0.
\end{equation}
Moreover, owing to the periodic boundary conditions, the following estimate holds by Taylor expansion
\begin{equation}\label{mbp2:GtFAC_dw}
	\Delta \phi( \mathbf{x}_{*}, t_{*} ) \geq 0.
\end{equation}
Thus, we obtain
$$
0 \geq {}_0^C D^{\alpha}_t \phi ( \mathbf{x}_{*}, t_{*}  ) - \mm \varepsilon^{2} \Delta \phi ( \mathbf{x}_{*}, t_{*}  ) =  \mm f( \phi ( \mathbf{x}_{*}, t_{*} ) ) > 0,
$$
which is a contradiction. The upper bound of the solution can be proved analogously.
 
It remains to prove the MBP \eqref{mbp:GtFAC} for the Flory--Huggins potential \eqref{poten:fh}. To this aim, for any $ \epsilon \in ( 0, \frac{ 1 - \beta }{ 2 } ) $ such that $\beta+\epsilon <1$, we define 
$$
t_{*}( \epsilon ) = \max\Bigl\{  0 < t \leq T \big\vert \max_{\mathbf{x} \in \bar{\Omega}} |\phi(\mathbf{x},t)| = \beta + \epsilon \  \text{and} \  \max_{\mathbf{x} \in \bar{\Omega}}|\phi(\mathbf{x},s)| < \beta + \epsilon, \forall s \in [0,t) \Bigr\}. 
$$
It is clear that $ t_{*} ( \epsilon ) $ is well-defined by the continuity assumption on $ \phi $. Without loss of generality, assume that  $ \phi ( \mathbf{x}, t_{*} (\epsilon) ) $ achieves the maximum absolute value at $ \mathbf{x}_{*}(\epsilon) \in \bar{\Omega} $ and $ \phi ( \mathbf{x}_{*}(\epsilon), t_{*} (\epsilon) ) = -\beta - \epsilon $. 
Similar to \eqref{mbp1:GtFAC_dw}--\eqref{mbp2:GtFAC_dw}, we have 
\begin{equation*}\label{mbp1:GtFAC}
{}_0^C D^{\alpha}_t \phi( \mathbf{x}_{*}(\epsilon), t_{*} (\epsilon) ) \leq  0 \quad \text{and} \quad \Delta \phi( \mathbf{x}_{*}(\epsilon), t_{*} (\epsilon) ) \geq 0.
\end{equation*}
Therefore, we arrive at the result
\begin{equation*}\label{mbp4:GtFAC}
	0 \geq {}_0^C D^{\alpha}_t \phi ( \mathbf{x}_{*}(\epsilon), t_{*} (\epsilon) ) - \mm \varepsilon^{2} \Delta \phi ( \mathbf{x}_{*}(\epsilon), t_{*} (\epsilon) ) =  \mm f( \phi ( \mathbf{x}_{*}(\epsilon), t_{*} (\epsilon) ) ) > 0,
\end{equation*}
where we have used \eqref{Condi2:MBP} in the last inequality. This contradiction shows that no point $ ( \mathbf{x}_{*}(\epsilon), t_{*} (\epsilon) ) \in \bar{\Omega} \times (0,T] $ satisfies $ \phi ( \mathbf{x}_{*}(\epsilon), t_{*} (\epsilon) ) = -\beta - \epsilon $ for any given $ T $ and $ \epsilon $. A similar conclusion holds for the upper bound of $ \phi $. Thus, the MBP \eqref{mbp:GtFAC} is proved for the Flory--Huggins potential, as both $ T > 0 $ and $ \epsilon \in ( 0, \frac{ 1 - \beta }{ 2 } ) $ are arbitrary.
\end{proof}

Introduce an auxiliary variable $ R(t) = E_1[\phi] := \int_{\Omega} F(\phi) d\mathbf{x} $, and define the modified energy
\begin{equation}\label{def:energy_1}
\me[\phi, R] :=  \frac{\varepsilon^2}{2} \int_{\Omega}  \vert \nabla \phi(\mathbf{x}) \vert^2  d \mathbf{x}  + R,
\end{equation}
 which is equivalent to the original energy $E[\phi]$ defined in \eqref{AC:energy} at the continuous level.
Moreover, let $g(\phi, R):=\frac{\exp \left\{R\right\}}{ \exp \left\{E_{1}[\phi]\right\} }$, which is always equal to 1 at the continuous level. Then, model \eqref{Model:tAC} can be rewritten into the following equivalent system:
\begin{equation}\label{sav_1}
	\begin{aligned}
		{}_{0}^{C}D^{\alpha}_{t} \phi  = \mm ( \varepsilon^2 \Delta \phi + V( g(\phi, R) ) f(\phi) ) ,
\quad
		R_t  = - V( g(\phi, R) ) \left(f(\phi), \phi_t\right) ,
	\end{aligned}
\end{equation}
where $V(\cdot)$ is an auxiliary functional satisfying the following three assumptions:

\begin{itemize}
	\item[(\bf A1).] The auxiliary functional $ V(\cdot) \in C^{1} ( \mathbb{R} ) \cap W^{2,\infty} ( \mathbb{R} )$ such that $ V(1) = 1 $, $ V'(1) = 0 $, and 
	there exists a positive constant $ K_{1} $ such that $ | V' \left( \cdot \right) | \leq K_{1}  $;
	\item[(\bf A2).] The auxiliary functional is nonnegative and bounded by a positive constant $ K_{2} $, i.e., $ 0 \leq V \left( \cdot \right) \leq K_{2}  $;
	\item[(\bf A3).] The auxiliary functional satisfies the following monotonicity, i.e., for any $ z_{1}, z_{2} \in \mathbb{R}$ satisfy $ | z_{1} - 1 | \leq | z_{2} - 1 | $, we have $ | V( z_{1} ) - 1 | \leq | V( z_{2} ) - 1 | $.
\end{itemize}
By \cite[Corollary 2.1]{SISC_Tang_2019}, it is straightforward to verify that the new SAV system \eqref{sav_1} satisfies energy stablity with respect to $ \me[\phi, R]$, i.e.,
$	\me[\phi, R](t) \leq \me[\phi, R](0)$ for $ t > 0$.

\subsection{ Stabilized $L1$-ESAV scheme }\label{subsec:af}
Given a positive integer $ M $, let $ h = L/M $ be the spatial grid length and set $ \Omega_{h} := \{ \mathbf{x}_{h} = ( ih, jh ) \mid 0 \leq i,j \leq M \} $. Let $ \mathbb{V}_h $ be the set of all $M$-periodic real-valued grid functions on $ \Omega_{h} $, i.e.,
$$
\mathbb{V}_h:=\bigl\{v \mid v= \{v_{i,j}\}_{i,j=1}^{M}~ \text{and}~  v ~ \text{is periodic}\bigr\}.
$$ 
Notice that $ \mathbb{V}_h $ is a finite-dimensional linear space, thus any grid function in $ \mathbb{V}_h $ and any linear operator $ P: \mathbb{V}_h \rightarrow \mathbb{V}_h $ can be treated as a vector in $ \mathbb{R}^{M^2} $ and a matrix
in $ \mathbb{R}^{M^2 \times M^2 } $, respectively.

Define the following discrete inner product, and discrete $ L^2 $ and $ L^{\infty} $ norms
$$
\langle v, w\rangle=h^2 \sum_{i, j=1}^M v_{i j} w_{i j}, \quad \|v\|=\sqrt{\langle v, v\rangle}, \quad \|v\|_{\infty}=\max _{1 \leq i, j \leq M} |v_{i j}|
$$
for any $v, w \in \mathbb{V}_h$. 
Moreover, we define a maximum-norm function space with the positive constant $ \beta $ for the underlying MBP problem, that is
$$
\mathbb{V}_\beta:=\left\{v \mid v \in \mathbb{V}_{h} ~ \text{with} ~ \| v \|_{\infty} \leq \beta \right\} .
$$ 
Next, we consider a central finite difference discretization for the spatial differential operators. For any $ v \in \mathbb{V}_h $, the discrete Laplace operator $\Delta_h$ is defined by
\begin{equation}\label{Formula:Laplace}
\Delta_h v_{i j}=\frac{1}{h^2}\left(v_{i+1, j}+v_{i-1, j}+v_{i, j+1}+v_{i, j-1}-4 v_{i j}\right), \quad 1 \leq i, j \leq M,
\end{equation}
and the discrete gradient operator $\nabla_h$ is defined by
$$
\nabla_h v_{i j}=\left(\frac{v_{i+1, j}-v_{i j}}{h}, \frac{v_{i, j+1}-v_{i j}}{h}\right)^\top, \quad 1 \leq i, j \leq M .
$$

Consider (generally nonuniform) time levels $ 0 = t_{0} < t_{1} < \cdots < t_{N} = T $ with time steps  $ \tau_{k} := t_{k} - t_{k-1} $ for $ 1 \leq k \leq N $. Denote by 
$ r_{k} := \tau_{k}/\tau_{k-1} (k\geq 2)$ the adjacent time-step ratio. Let $v^k = v\left(t_k\right)$ and the temporal difference operator $\nabla_\tau v^k=v^k-v^{k-1}$ and $ \mathbb{D}_{\tau} v^{k} = \nabla_\tau v^k/\tau_{k}$ for $k \geq 1$. The nonuniform $L1$ formula of the Caputo derivative is given by
\begin{equation}\label{Formula:L1}
		\mathbb{D}_\tau^\alpha v^n 
		= \sum_{k=1}^n \int_{t_{k-1}}^{t_k} \frac{\omega_{1-\alpha} (t_n-s)}{\tau_k} \nabla_\tau v^k  ds 
		= \sum_{k=1}^n A_{n-k}^{(n)} \nabla_\tau v^k,
\end{equation}
where the time-dependent convolution kernels 
$
A_{n-k}^{(n)}:=\frac{1}{\tau_k}\int_{t_{k-1}}^{t_k} {\omega_{1-\alpha}(t_n-s)} ds 
$
satisfy the following two main properties (see \cite{SINUM_Liao_2018,SCM_Liao_2024}):\\

\textbf{(P1).} The discrete kernels are positive, i.e., $ A_{n-k}^{(n)} > 0 $ for $ 1 \leq k \leq n $;

\textbf{(P2).} The discrete kernels are monotone, i.e., $ A_{n-k+1}^{(n)} < A_{n-k}^{(n)} $ for $ 1 \leq k \leq n $.\\

Now, let $\{\phi^{n}, R^{n} \}$ be the fully discrete approximations of the exact solutions $ \{\phi( t_{n} ), R( t_{n} ) \}$ to the original continuous problem \eqref{sav_1}. We apply the nonuniform $L1$ formula \eqref{Formula:L1} and central difference approximation \eqref{Formula:Laplace} to derive the following stabilized $ (2 - \alpha) $th order ESAV scheme (denoted as $L1$-sESAV):

\paragraph{\indent \bf Step 1}
  Let $\hat{\phi}^{0} = \phi^{0}:=\phi_{\text {init }}$. For $ n = 1 $, first solve a predicted solution $ \hat{\phi}^{1} $ from the fully-implicit difference scheme:
\begin{equation}\label{sch:L1_non_1}
		\mathbb{D}_\tau^\alpha \hat{\phi}^{1}:=  A^{(1)}_{0} ( \hat{\phi}^{1} - \phi^{0} )  = \mm \bigl( \varepsilon^2 \Delta_h \hat{\phi}^{1} + f(\hat{\phi}^{1}) \bigr) ,
\end{equation}
and then, find $ \{ \phi^{1}, R^{1} \} \in \mathbb{V}_{h} \times \mathbb{R} $ by 
\begin{equation}\label{sch:L1_1_1}
	\begin{aligned}
		\mathbb{D}^{\alpha}_{\tau} \phi^{1}  = \mm \bigl( \varepsilon^2 \Delta_h \phi^{1} + V( g_{h}(\hat{\phi}^{1}, R^{0}) ) f(\hat{\phi}^{1}) - \kappa V( g_{h}(\hat{\phi}^{1}, R^{0}) ) ( \phi^{1} - \hat{\phi}^{1} ) \bigr),
	\end{aligned}
\end{equation}
\begin{equation}\label{sch:L1_1_2}
	\begin{aligned}
		\mathbb{D}_{\tau} R^{1} = V( g_{h}(\hat{\phi^{1}}, R^{0}) ) \, \big\langle - f(\hat{\phi}^{1}) + \kappa\, ( \phi^{1} - \hat{\phi}^{1} ), \mathbb{D}_{\tau} \phi^{1} \big\rangle.
	\end{aligned}
\end{equation}
where $ \kappa \geq 0 $ is a stabilizing constant, and $g_h$, a discrete version of $g$, denoted by $ g_h(v, w):= \frac{\exp \left\{w\right\}}{ \exp \left\{E_{1h}[v]\right\} } $ with $E_{1 h}$  the discrete counterpart of $E_1$, i.e., $E_{1 h}[v]:=\langle F(v), 1\rangle$.  

\paragraph{\indent \bf Step 2} For $ n \geq 2 $,  given $ \phi^{n-1}$ and $ \phi^{n-2} $, first predict a MBP-preserving solution $ \hat{\phi}^{n} $ by pre-processing the standard linear extrapolation, i.e.,
\begin{equation}\label{sch:L1_0}
		\hat{\phi}^{n} = \min\left\{  \max\left\{ ( 1 + r_{n} ) \phi^{n-1} - r_{n} \phi^{n-2}  , - \beta \right\} , \beta \right\},
\end{equation}
and then, find $ \{ \phi^{n}, R^{n} \} \in \mathbb{V}_{h} \times \mathbb{R} $ such that
\begin{equation}\label{sch:L1_n_1}
	\begin{aligned}
		\mathbb{D}^{\alpha}_{\tau} \phi^{n}  = \mm \bigl( \varepsilon^2 \Delta_h \phi^{n} + V( g_{h}(\hat{\phi}^{n}, R^{n-1}) ) f(\hat{\phi}^{n}) - \kappa V( g_{h}(\hat{\phi}^{n}, R^{n-1}) ) ( \phi^{n} - \hat{\phi}^{n} )  \bigr),
	\end{aligned}
\end{equation}
\begin{equation}\label{sch:L1_n_2}
	\begin{aligned}
		\mathbb{D}_{\tau} R^{n}  = V( g_{h}(\hat{\phi^{n}}, R^{n-1}) ) \big\langle - f(\hat{\phi}^{n}) + \kappa( \phi^{n} - \hat{\phi}^{n} ), \mathbb{D}_{\tau} \phi^{n} \big\rangle.
	\end{aligned}
\end{equation}

\begin{remark}
    In the peoposed $L1$-sESAV scheme \eqref{sch:L1_non_1}--\eqref{sch:L1_n_2}, the assumption (\textbf{A1}) on the functional $ V(\cdot) $ ensures that the first-order approximation of the auxiliary variable $ R $ does not affect the temporal convergence of the phase-field function $ \phi $. Actually, a direct application of Taylor expansion, together with the assumption (\textbf{A1}), gives us
\begin{equation*} 
		V( z )  = 1 + \int_{1}^{ z } (  z - s ) V^{\prime\prime}  ( s ) ds,
\end{equation*}
which implies that if $ z $ is a first-order approximation to $1$, then $ V( z ) $ will be a second-order approximation to $ 1 $. Moreover, the non-negativity and boundedness assumption (\textbf{A2}) allows us to develop effective stabilization term to dominate the nonlinear potential term, leading to MBP-preserving numerical methods, see Theorems \ref{thm:MBP_L1} and \ref{thm:MBP_L21} for details.
\end{remark} 

\begin{remark}\label{rem:iter}
To effectively solve the nonlinear scheme \eqref{sch:L1_non_1} in \textbf{Step 1} for $n=1$, we propose the following simple iteration scheme with stabilization:
\begin{equation}\label{sch:L1_non_2}
	\begin{aligned}
		A^{(1)}_{0} ( \hat{\phi}^{1}_{(\mathfrak{s} )} - \phi^{0} )  = \mm  \bigl( \varepsilon^2 \Delta_h \hat{\phi}^{1}_{(\mathfrak{s} )} + f(\hat{\phi}^{1}_{(\mathfrak{s} -1)}) - \kappa \, ( \hat{\phi}^{1}_{(\mathfrak{s} )} - \hat{\phi}^{1}_{(\mathfrak{s} -1)} ) \bigr) , \quad \mathfrak{s}  \geq 1,
	\end{aligned}
\end{equation}
where the subscript $ \mathfrak{s} $ denotes the iteration index, and the initial iteration value is taken as $ \hat{\phi}^{1}_{(0)} = \phi^{0} $. In the next section, we will show that the iteration scheme \eqref{sch:L1_non_2} unconditionally preserves the discrete MBP, and meanwhile, the solution $\hat{\phi}^{1}_{(\mathfrak{s} )}$ converges to the unique solution $\hat{\phi}^{1}$ of the nonlinear scheme \eqref{sch:L1_non_1}. The detailed implementation of the $L1$-sESAV scheme is summarized as below. 
\end{remark}
\begin{algorithm} [!ht]
 \caption{Implementation of the $L1$-sESAV scheme}
 \label{alg:L1}
 \begin{algorithmic} [1]

   \STATE Given initial value $ \phi^{0}=\phi_{\text {init }} $ and $ R^{0} =R(t_0)$.
  
  \IF{$n=1$}
  \STATE \textbf{Step 1:} Compute $\hat{\phi}^{1}:=\hat{\phi}^{1}_{(\mathfrak{s} )}$ by the iteration scheme \eqref{sch:L1_non_2} under a given termination tolerance $ tol $ and initial guess  $ \hat{\phi}^{1}_{(0)} = \phi^{0} $ until $\|\hat{\phi}^{1}_{(\mathfrak{s} )}-\hat{\phi}^{1}_{(\mathfrak{s-1} )}\|_\infty \le tol$.
  \STATE \textbf{Step 2:} Compute $\phi^{1}$ by \eqref{sch:L1_1_1} via $\hat{\phi}^{1}$ and $ R^{0} $.
  \STATE \textbf{Step 3:} Compute $R^{1}$ by \eqref{sch:L1_1_2} via $\phi^{1}$, $\hat{\phi}^{1}$ and $ R^{0} $.
  \ENDIF

   \FOR{$n=2$ to $N$}
  \STATE \textbf{Step 4:} Compute $\hat{\phi}^{n}$ by \eqref{sch:L1_0} via $\phi^{n-1}$ and $\phi^{n-2}$.
  \STATE \textbf{Step 5:} Compute $\phi^{n}$ by \eqref{sch:L1_n_1} via $\hat{\phi}^{n}$ and $ R^{n-1} $.
  \STATE \textbf{Step 6:} Compute $R^{n}$ by \eqref{sch:L1_n_2} via $\phi^{n}$, $\hat{\phi}^{n}$ and $ R^{n-1} $.
  \ENDFOR 
 \end{algorithmic}
\end{algorithm}

In the end of this subsection, we discuss the linear extrapolation error $ \mathcal{R}^{n} [v] := v( t_{n} ) - \hat{v}( t_{n} ) $  for $ n \geq 2 $ under the nonuniform mesh, where $ \hat{v}( t_{n} ) := ( 1 + r_{n} ) v( t_{n-1} ) - r_{n} v( t_{n-2} ) $ denotes the linear extrapolation. It is well known that if $v$ is sufficiently smooth, the extrapolation can achieve second-order temporal accuracy. However, for the tFAC model that exhibits an initial weak singularity, the temporal accuracy may deteriorate. The following lemma provides an estimate of $ \mathcal{R}^{n} [v]$, see also \ref{App:B} for detailed proof.
\begin{lemma}\label{lem:Trun_L1} 
	Assume that $ v \in C^{2}( 0, T ] $ and $ \vert  v^{\prime\prime}(t) \vert \leq C ( 1 + t^{\iota-2} )  $ for any regularity parameter $ 0 < \iota < 1 $. 
	Then, there exists a constant $ C_{v} $ that depends only on $ v $ such that
	\begin{equation}\label{TrunL1:00}
		\begin{array}{l}
			\vert \mathcal{R}^{n} [v] \vert \leq
			\left\{
			\begin{aligned}
				& C_{v} \bigl( ( \tau_{1} + \tau_{2} )^{\iota}/\iota + t_{1}^{\iota-2} \tau_{2}^{2} \bigr), & n=2 ,\\
				& C_{v} \bigl( t_{n-2}^{\iota-2} ( \tau_{n-1} + \tau_{n} )^{2} + t_{n-1}^{\iota-2} \tau_{n}^{2} \bigr),  &  3 \leq n \leq N.
			\end{aligned}
			\right.
		\end{array}
	\end{equation}
	Moreover, if the graded temporal grids $ t_{k} = T(k/N)^{\gamma} $ with grading parameter $ \gamma \geq 1 $ are employed, it follows that
    \begin{equation}\label{TrunL1:02}
		\vert \mathcal{R}^{n} [v] \vert \leq C_{v,\gamma} N^{-\min\{2,\gamma\iota\}}, \quad 2 \leq n \leq N,
    \end{equation}
where $ C_{v,\gamma} $ is a constant depending only on $ v $ and $ \gamma $.
\end{lemma}

\begin{remark} It is well known that under the initial weak singularity assumption
\begin{equation}\label{Assu:weak}
v \in C^{3}( 0, T ]  \quad \text{and} \quad \vert \partial_t^{(\ell)} v(t)  \vert \leq C ( 1 + t^{\iota-\ell} ), \  \ell\le 3,\ \iota \in (0,1),
\end{equation}
the global consistency error, which is commonly used in error analysis, attained by the $L1$ and $L2$-$1_{\sigma}$ formulas achieves the order of $\min{ \{ 2 - \alpha, \gamma \iota \} }$ and $\min{ \{ 2, \gamma \iota \} }$ on graded temporal grids, respectively; see \cite{SINUM_Liao_2018,CiCP_Liao_2021} for more details. Together with the interpolation error $ \mathcal{R}^{n}[v] $  established in Lemma \ref{lem:Trun_L1}, this indicates that, to compensate for the accuracy reduction caused by the weak singularity of the solution to the tFAC model \eqref{Model:tAC}, it is sufficient to select the mesh grading parameter $ \gamma = (2-\alpha)/\iota$ for the $L1$ scheme and $ \gamma = 2/\iota$ for the $L2$-$1_{\sigma}$ scheme introduced in Section \ref{Sec:L21}. The numerical results presented in Examples \ref{subsec:L1_TimeCon} and \ref{subsec:Ls_TimeCon} have confirmed the validity of this statement.
\end{remark}

\subsection{MBP-preservation and unique solvability of the prediction strategy }\label{subsec:iter}
In this subsection, we shall show that the predicted solution $ \hat{\phi}^{n} $ yielded by \textbf{Step 1} or \textbf{Step 2} is MBP-preserving and uniquely solvable. From \eqref{sch:L1_0}, it can be seen that the conclusion is clearly valid for $ n \geq 2 $. It remains to consider the case $ n = 1 $. The following two lemmas will be used later.

\begin{lemma}[\cite{JCM_Tang_2016}]\label{lem:MBP_left}  
	For any $ \lambda > 0 $, we have $ \| ( \lambda I - \Delta_{h} )^{-1} \|_{\infty} \leq \lambda ^{-1} $, where $ I $ represents the identity operator.
\end{lemma}

\begin{lemma}[\cite{SIREV_Du_2021}]\label{lem:MBP_right} 
	If $\kappa \ge \|f^{\prime} \|_{C[-\beta, \beta]}$ holds for some positive constant $\kappa$, then we have $| f( \xi ) + \kappa \xi | \leq \kappa \beta $ for any $\xi \in[-\beta, \beta]$.
\end{lemma}

Next, we use the well-known Banach’s fixed point theorem to prove that the nonlinear scheme \eqref{sch:L1_non_1} is uniquely solvable and preserves the discrete MBP.

\begin{theorem}\label{thm:cover_L1_iter} 
	If $\kappa \ge \|f^{\prime} \|_{C[-\beta, \beta]}$, the simple iterative scheme \eqref{sch:L1_non_2} unconditionally preserves the MBP for $\{ \hat{\phi}^1_{(\mathfrak{s} )} \}$, that is,
\begin{equation}\label{MBP:L1_inner}
		\text{if}~~\| \phi^{0} \|_{\infty} \leq \beta \quad  \Longrightarrow \quad 
        \|\hat{\phi}^1_{(\mathfrak{s} )}\|_{\infty} \leq \beta \quad \text{for} \  \mathfrak{s}  \geq 1.
\end{equation}
Furthermore, if the first level time-step $ \tau_{1} $ is sufficiently small such that
\begin{equation}\label{Cover:L1_tau}
		\tau_{1} \le 1/\sqrt[\alpha]{ \kappa \mm \Gamma( 2 - \alpha ) },
\end{equation}
	then (i) the iterative scheme \eqref{sch:L1_non_2} converges to the unique solution $ \hat{\phi}^{1} $ of the nonlinear scheme \eqref{sch:L1_non_1} in the maximum norm such that $\|\hat{\phi}^1\|_{\infty} \leq \beta$, and (ii) the $L1$-sESAV scheme \eqref{sch:L1_non_1}--\eqref{sch:L1_n_2} is uniquely solvable.
\end{theorem}
\begin{proof}
For any $ \mathfrak{s}  \geq 1 $, we assume that
$ 
	\| \hat{\phi}^{1}_{ (\mathfrak{s}-1) } \|_{\infty} \leq \beta.
$ 
Rewrite \eqref{sch:L1_non_2} equivalently as
\begin{equation*}
		\bigl( ( A^{(1)}_{0} + \kappa \mm  ) I - \mm \varepsilon^{2} \Delta_h \bigr) \hat{\phi}^{1}_{(\mathfrak{s} )}  
         = A^{(1)}_{0} \phi^{0} + \mm  \bigl(  f(\hat{\phi}^{1}_{(\mathfrak{s} -1)}) +\kappa \hat{\phi}^{1}_{(\mathfrak{s} -1)} \bigr),
\end{equation*}
which together with Lemmas \ref{lem:MBP_left}--\ref{lem:MBP_right} and $ \| \phi^{0} \|_{\infty} \leq \beta $ yields
\begin{equation*}
	( A^{(1)}_{0} + \kappa \mm ) \| \hat{\phi}^{1}_{(\mathfrak{s} )} \|_{\infty} \le ( A^{(1)}_{0} + \kappa \mm ) \beta.
\end{equation*}
Thus, the conclusion \eqref{MBP:L1_inner} is proved.

Next, we discuss the MBP-preservation and unique solvability of the nonlinear scheme \eqref{sch:L1_non_1}. To this aim, for any grid function $ v \in \mathbb{V}_{\beta} $, we define the mapping $ \mt : = \mathbb{V}_{\beta} \rightarrow \mathbb{V}_{\beta} $ through \eqref{sch:L1_non_2} by $ \mt[v] = w \in \mathbb{V}_{\beta} $, i.e.,
\begin{equation}\label{Cover:L1_1}
	\begin{aligned}
		A^{(1)}_{0} ( w - \phi^{0} )  = \mm \bigl(  \varepsilon^2 \Delta_h w + f(v) - \kappa ( w - v ) \bigr).
	\end{aligned}
\end{equation}
We shall show that $\mt$ is a contraction mapping under the time-step restriction \eqref{Cover:L1_tau} in what follows. For any $ v_{1}, v_{2} \in \mathbb{V}_{\beta} $, let $ w_{1} = \mt[ v_{1} ], w_{2} = \mt[ v_{2} ] $, and define $ \delta w = w_{1} - w_{2} $. Then it follows from \eqref{Cover:L1_1} that
\begin{equation*}
	\big( ( A^{(1)}_{0} + \kappa \mm  ) I - \mm \varepsilon^{2} \Delta_h \big) \delta w  = \kappa \mm ( v_{1} - v_{2} ) + \mm ( f(v_{1}) - f(v_{2}) ).
\end{equation*}
Further, Lemma \ref{lem:MBP_left} and the condition $ \kappa \ge \| f' \|_{ C[-\beta,\beta] } $ gives us
\begin{equation*}
	\begin{aligned}
		( A^{(1)}_{0} + \kappa \mm  ) \| \delta w \|_{\infty} \leq 2 \kappa \mm \| v_{1} - v_{2} \|_{\infty}.
	\end{aligned}
\end{equation*}
As a consequence, it holds that
\begin{equation}\label{Contra_fact}
	\begin{aligned}
		\| \mt[ v_{1} ] - \mt[ v_{2} ] \|_{\infty} = \| \delta w \|_{\infty} \leq \ell_{\mt} \| v_{1} - v_{2} \|_{\infty} \quad  \text{with} \  \ell_{\mt} = \frac{ 2 \kappa \mm }{ A^{(1)}_{0} + \kappa \mm }.
	\end{aligned}
\end{equation}
Since $A_{0}^{(1)} = \frac{1}{\tau_1} \int_{t_{0}}^{t_1}  \omega_{1-\alpha} (t_1-s) ds = \frac{ \tau_{1}^{-\alpha} }{ \Gamma( 2 - \alpha ) }$, we conclude that if $ \tau_{1} \leq \sqrt[\alpha]{1/(\kappa \mm \Gamma( 2 - \alpha ) )} $, it follows that $ \ell_{\mt} < 1 $, ensuring that $ \mt $ is a contractive mapping. By Banach’s fixed point theorem, the mapping $ \mt $ admits a unique fixed point $ \hat{\phi}^{1} \in \mathbb{V}_{\beta} $. Consequently, the iterative scheme \eqref{sch:L1_non_2} is
convergent in the maximum norm, and the fixed point $\hat{\phi}^{1}$ is just the unique solution of the nonlinear scheme \eqref{sch:L1_non_1}.

For the simplicity of presentation, below we denote $ V^n := V( g_{h}(\hat{\phi}^{n}, R^{n-1}) ) $ for $ n \geq 1 $. Then, equations \eqref{sch:L1_1_1}--\eqref{sch:L1_1_2} and \eqref{sch:L1_n_1}--\eqref{sch:L1_n_2} can be equivalently and uniformly rewritten as:
\begin{equation}\label{sch:L1_3}
	\big[ ( A^{(n)}_{0} + \kappa \mm V^{n} ) I - \mm \varepsilon^2 \Delta_h \big] \phi^{n}  = A^{(n)}_{0} \phi^{n-1} - \sum^{n-1}_{k=1} A^{(n)}_{n-k} \nabla_{\tau} \phi^{k} + \mm V^{n} \bigl(  f(\hat{\phi}^{n})+  \kappa \hat{\phi}^{n}\bigr),
\end{equation}
\begin{equation}\label{sch:L1_4}
		R^{n}  = R^{n-1} + V^{n} \big\langle - f(\hat{\phi}^{n}) + \kappa( \phi^{n} - \hat{\phi}^{n} ), \nabla_{\tau}\phi^{n}  \big\rangle,
\end{equation}
for $n \ge 1$. 
It follows from $ A^{(n)}_{0} > 0 $ and assumption (\textbf{A2}) that 
the coefficient matrix of \eqref{sch:L1_3} is symmetric and positive definite, and thus \eqref{sch:L1_3}--\eqref{sch:L1_4} is uniquely solvable, which further implies the uniquely solvablity of the $L1$-sESAV scheme \eqref{sch:L1_non_1}--\eqref{sch:L1_n_2}. The proof is completed.
\end{proof}

\begin{remark} 
Note that a useful estimate \eqref{Contra_fact} for the contraction factor $ \ell_{ \mt } $ is provided in the proof of Theorem \ref{thm:cover_L1_iter}, indicating that the iteration error decreases by at least a factor of $ \ell_{ \mt } $ per iteration. Let the first step iteration error be denoted by $ e_{(1)}:=\| \hat{\phi}^{1}_{(1)} - \phi^{0} \|_{\infty} \leq 2 \beta $. Then, for a given termination tolerance $ tol $, a maximum number of iterations
$
\mathfrak{s}^{*} = \big\lceil \frac{ \ln ( tol/ e_{(1)} ) }{ \ln(\ell_{ \mt }) } \big\rceil + 1 \le \big\lceil \frac{ \ln ( tol/ (2\beta) ) }{ \ln(\ell_{ \mt }) } \big\rceil + 1
$
is sufficient to ensure $ \| \hat{\phi}^{1}_{(\mathfrak{s}^{*}+1)} - \hat{\phi}^{1}_{(\mathfrak{s}^{*})} \|_{\infty} \le tol $. Moreover, a smaller stabilization constant $\kappa$ leads to faster convergence of the simple iterative scheme \eqref{sch:L1_non_2}.
\end{remark}

\subsection{ Discrete energy stability and MBP of the $L1$-sESAV scheme}
Define the discrete version energy that corresponding to \eqref{def:energy_1}:
\begin{equation*}
	\begin{aligned}
		\me_h[ \phi^n, R^n ] := \frac{\varepsilon^2}{2} \|\nabla_h \phi^n \|^2 + R^n, \quad n \geq 0.
	\end{aligned}
\end{equation*}
Next, we will establish the energy stability of the $L1$-sESAV scheme.

\begin{lemma}[\cite{SCM_Liao_2024}]\label{lem:DC_positive} 
	For fixed $n \geqslant 1$, the convolution kernels $\{A_{n-k}^{(n)}\}$ are positive semi-definite in the sense that
	$$
	\sum_{k=1}^n w_k \sum_{j=1}^k A_{k-j}^{(k)} w_j \geqslant 0 \quad \text { for any sequence }\left\{w_k\right\}_{k=1}^n .
	$$
\end{lemma}

\begin{theorem}\label{thm:energy_L1}
    Assume that the conditions in Theorem \ref{thm:cover_L1_iter} hold, the $L1$-sESAV scheme \eqref{sch:L1_non_1}--\eqref{sch:L1_n_2} is unconditionally energy-stable in the sense that $\me_h[\phi^{n}, R^{n}] \leq \me_h[\phi^0, R^0]$ for $ n \geq 1 $.
\end{theorem}
\begin{proof}
Let $ n = k $ in \eqref{sch:L1_n_1}, and then taking the inner product with $\nabla_{\tau}\phi^{k}=\phi^{k}-\phi^{k-1}$, we have
\begin{equation*} 
	\big\langle \mathbb{D}^{\alpha}_{\tau} \phi^{k}, \nabla_{\tau}\phi^{k} \rangle  = \mm \varepsilon^2 \big\langle \Delta_h \phi^{k}, \phi^{k}-\phi^{k-1} \big\rangle  - \mm V^{k} \big\langle - f(\hat{\phi}^{k}) + \kappa ( \phi^{k} - \hat{\phi}^{k} ), \nabla_{\tau}\phi^{k} \big\rangle ,
\end{equation*}
which, together with \eqref{sch:L1_4}, leads to
\begin{equation}\label{energy1_f1}
	\me_h[\phi^{k}, R^{k}]-\me_h[\phi^{k-1}, R^{k-1}] \le -\frac{ \varepsilon^2}{2}\left\|\nabla_h (\phi^{k}- \phi^{k-1})\right\|^2 - \frac{1}{\mm} \big\langle \mathbb{D}^{\alpha}_{\tau} \phi^{k}, \nabla_{\tau}\phi^{k} \big\rangle, 
\end{equation}
for $k \geq 2.$
Similarly, it follows from \eqref{sch:L1_1_1} and \eqref{sch:L1_4} that \eqref{energy1_f1} also holds for $ k = 1 $.

Now, sum the above inequality from $ k = 1 $ to $ n $, we obtain from \eqref{Formula:L1} and Lemma \ref{lem:DC_positive} that
\begin{equation*} 
  \me_h[\phi^{n}, R^{n}]-\me_h[\phi^{0}, R^{0}] \le  - \frac{1}{\mm} \sum_{k=1}^{n}  \big\langle \mathbb{D}^{\alpha}_{\tau} \phi^{k}, \nabla_{\tau} \phi^{k}  \big\rangle = - \frac{1}{\mm}  \sum_{k=1}^{n} \Big\langle  \sum_{j=1}^k A_{k-j}^{(k)} \nabla_\tau \phi^j, \nabla_{\tau} \phi^{k}  \Big\rangle \leq 0,
\end{equation*}
 which completes the proof.
\end{proof}

In the following theorem, we are ready to establish the discrete MBP for the proposed $L1$-sESAV scheme.
\begin{theorem}\label{thm:MBP_L1}
	Assume that the conditions in Theorem \ref{thm:cover_L1_iter} hold, the $L1$-sESAV scheme \eqref{sch:L1_non_1}--\eqref{sch:L1_n_2} unconditionally preserves the MBP for $\left\{\phi^n\right\}$, that is,
   \begin{equation*}\label{MBP:L1}
	\text{if}~~\|\phi^{0} \|_{\infty} \leq \beta \quad  \Longrightarrow \quad  \|\phi^n \|_{\infty} \leq \beta, \quad \forall n \geq 1.
   \end{equation*}
\end{theorem}
\begin{proof}
This claim will be verified by the complete mathematical induction argument. 
Assume that $ \| \phi^{k} \|_{\infty} \leq \beta $ for $ 0 \leq k \leq n-1 $. From \eqref{sch:L1_0} and Theorem \ref{thm:cover_L1_iter}, we know that $ \| \hat{ \phi }^{n} \|_{\infty} \leq \beta $, which together with \eqref{sch:L1_3} and Lemmas \ref{lem:MBP_left}--\ref{lem:MBP_right} leads to
\begin{equation}\label{Formu:L1_MBP_1}
	( A^{(n)}_{0} + \kappa \mm V^{n} ) \| \phi^{n} \|_{\infty}  
    \le \bigl\| A_{0}^{(n)} \phi^{n-1} - \sum_{k=1}^{n-1} A_{n-k}^{(n)} \nabla_\tau \phi^k \bigr\|_{\infty}    + \kappa \mm V^{n} \beta,
\end{equation}
where the assumption (\textbf{A2}) has been used. Due to the positivity and monotonicity of the discrete kernels $\{ A^{(n)}_{n-k} \}$, it holds that
\begin{equation}\label{Formu:L1_MBP_2}
  \begin{aligned}
  \big \| A^{(n)}_{0} \phi^{n-1} - \sum^{n-1}_{k=1} A^{(n)}_{n-k} \nabla_{\tau} \phi^{k}  \big\|_{\infty} 
    & = \bigl\| \sum^{n-1}_{k=1} ( A^{(n)}_{n-k-1} - A^{(n)}_{n-k} ) \phi^{k} + A^{(n)}_{n-1} \phi^{0} \bigr\|_{\infty} \\
	& \leq \sum^{n-1}_{k=1} ( A^{(n)}_{n-k-1} - A^{(n)}_{n-k} ) \, \| \phi^{k} \|_{ \infty } + A^{(n)}_{n-1}\, \| \phi^{0} \|_{ \infty } 
     \leq A_{0}^{(n)} \beta.
  \end{aligned}
\end{equation} 
Therefore, inserting \eqref{Formu:L1_MBP_2} into \eqref{Formu:L1_MBP_1} yields the conclusion  for $ k =n $, and thereby the proof is completed.
\end{proof}

\begin{remark}\label{rem:eop}  
The linear $L1$-sESAV scheme developed in this section for the tFAC model \eqref{Model:tAC} is MBP-preserving. However, as noted in the introduction, the SAV approach only guarantees a modified form of energy stability. To further reduce the gap between the modified and original energies, motivated by the EOP strategy \cite{JSC_Liu_2024}, an improved strategy, referred to as the $L1$-sESAV-EOP scheme, is proposed as follows: Let $ \tilde{R}^{n}$ be the solution of \eqref{sch:L1_n_2} in Step 2, the auxiliary variable $R^{n}$ is then updated by
\begin{equation}\label{sch:L1_R_n_3}
		R^{n} = \min\Bigl\{ \me_h[ \phi^0, R^0 ] - \frac{\varepsilon^2}{2} \|\nabla_h \phi^n \|^2, E_{1 h}[ \phi^{n} ] \Bigr\},
\end{equation}
which is a solution to the following linear optimization problem
\begin{equation}\label{sch:L1_R_n_4}
\eta_{0} = \min_{\eta \in [0,1]} \eta, \quad \text{s.t.} \  \me_h[ \phi^n, R^n ] \leq \me_h[ \phi^0, R^0 ],~ R^{n} = \eta_{0} \tilde{R}^{n} + ( 1- \eta_{0} ) E_{1 h}[ \phi^{n} ].
\end{equation}
It is clear that the resulting $L1$-sESAV-EOP scheme remains energy-stable with respect to a modified energy, which, notably, closely approximates the original energy as $\eta_{0} \in[0,1]$. In particular, when $R^{n} =  E_{1 h}[ \phi^{n} ]$ in \eqref{sch:L1_R_n_3}, we have $\eta_{0}=0$ in  \eqref{sch:L1_R_n_4}. In this case, the modified energy coincides exactly with the original energy. Moreover, by the same argument as in Theorem \ref{thm:MBP_L1}, the improved scheme also preserves the discrete MBP.
\end{remark}

\section{ Numerical experiments }\label{Sec:L1_Numer}
In this section, ample numerical tests are provided to illustrate the accuracy and structure-preserving properties of the $L$1-sESAV scheme for the tFAC model \eqref{Model:tAC}. Note that the auxiliary functional introduced in Section \ref{Sec:L1} plays a crucial role in the construction and analysis of the MBP-preserving $L$1-sESAV scheme. In the numerical implementation, we choose the following piecewise polynomial function that satisfies the
assumptions (\textbf{A1})--(\textbf{A3}) with $ K_{1} = 49/24 $ and $ K_{2} = 1 $:
\begin{equation}\label{f_cut}
	\begin{array}{l}
		V(z) :=
		\left\{
		\begin{aligned}
			& 0, & z \in (-\infty, 0] ,\\
			& - 8 z^3 + 7 z^2,  &  z \in (0,1/2),          \\
			& 2z - z^2,  &    z \in [1/2, 3/2],\\				
			& 8 z^3 - 41 z^2 + 68 z - 36,  &  z \in (3/2,2).       \\
			& 0, &  z \in [2,\infty).
		\end{aligned}
		\right.
	\end{array}
\end{equation}
In all the following tests, we always set the stabilization constant $ \kappa = \| f' \|_{C[-1,1]} = 2 $ for the double-well potential \eqref{poten:dw}, and $ \kappa = \| f' \|_{C[-\beta,\beta]} \approx 8.02 $ for the Flory--Huggins potential \eqref{poten:fh}, for which the parameters are chosen as $\theta =0.8$ and $ \theta_{c} = 1.6 $.

Moreover, to address the initial weak singularity, we adopt two types of mixed nonuniform temporal meshes composed of a graded mesh in the initial subinterval. Specifically, for given $ \hat{T} $ and $ \hat{N} $, we split the time interval $ [0,T] $ into two parts $ [ 0, \hat{T} ] $ and $ [\hat{T}, T] $ with total $N$ subintervals, and employ the graded mesh to subdivide $ [ 0, \hat{T} ] $ as follows
\begin{equation}\label{mesh:graded}
t_{k} = \hat{T} ( k/\hat{N} )^{\gamma} \  \  \text{and} \  \  \tau_{k} = t_{k} - t_{k-1}, \quad 1 \leq k \leq \hat{N},
\end{equation}
where $ \gamma \geq 1 $ is the mesh grading parameter. For the remainder interval $ [\hat{T}, T] $, the following two different types of temporal girds are considered
\begin{itemize}
	\item uniform mesh:
	\begin{equation}\label{mesh:uni}
	\tau_{k} = \frac{ T - \hat{T} }{ N - \hat{N} }, \quad \hat{N}+1 \leq k \leq N;
    \end{equation}
	
	\item adaptive mesh  based on the energy variation \cite{SISC_Liao_2021,SISC_2011_Qiao}:
	\begin{equation}\label{Alg1:adaptive}
			\tau_{k} =  \max \Bigl\{ \tau_{\min} , \f{\tau_{\max}}{ \sqrt{ 1 + \eta \vert \p_{\tau} E^{k-1} \vert^{2} } } \Bigr\},
	\end{equation}
	where $ \tau_{\max}, \tau_{\min} $ are predetermined maximum and minimum time steps, and $ \eta $ is a tunable parameter.
\end{itemize}

\begin{remark}
The energy-based adaptive time-stepping strategy \eqref{Alg1:adaptive} employed in this work adjusts  time steps based on variations in the discrete energy \cite{SISC_Liao_2021,SISC_2011_Qiao}, which significantly enhances the computational efficiency for long-term simulations as shown in Table \ref{tab1_Ex3}. In practice, the error-based adaptive time-stepping strategy \cite{JCP_Wodo_2011,JCP_Gomez_2011,JCP_Liao_2020} also offers an effective alternative for the efficient simulation of phase-field models. Note that both strategies inherently generate nonuniform temporal meshes. Since this paper focuses on the development and analysis of MBP-preserving numerical schemes for general nonuniform time grids, our theoretical results provide a unified foundation to ensure their reliable application. 

Furthermore, owing to the well-structured profile of the phase variable---taking distinct (close-to-constant) values in the bulk phases and exhibiting large gradients only within the thin diffuse interface \cite{JSC_Feng_2005}---the development and analysis of numerical schemes on nonuniform spatial meshes are also highly meaningful, as they furnish theoretical support for space adaptivity \cite{JCP_Provatas_1999,JCP_Mackenzie_2002,JCP_Feng_2006}. In addition, combining the proposed numerical method with $h$-refinement and local time-stepping (LTS) techniques---where each grid point is allowed to adopt its own locally time steps \cite{SISC_Fernando_2022,SISC_Mehlin_2015}---may further enhance computational efficiency and accuracy. Nevertheless, preserving key physical properties, such as the discrete energy stability and the discrete MBP, within nonuniform spatial meshes and LTS frameworks remains highly challenging and requires further investigation in future.
\end{remark}

In the following numerical tests, the error is measured in the spatio-temporal maximum-norm that $ e(N) := \max_{1\leq n \leq N} \| \phi(t_{n}) - \phi^{n} \|_{\infty} $.
\subsection{Temporal convergence}\label{subsec:L1_TimeCon}
Consider the exterior-forced tFAC model
\begin{equation}\label{Model:tAC_source}
		{}_0^C D^{\alpha}_t \phi =  \mm ( \varepsilon^2 \Delta \phi + f( \phi ) ) + g( \mathbf{x}, t ), \quad (\mathbf{x}, t)\in (0, 2\pi)^2 \times ( 0, 0.5 ], 
\end{equation}
with parameters $ \mm = 0.01 $ and $ \varepsilon = 1 $. Here, the linear part $ g( \mathbf{x}, t ) $ is chosen such that the exact solution is given by $ \phi( \mathbf{x}, t ) = \omega_{1+\iota} (t) \sin(x) \sin(y) $ with parameter $ \iota \in (0,1) $. It is clear that the solution exhibits a weak singularity as indicated in \eqref{Assu:weak}.

For a given mesh grading parameter $ \gamma \geq 1 $, we set $ \hat{T} = \min\{ 1/\gamma, T \} $ and $ \hat{N} = \lceil \frac{ 1 }{ T + 1 + \gamma^{-1} }  \rceil $ in \eqref{mesh:graded}. 
The uniform temporal mesh \eqref{mesh:uni} is employed in the remainder interval $ [\hat{T}, T] $. 
The spatial domain is partitioned into $ M = 400 $ uniform elements along each spatial direction. We test the proposed $L$1-sESAV scheme for solving \eqref{Model:tAC_source} with different nonlinear potentials and a fixed regularity parameter $ \iota = 0.4 $. For fractional orders $ \alpha = 0.4 $ and $ 0.8 $, different mesh grading parameters $ \gamma = 2,3,4 $ are considered, respectively. The corresponding numerical results are presented in Tables \ref{Ex1:tab1}--\ref{Ex1:tab2}, from which we can observe that the expected temporal accuracy $ \mathcal{O}( N^{- \min\{ 2-\alpha, \gamma \iota \} } ) $ is achieved.
\begin{table}[!ht]
\vspace{-10pt}
	\caption{Time accuracy of the $L$1-sESAV scheme with $ \alpha = 0.4 $ and $ \iota = 0.4$ \label{Ex1:tab1}}%
	{\footnotesize
     \begin{tabular*}{\columnwidth}{@{\extracolsep\fill}cccccccc@{\extracolsep\fill}}
			\toprule
			\multicolumn{1}{c}{\multirow{2}{*}{potential}} & \multicolumn{1}{c}{\multirow{2}{*}{$ N $}} & \multicolumn{2}{c}{$\gamma=2$} & \multicolumn{2}{c}{$\gamma=3$} & \multicolumn{2}{c}{$\gamma=4$ $(=(2-\alpha)/\iota)$} \\
			\cmidrule{3-4} 
			\cmidrule{5-6}
			\cmidrule{7-8}
			&      & $ e(N) $ & Order & $ e(N) $ & Order & $ e(N) $ & Order \\
			\midrule
			 &  20     & $5.06 \times 10^{-2}$   &  ---    &  $1.68 \times 10^{-2}$ & ---    & $1.25 \times 10^{-2}$   &  ---    \\
			\multicolumn{1}{c}{\multirow{2}{*}{double-well}} &  40     & $2.91 \times 10^{-2}$   &  0.80   &  $7.64 \times 10^{-3}$ & 1.13   & $4.35 \times 10^{-3}$   &  1.52   \\
			  &  80     & $1.67 \times 10^{-2}$   &  0.80   &  $3.40 \times 10^{-3}$ & 1.17   & $1.50 \times 10^{-3}$   &  1.54   \\
			&  160    & $9.59 \times 10^{-3}$   &  0.80   &  $1.48 \times 10^{-3}$ & 1.20   & $5.07 \times 10^{-4}$   &  1.56   \\
			\midrule
			&  20     & $5.06 \times 10^{-2}$   &  ---    &  $1.61 \times 10^{-2}$ & ---       & $1.23 \times 10^{-2}$   &  ---    \\
			\multicolumn{1}{c}{\multirow{2}{*}{Flory--Huggins}} 
			&  40     & $2.91 \times 10^{-2}$   &  0.80   &  $7.50 \times 10^{-3}$ & 1.10      & $4.12 \times 10^{-3}$   &  1.58   \\			 
			&  80     & $1.67 \times 10^{-2}$   &  0.80   &  $3.37 \times 10^{-3}$ & 1.15      & $1.43 \times 10^{-3}$   &  1.52   \\
			&  160    & $9.59 \times 10^{-3}$   &  0.80   &  $1.48 \times 10^{-3}$ & 1.19      & $4.91 \times 10^{-4}$   &  1.54   \\
			\midrule
			\multicolumn{2}{c}{$ \min\{ 2-\alpha, \gamma \iota \} $ }   
			       &       & 0.80       &    & 1.20 &   & 1.60 \\
			\bottomrule
	\end{tabular*}}
\end{table}
\begin{table}[!ht]
\vspace{-12pt}
	\caption{Time accuracy of the $L$1-sESAV scheme with $ \alpha = 0.8 $ and $ \iota = 0.4 $ \label{Ex1:tab2}}%
	{\footnotesize\begin{tabular*}{\columnwidth}{@{\extracolsep\fill}cccccccc@{\extracolsep\fill}}
			\toprule
			\multicolumn{1}{c}{\multirow{2}{*}{potential}} & \multicolumn{1}{c}{\multirow{2}{*}{$ N $}} & \multicolumn{2}{c}{$\gamma=2$} & \multicolumn{2}{c}{$\gamma=3$  $(=(2-\alpha)/\iota)$} & \multicolumn{2}{c}{$\gamma=4$} \\
			\cmidrule{3-4} 
			\cmidrule{5-6}
			\cmidrule{7-8}
			&      & $ e(N) $ & Order & $ e(N) $ & Order & $ e(N) $ & Order \\
			\midrule
			&  20     & $1.38 \times 10^{-1}$   &  ---   &  $1.00 \times 10^{-1}$ & ---    & $1.05 \times 10^{-1}$   &  ---    \\
			\multicolumn{1}{c}{\multirow{2}{*}{double-well}} 
            &  40     & $7.91 \times 10^{-2}$   &  0.80  &  $4.59 \times 10^{-2}$ & 1.13   & $4.72 \times 10^{-2}$   &  1.15  \\
			  &  80     & $4.55 \times 10^{-2}$   &  0.80  &  $2.09 \times 10^{-2}$ & 1.14   & $2.09 \times 10^{-2}$   &  1.18   \\
			&  160    & $2.61 \times 10^{-2}$   &  0.80  &  $9.28 \times 10^{-3}$ & 1.17   & $9.15 \times 10^{-3}$   &  1.19 \\
			\midrule
			&  20     & $1.37 \times 10^{-1}$   &  ---    &  $1.00 \times 10^{-1}$ & ---       & $1.04 \times 10^{-1}$   &  ---    \\
			\multicolumn{1}{c}{\multirow{2}{*}{Flory--Huggins}} 
			&  40     & $7.91 \times 10^{-2}$   &  0.80   &  $4.58 \times 10^{-2}$ & 1.13      & $4.71 \times 10^{-2}$   &  1.15   \\			 
			&  80     & $4.55 \times 10^{-2}$   &  0.80   &  $2.08 \times 10^{-2}$ & 1.14      & $2.08 \times 10^{-2}$   &  1.18   \\
			&  160    & $2.61 \times 10^{-2}$   &  0.80   &  $9.27 \times 10^{-3}$ & 1.17      & $9.14 \times 10^{-3}$   &  1.19   \\
			\midrule
			\multicolumn{2}{c}{$ \min\{ 2-\alpha, \gamma \iota \} $ }   
			&       & 0.80       &    & 1.20 &   & 1.20 \\
			\bottomrule
	\end{tabular*}}
    \vspace{-5pt}
\end{table}

\subsection{Unconditional preservation of MBP and energy stability }\label{Ex4_2}
In this subsection, we numerically verify the discrete MBP and energy stability of the proposed $L$1-sESAV scheme by simulating the spinodal decomposition of a homogeneous mixture into two coexisting phases governed by the tFAC model \eqref{Model:tAC} with $ \alpha = 0.5 $, $\mm = 1$, $ \varepsilon = 0.01 $ in $ \Omega =  (0,1)^2 $. For the forthcoming simulation, we set $ M = 128 $, and the initial phase field is generated using uniformly distributed random data in the range $ [-0.8, 0.8] $, which is highly oscillating, see Figure \ref{figEx2_0}. 

\begin{figure}[!ht]
	\centering  
	\includegraphics[width=0.45\textwidth]{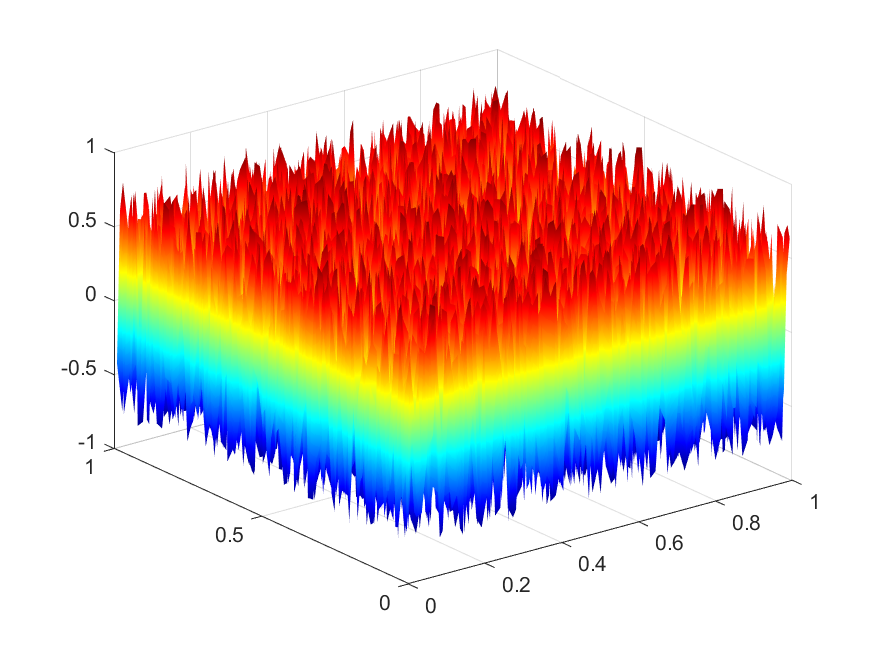}  
	\setlength{\abovecaptionskip}{0.0cm} 
	\setlength{\belowcaptionskip}{0.0cm}
	\caption{The initial highly-oscillated phase field}	\label{figEx2_0}
    \vspace{-10pt}
\end{figure}

Although a modified energy is introduced as an approximation of the original one to facilitate the proof of energy stability, we primarily focus on the behavior of the original (discrete) energy, which reflects the real physical mechanism of the dynamic process. In Figure \ref{figEx2_1}, we show the time evolutions of the maximum norms and original energies of the numerical solutions computed using the proposed $L$1-sESAV scheme with various time steps $ \tau = 2, 0.2, 0.02, 0.002 $ for the case of double-well potential. It shows that the $L$1-sESAV scheme perfectly preserves both the MBP and the energy stability, even for large time steps (e.g., $\tau = 2$ and $0.2$), which indicates the unconditional preservation of the discrete energy stability and MBP as proved in Theorems \ref{thm:energy_L1}--\ref{thm:MBP_L1}. Next, we also consider the simulation of Flory--Huggins potential case and corresponding numerical results are displayed in Figure \ref{figEx2_2}. Similar conclusions can also be observed. 
\begin{figure}[!htbp]
	\vspace{-12pt}
	\centering
	\subfigure[maximum norm of $\phi$]
	{
		\includegraphics[width=0.45\textwidth]{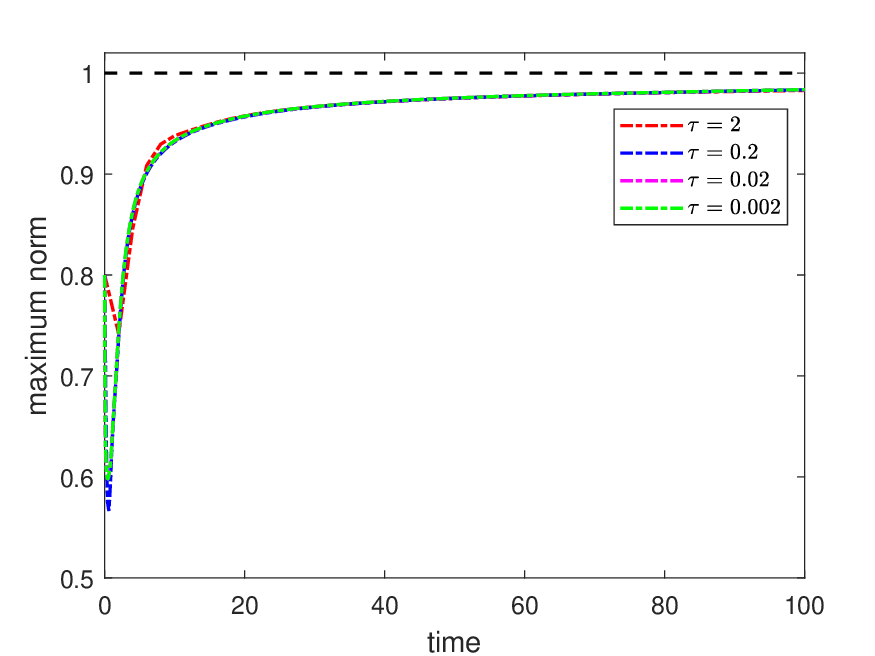}		\label{figEx2_1a}
	}%
	\subfigure[energy]
	{
		\includegraphics[width=0.45\textwidth]{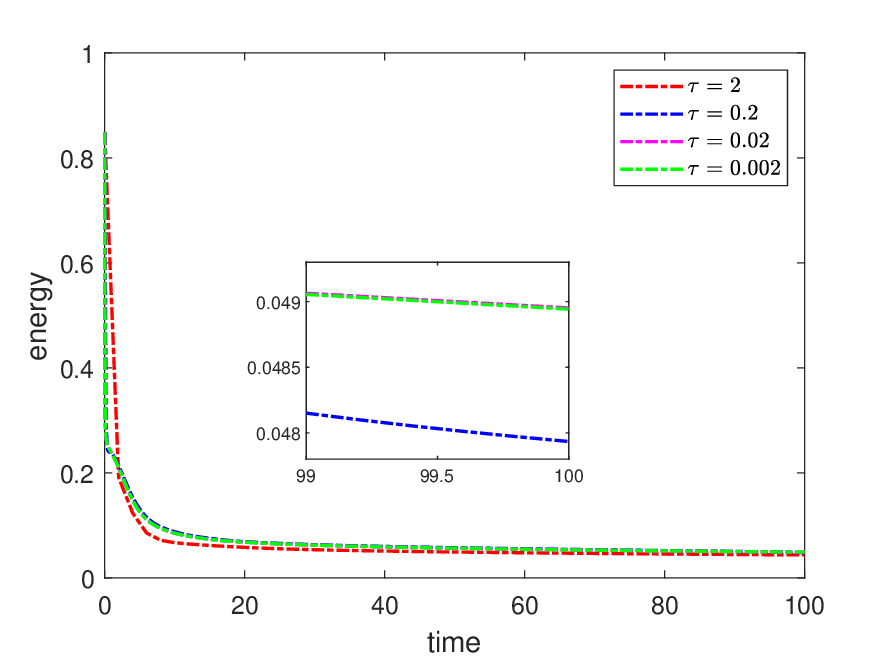}	\label{figEx2_1b}
	}%
	\setlength{\abovecaptionskip}{0.0cm} 
	\setlength{\belowcaptionskip}{0.0cm}
	\caption{Time evolutions of the maximum norm and energy of simulated solutions computed by the $L$1-sESAV scheme with different time steps: the double-well potential}	
	\label{figEx2_1}
\end{figure}
\begin{figure}[!htbp]
	\vspace{-12pt}
	\centering
	\subfigure[maximum norm of $\phi$]
	{
		\includegraphics[width=0.45\textwidth]{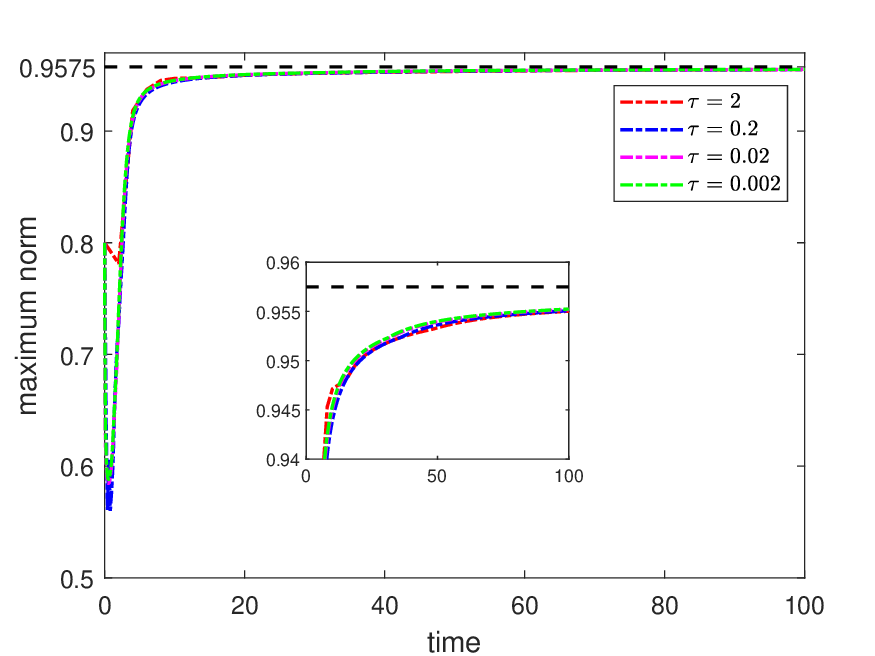}		\label{figEx2_2a}
	}%
	\subfigure[energy]
	{
		\includegraphics[width=0.45\textwidth]{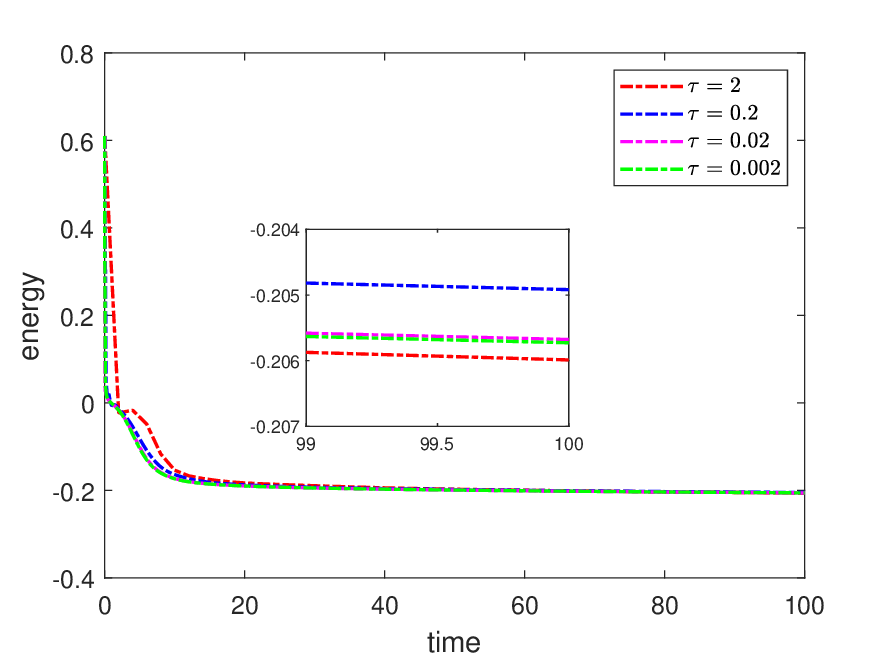}	\label{figEx2_2b}
	}%
	\setlength{\abovecaptionskip}{0.0cm} 
	\setlength{\belowcaptionskip}{0.0cm}
	\caption{Time evolutions of the maximum norm and energy of simulated solutions computed by the $L$1-sESAV scheme with different time steps: the Flory--Huggins potential}	
	\label{figEx2_2}
    \vspace{-10pt}
\end{figure}

Furthermore, we compare the modified and original energies yielded by the $L1$-sESAV scheme and $L1$-sESAV-EOP scheme, respectively, with $ \tau = 0.002 $ for the double-well potential. As seen from Figure \ref{figEx2_3}, a discrepancy exists between the modified energy and the original energy for the original $L1$-sESAV scheme, however, by applying the EOP technique introduced in Remark~\ref{rem:eop}, this discrepancy is substantially reduced for the $L1$-sESAV-EOP scheme. Moreover, Figure~\ref{figEx2_4} demonstrates that the $L1$-sESAV-EOP scheme continues to preserve both the discrete energy stability and MBP.
\begin{figure}[!htbp]
	\vspace{-12pt}
	\centering
	\subfigure[$L1$-sESAV scheme]
	{
		\includegraphics[width=0.45\textwidth]{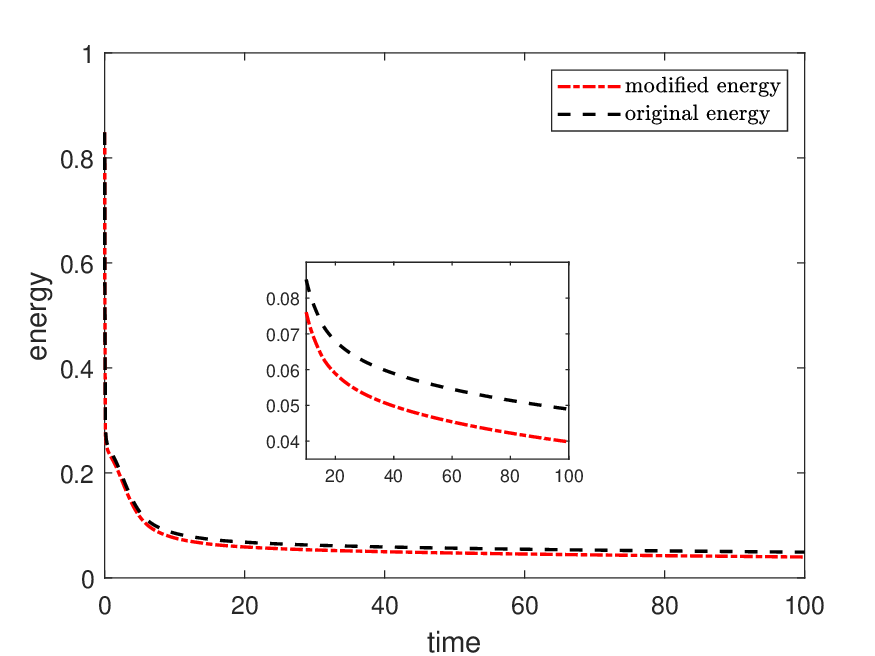}		\label{figEx2_3a}
	}%
	\subfigure[$L1$-sESAV-EOP scheme]
	{
		\includegraphics[width=0.45\textwidth]{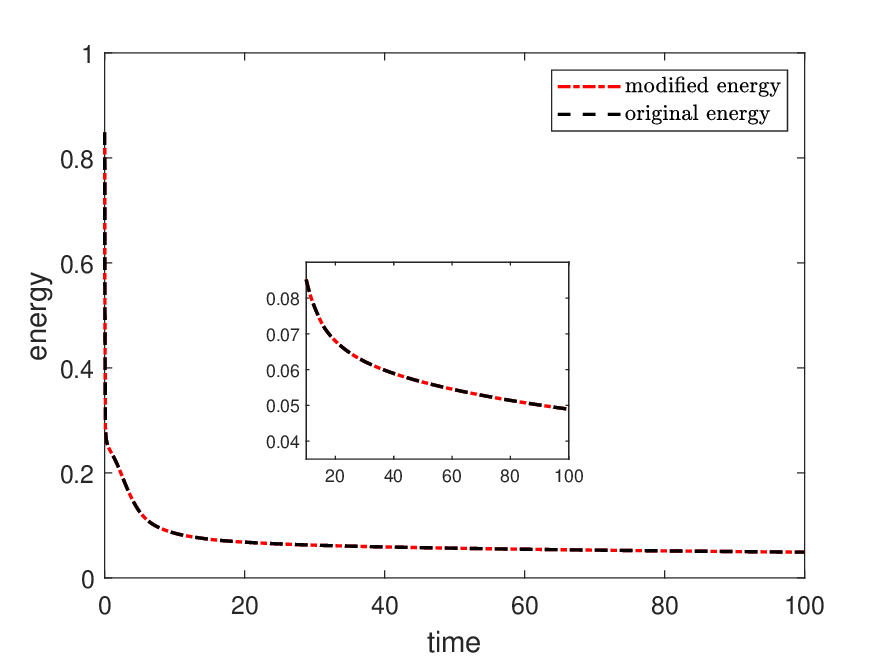}	\label{figEx2_3b}
	}%
	\setlength{\abovecaptionskip}{0.0cm} 
	\setlength{\belowcaptionskip}{0.0cm}
	\caption{Comparisons of the modified and original energies generated by the $L$1-sESAV and $L$1-sESAV-EOP schemes with $\tau = 0.002$: the double-well potential}	
	\label{figEx2_3}
        \vspace{-10pt}
\end{figure}
\begin{figure}[!htbp]
	\vspace{-12pt}
	\centering
	\subfigure[maximum norm of $\phi$]
	{
		\includegraphics[width=0.45\textwidth]{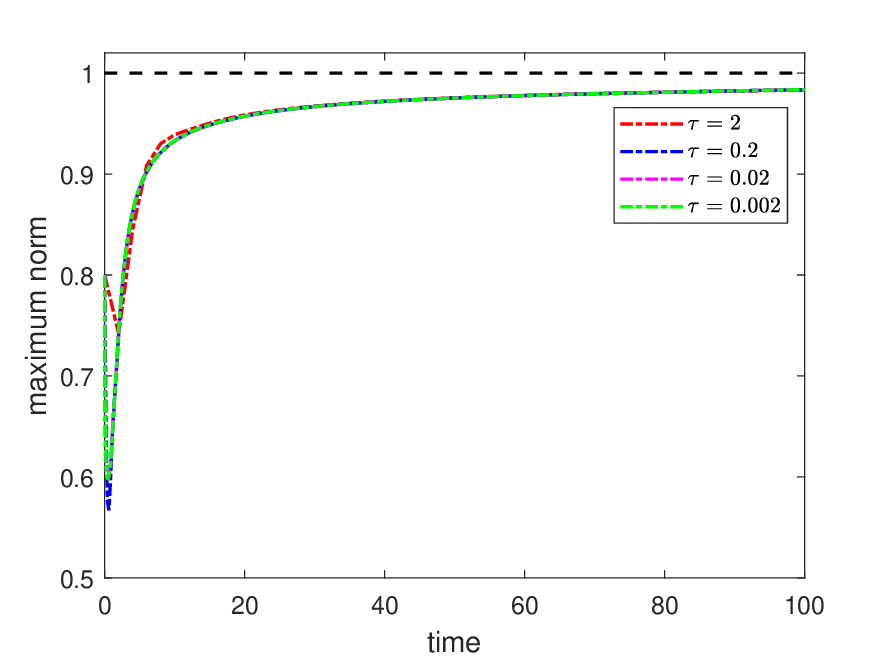}		\label{figEx2_4a}
	}%
	\subfigure[energy]
	{
		\includegraphics[width=0.45\textwidth]{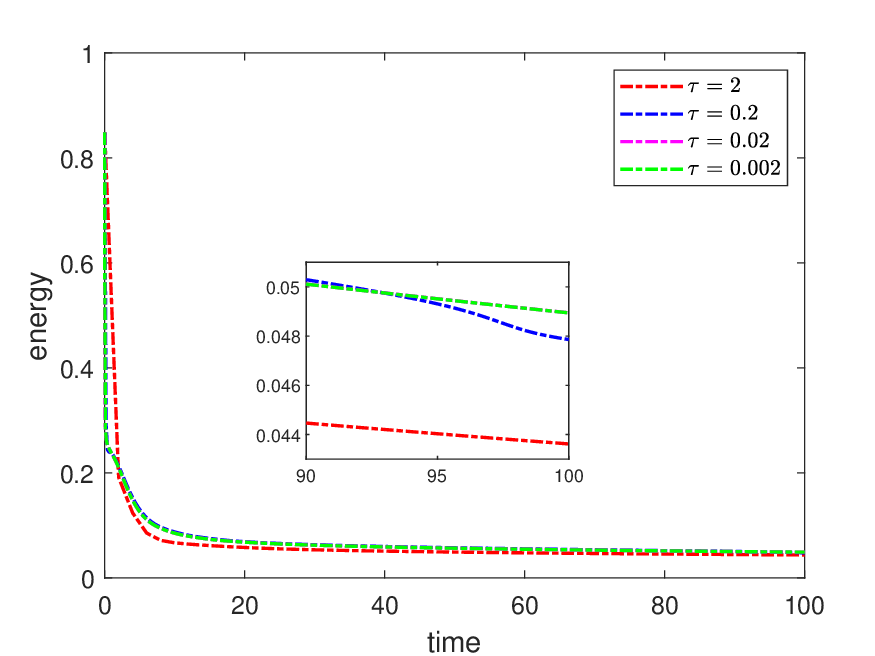}	\label{figEx2_4b}
	}%
	\setlength{\abovecaptionskip}{0.0cm} 
	\setlength{\belowcaptionskip}{0.0cm}
	\caption{Time evolutions of the maximum norm and energy of simulated solutions computed by the $L$1-sESAV-EOP scheme with different time steps: the double-well potential}	\label{figEx2_4}
    \vspace{-10pt}
\end{figure}

\subsection{Application of adaptive time-stepping strategy}
It is well known that the dynamics process governed by the tFAC model involves multiple time scales, and typically requires a long time to reach the steady state. Hence, an adaptive time-stepping strategy (e.g., \eqref{Alg1:adaptive}) is an effective means for improving the computational efficiency without sacrificing accuracy. In our computations, the graded mesh with $\hat{T} = 0.5$, $ \hat{N} = 30 $ and $ \gamma = (2-\alpha)/\alpha $ is employed in the initial interval $ [0,\hat{T}] $, and the remainder $ [\hat{T}, T] $ is partitioned using three different types of temporal meshes, i.e., uniform large time step $\tau=2$ and small time step $\tau=0.02$, and adaptive time step based on \eqref{Alg1:adaptive}, with the parameters $ \tau_{\max} = 2, \tau_{\min} = 0.02, \eta = 10^{6}$. The other settings are the same as those used in Example \ref{Ex4_2}.

Figure \ref{figEx3_1} presents a comparison of the solution snapshots obtained using three types of temporal meshes up to $ T = 500 $ for the model with double-well potential. It is observed that using a large uniform time step $ \tau = 2 $ results in an inaccurate solution $ \phi $, while the adaptive time-stepping strategy yields coarsening patterns consistent with those obtained from the small uniform time step $ \tau = 0.02 $. Besides, the time evolutions of the maximum-norm of $\phi$ and the energy are depicted in Figures \ref{figEx3_2a}--\ref{figEx3_2b}, from which we can observe that (i) the proposed $L$1-sESAV scheme is energy-stable and MBP-preserving; (ii) the evolution curves of the maximum-norm of $ \phi $ and energy obtained by the adaptive strategy match very well with those generated by the small uniform time step. Moreover, Table \ref{tab1_Ex3} indicates the high efficiency of the proposed scheme combined with the adaptive strategy \eqref{Alg1:adaptive}. For instance, it takes more than one and a half hours for the implementation with uniform time step $ \tau = 0.02 $, whereas the adaptive method requires only about 12 seconds! In fact, the total number of adaptive time steps is only 577, which is significantly fewer than in the small uniform time step case. This improvement is mainly due to the frequent use of large time steps in the adaptive method. As shown in Figure \ref{figEx3_2c}, only a few small time steps are employed when the energy decays rapidly. For the Flory--Huggins potential, similar conclusions can also be drawn from Figures \ref{figEx3_3}--\ref{figEx3_4} and Table \ref{tab1_Ex3}.
\begin{figure}[!htbp]
\vspace{-10pt}
	\centering
	\subfigure[$t=5$]
	{
		\begin{minipage}[t]{0.24\linewidth}
			\centering
			\includegraphics[width=1.5in]{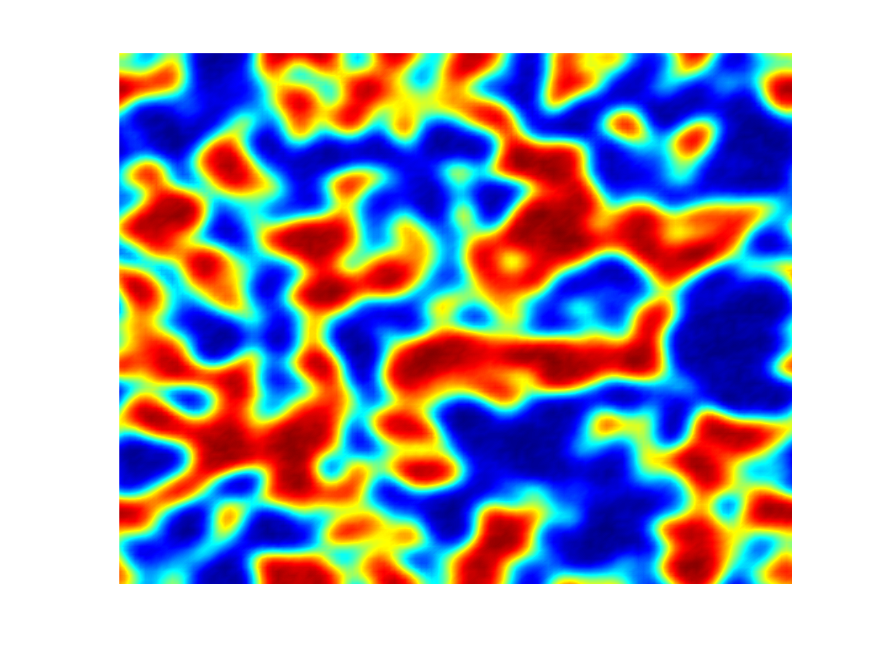}
			\includegraphics[width=1.5in]{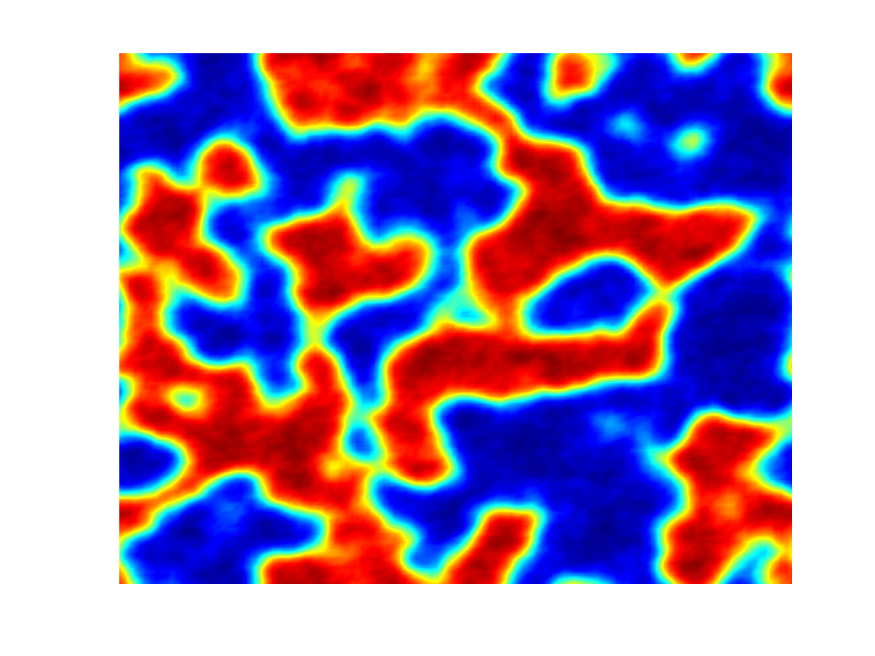}
			\includegraphics[width=1.5in]{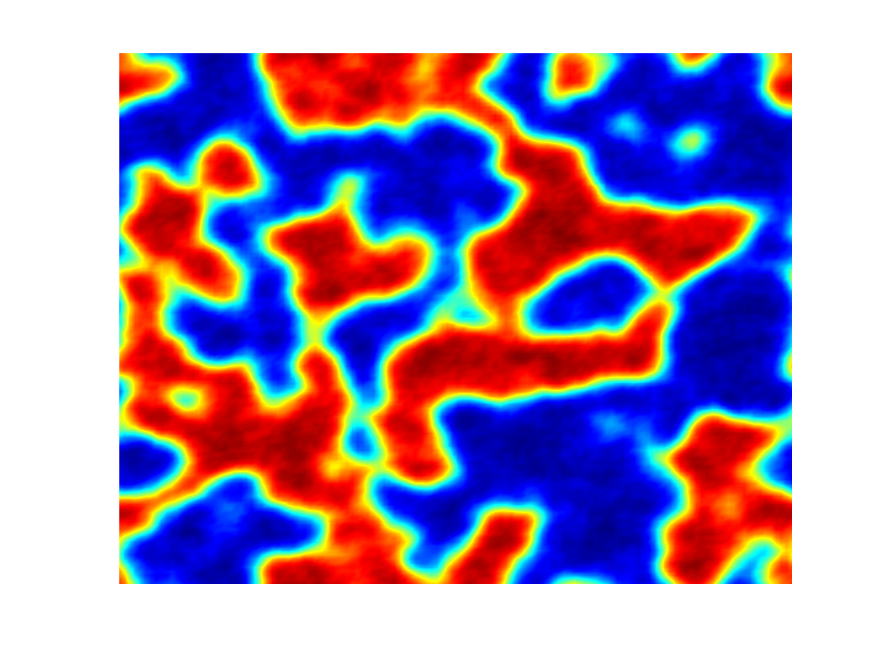}
		\end{minipage}%
	}%
	\subfigure[$t=20$]
	{
		\begin{minipage}[t]{0.24\linewidth}
			\centering
			\includegraphics[width=1.5in]{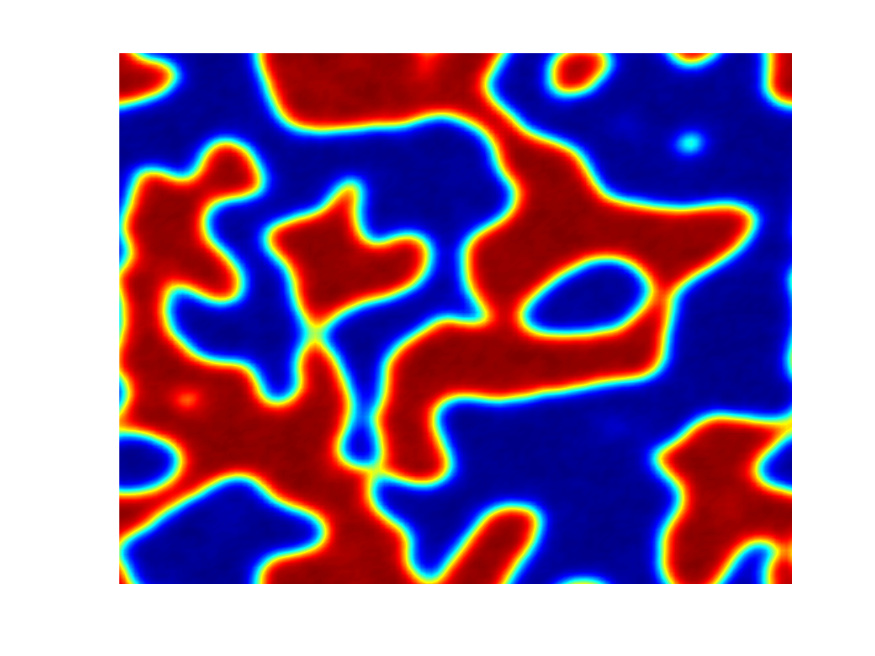}
			\includegraphics[width=1.5in]{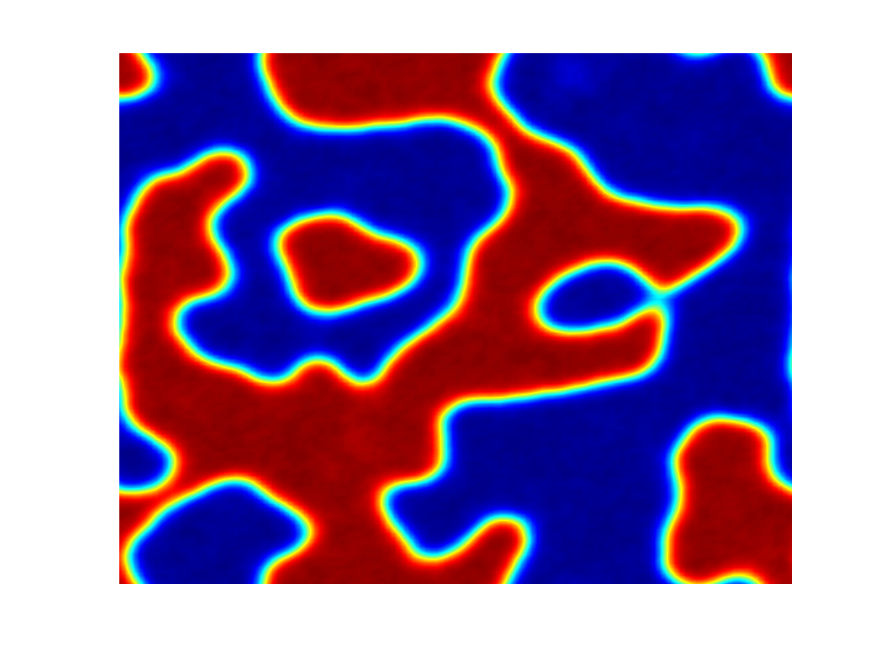}
			\includegraphics[width=1.5in]{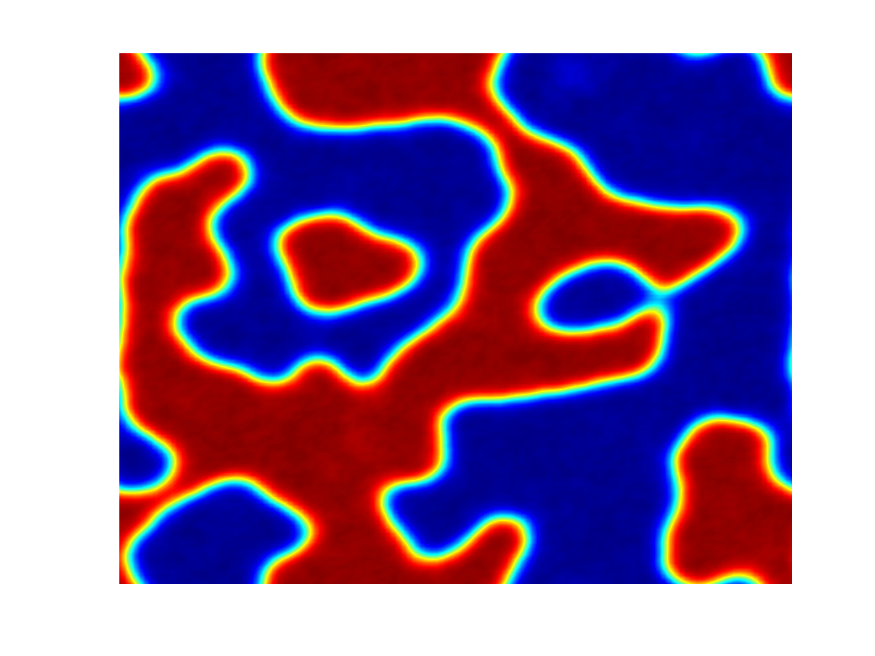}
		\end{minipage}%
	}%
	\subfigure[$t=100$]{
		\begin{minipage}[t]{0.24\linewidth}
			\centering
			\includegraphics[width=1.5in]{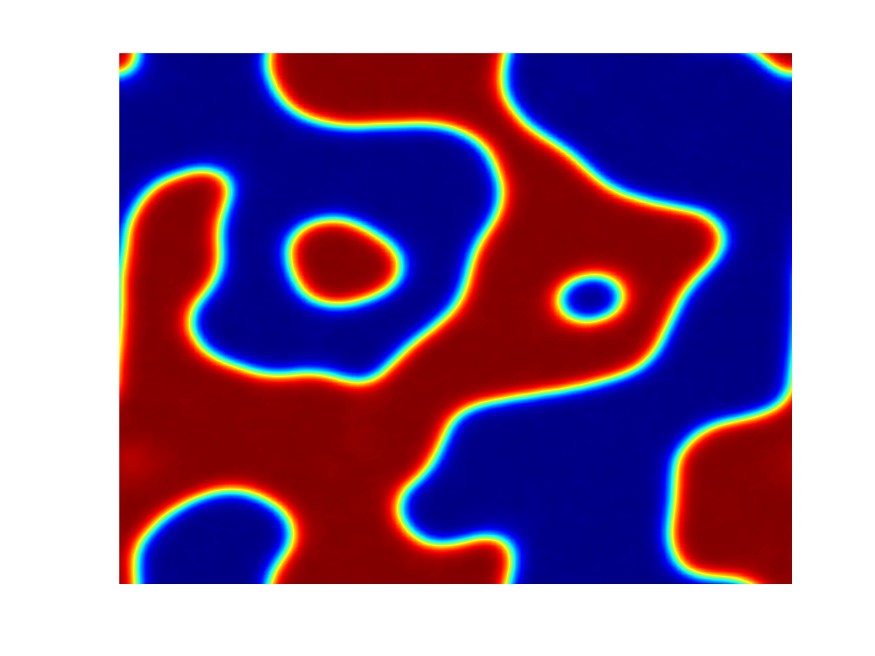}
			\includegraphics[width=1.5in]{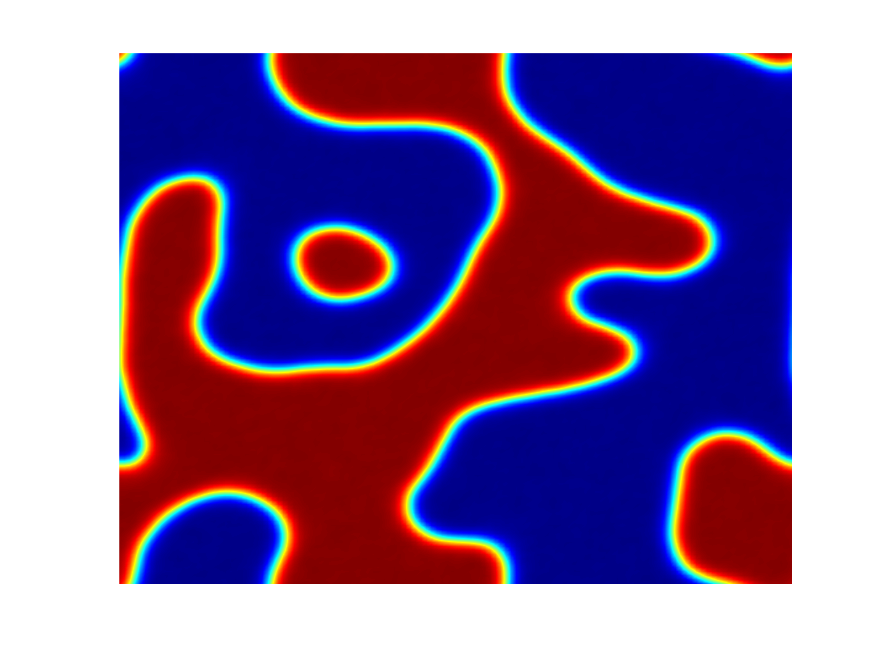}
			\includegraphics[width=1.5in]{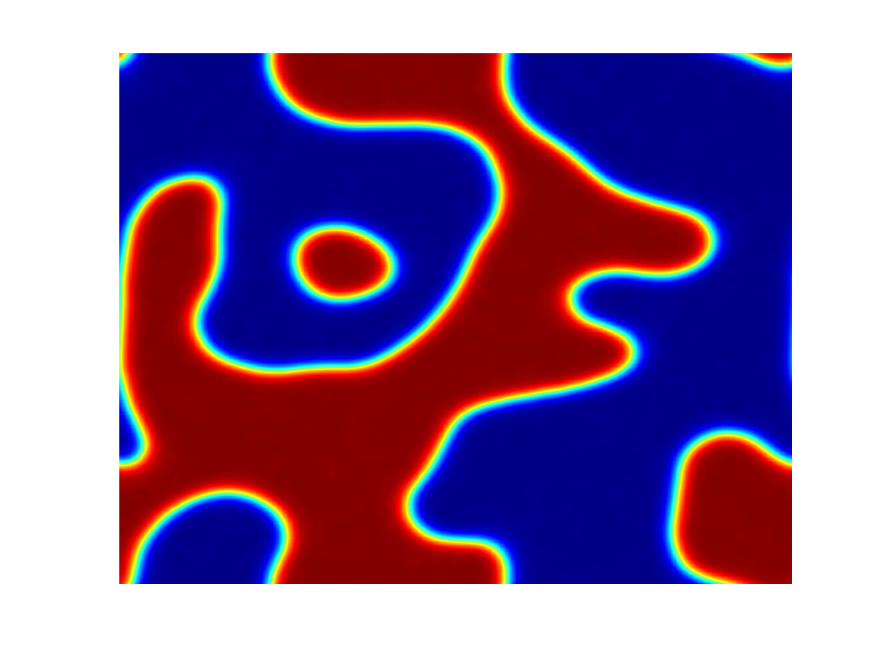}
		\end{minipage}%
	}%
	\subfigure[$t=500$]
	{
		\begin{minipage}[t]{0.24\linewidth}
			\centering
			\includegraphics[width=1.5in]{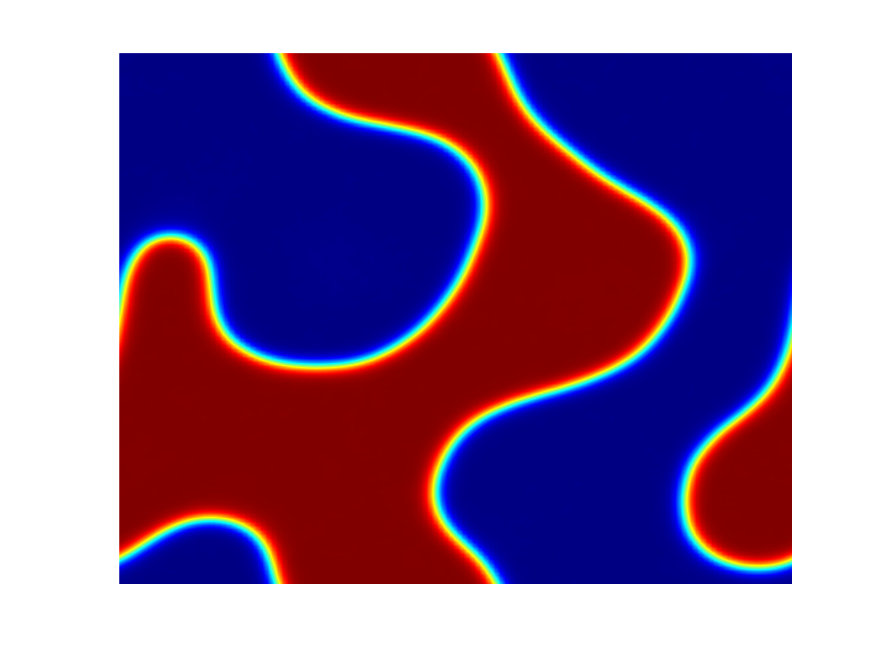}
			\includegraphics[width=1.5in]{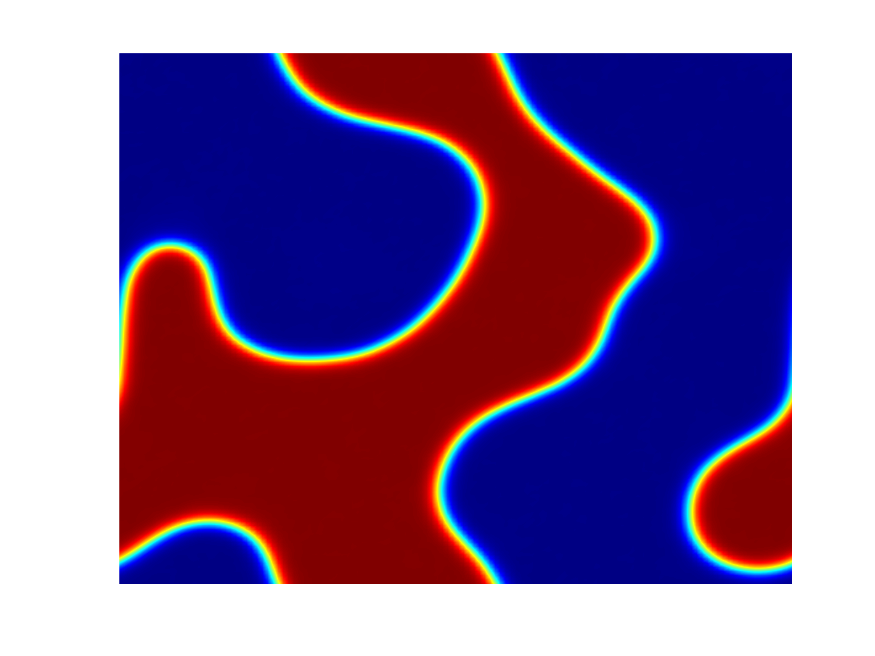}
			\includegraphics[width=1.5in]{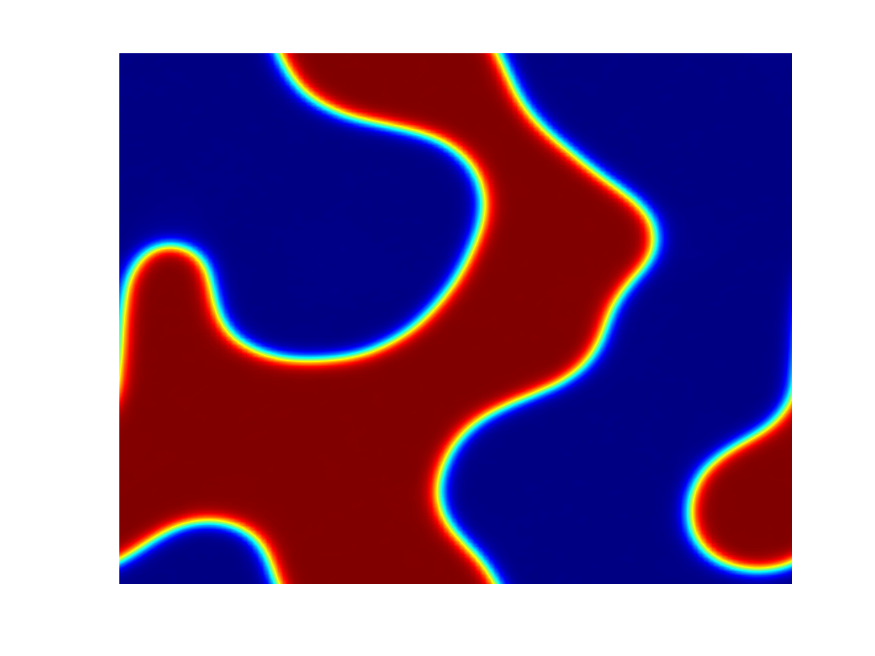}
		\end{minipage}%
	}%
	\setlength{\abovecaptionskip}{0.0cm} 
	\setlength{\belowcaptionskip}{0.0cm}
	\caption{The dynamic snapshots of the numerical solution $\phi$ obtained by the $L$1-sESAV scheme with the uniform  (top, $\tau = 2$; bottom, $\tau = 0.02$) and adaptive (middle) time steps: the double-well potential}	\label{figEx3_1}
        \vspace{-10pt}
\end{figure}
\begin{figure}[!htbp]
	\vspace{-12pt}
	\centering
	\subfigure[maximum-norm of $\phi$]
	{
		\includegraphics[width=0.32\textwidth]{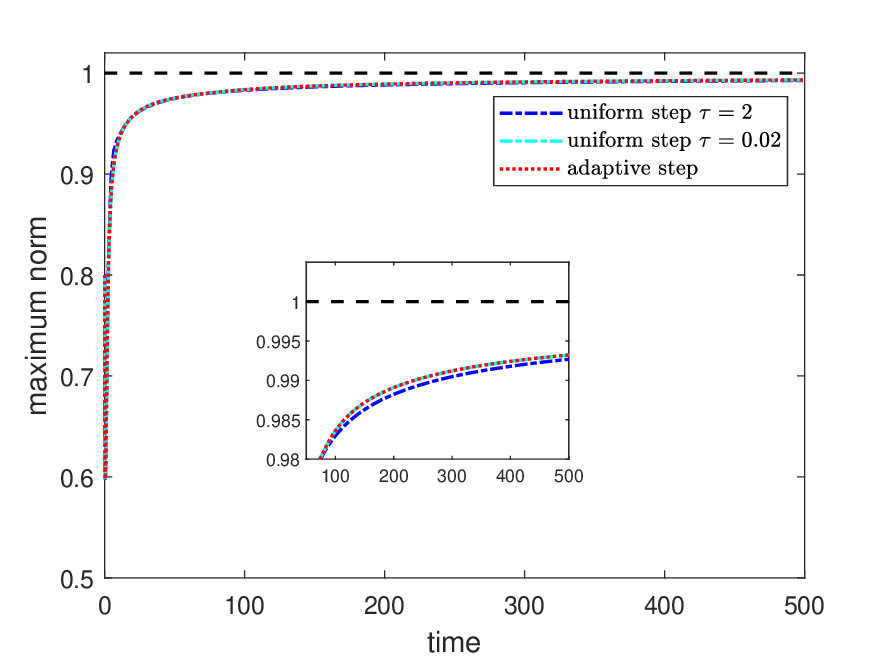}	\label{figEx3_2a}
	}%
	\subfigure[energy]
	{
		\includegraphics[width=0.32\textwidth]{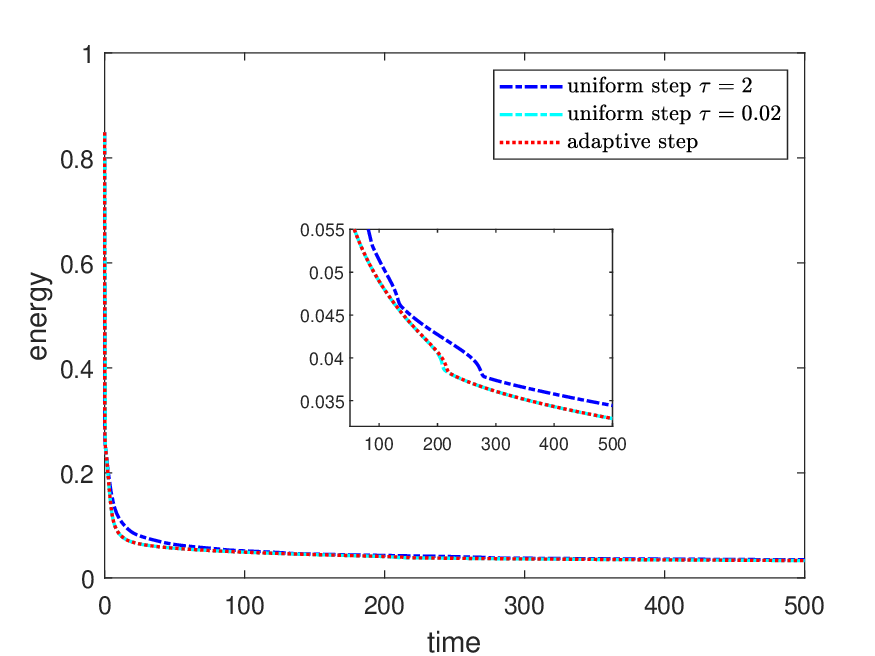}	\label{figEx3_2b}
	}%
	\subfigure[time steps]
	{
		\includegraphics[width=0.32\textwidth]{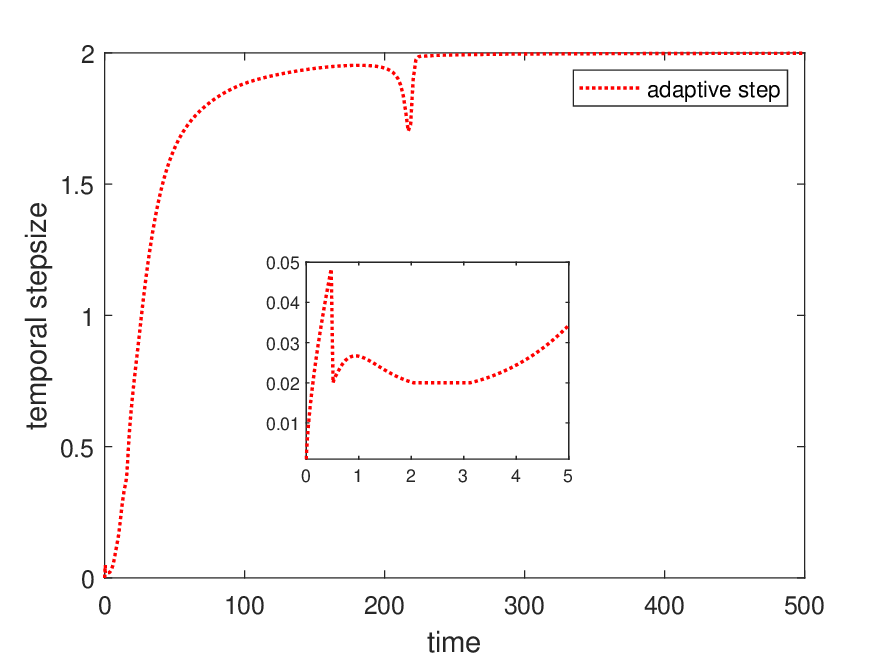}	\label{figEx3_2c}
	}%
	\setlength{\abovecaptionskip}{0.0cm} 
	\setlength{\belowcaptionskip}{0.0cm}
	\caption{Time evolutions of the maximum norm (left), energy (middle), and time steps (right) for the $L$1-sESAV scheme: the double-well potential}	\label{figEx3_2}
    \vspace{-10pt}
\end{figure}
\begin{table}[!htbp]
	\vspace{-12pt}
	\caption{CPU times and the total number of time steps yielded by the $L$1-sESAV scheme}\label{tab1_Ex3}
	{\footnotesize\begin{tabular*}{\columnwidth}{@{\extracolsep\fill}ccccc@{\extracolsep\fill}}
			\toprule
			 & \multicolumn{2}{c}{double-well potential } & \multicolumn{2}{c}{Flory--Huggins potential  }\\
			\midrule
			time-stepping strategy  & $N$ &  CPU times & $N$ & CPU times \\
			\midrule
			uniform step  $ \tau = 2$      &   281   &  4.81 s   &  281   &  4.82 s          \\ 
			adaptive step    &   577   &  11.17 s   &  598   &  12.05 s           \\ 
			uniform step  $ \tau = 0.02 $     &   25007  &  1 h 39 m 24 s  &  25007  &  1 h 40 m 24 s  \\ 
			\bottomrule
	\end{tabular*}}
   	\vspace{-10pt}
\end{table}
\begin{figure}[!thbp]
	\centering
	\subfigure[$t=5$]
	{
		\begin{minipage}[t]{0.24\linewidth}
			\centering
			\includegraphics[width=1.5in]{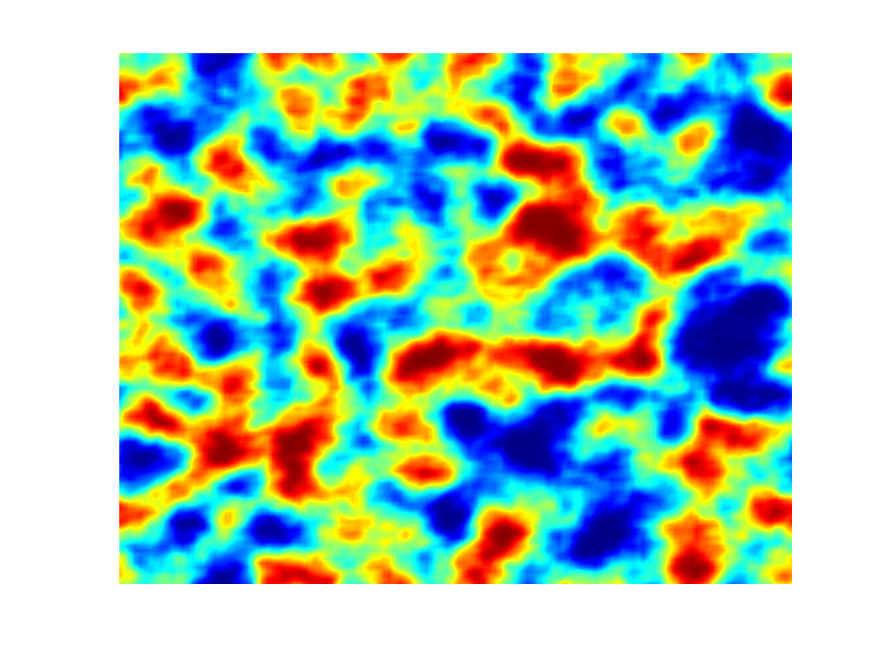}
			\includegraphics[width=1.5in]{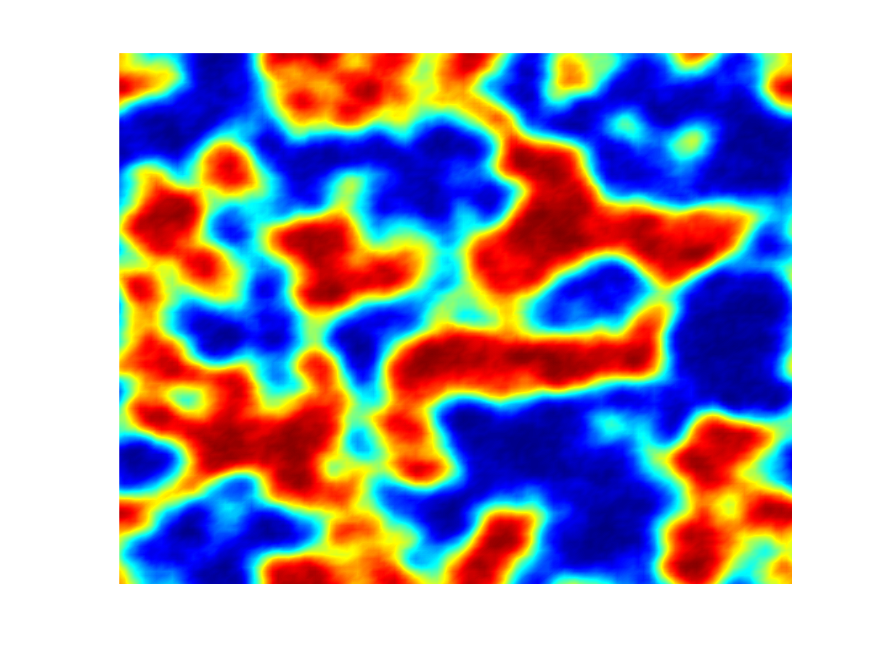}
			\includegraphics[width=1.5in]{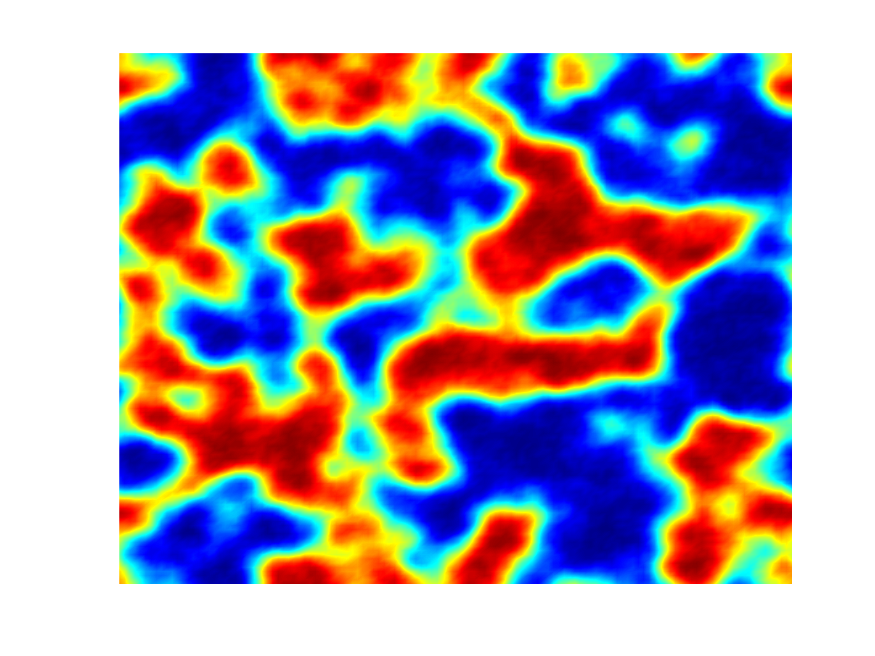}
		\end{minipage}%
	}%
	\subfigure[$t=20$]
	{
		\begin{minipage}[t]{0.24\linewidth}
			\centering
			\includegraphics[width=1.5in]{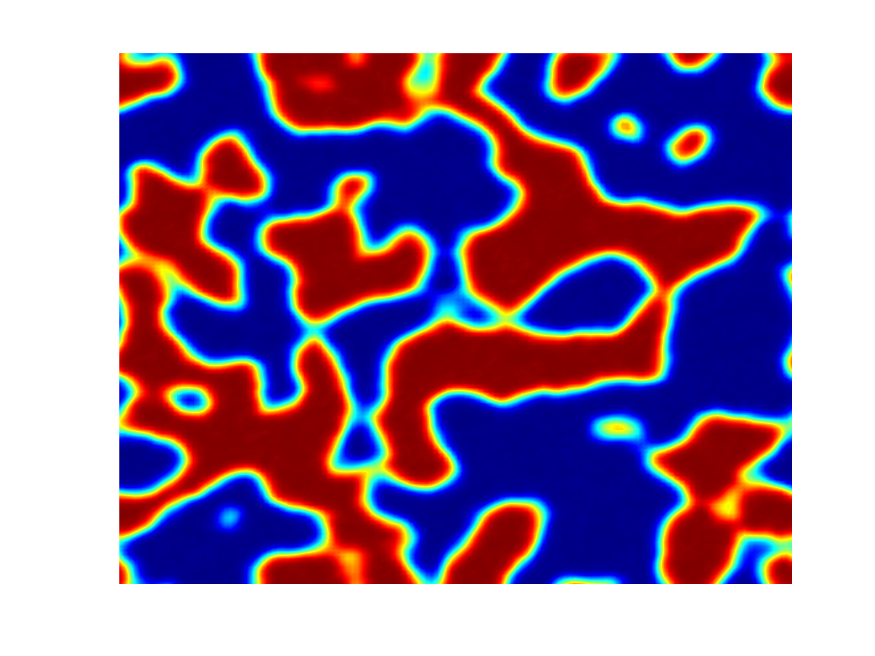}
			\includegraphics[width=1.5in]{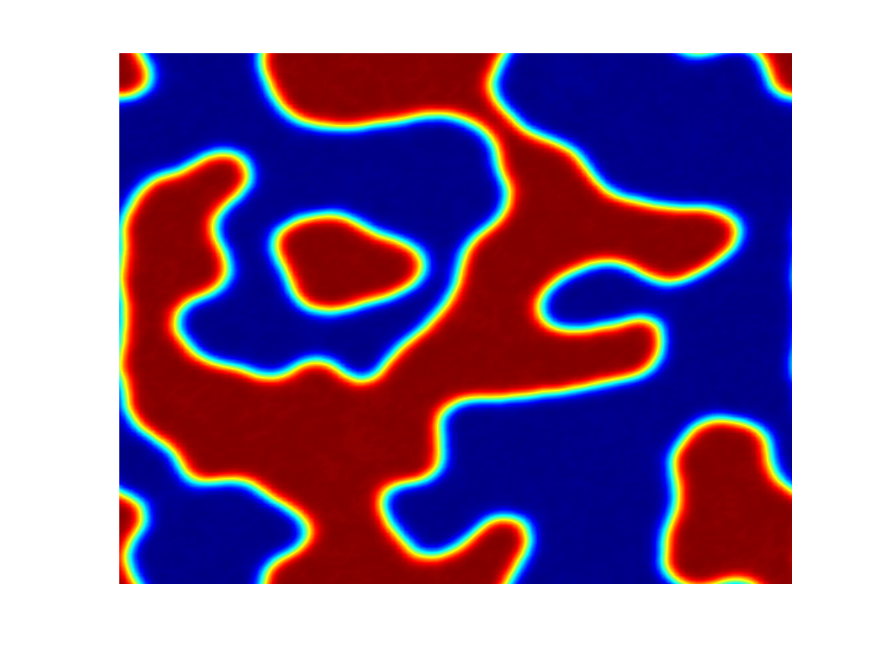}
			\includegraphics[width=1.5in]{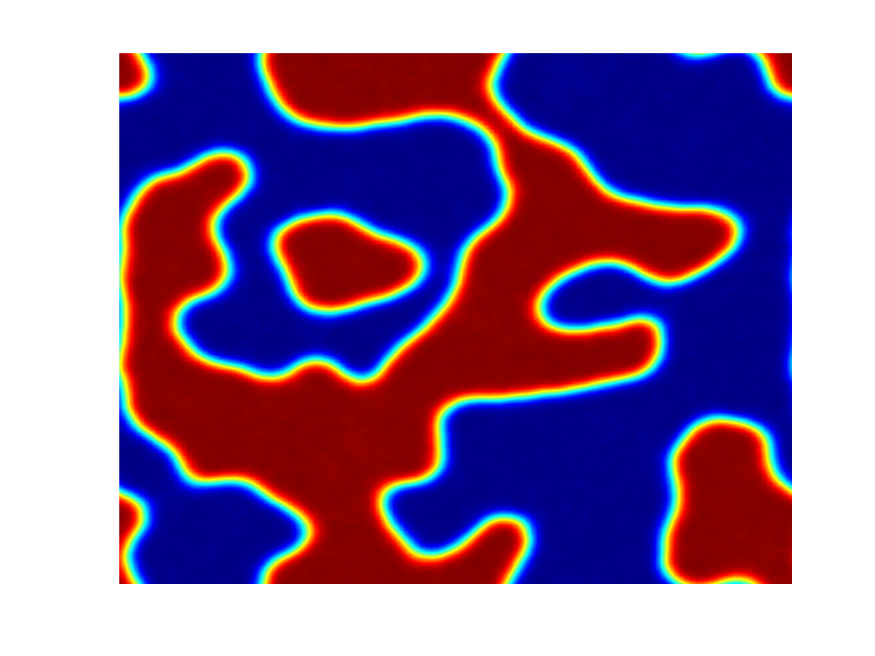}
		\end{minipage}%
	}%
	\subfigure[$t=100$]{
		\begin{minipage}[t]{0.24\linewidth}
			\centering
			\includegraphics[width=1.5in]{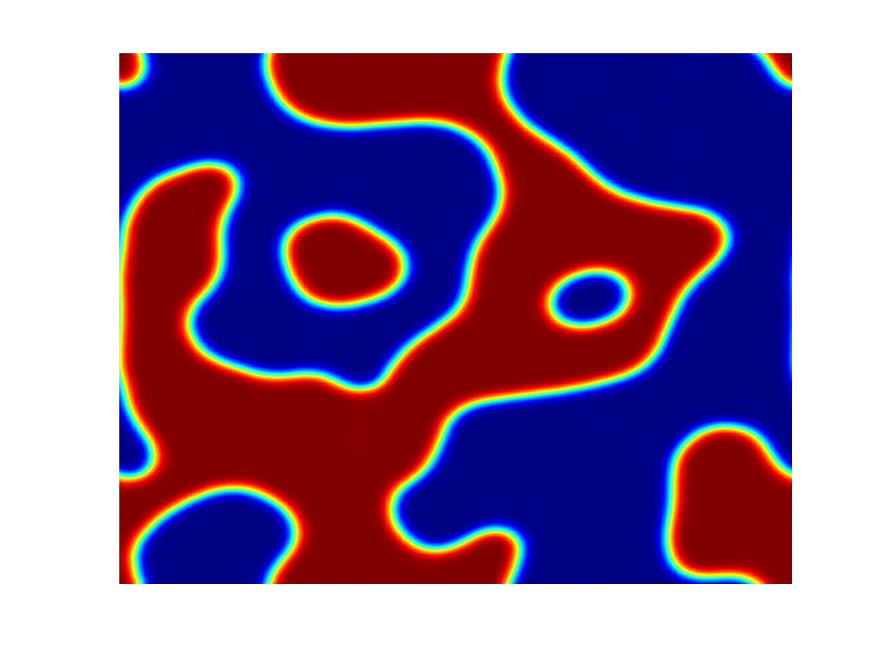}
			\includegraphics[width=1.5in]{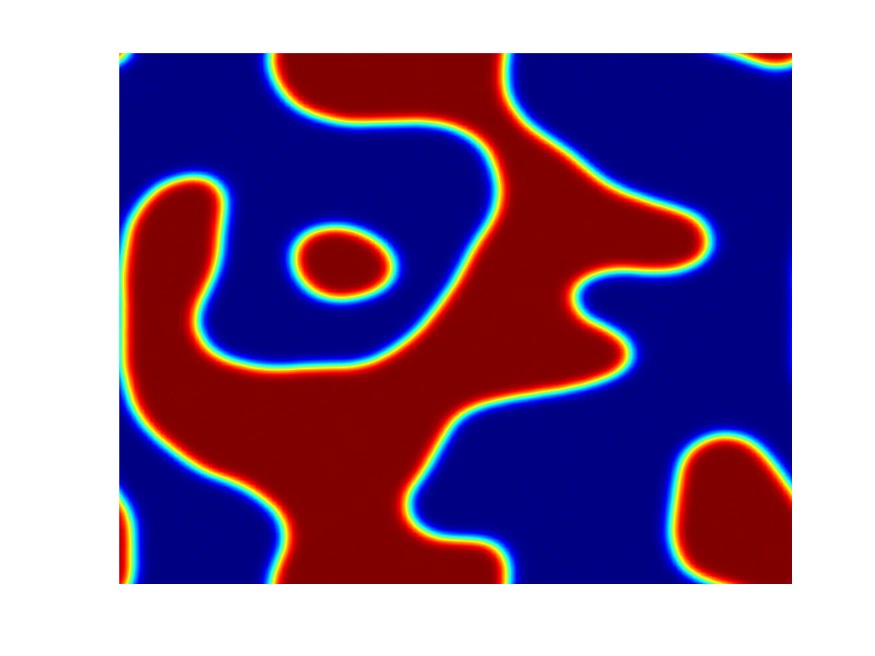}
			\includegraphics[width=1.5in]{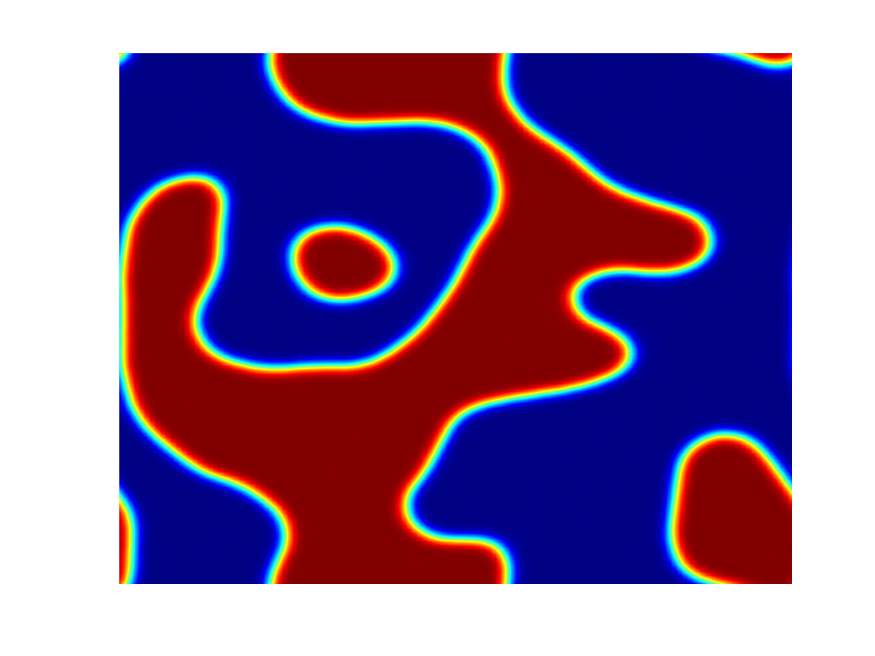}
		\end{minipage}%
	}%
	\subfigure[$t=500$]
	{
		\begin{minipage}[t]{0.24\linewidth}
			\centering
			\includegraphics[width=1.5in]{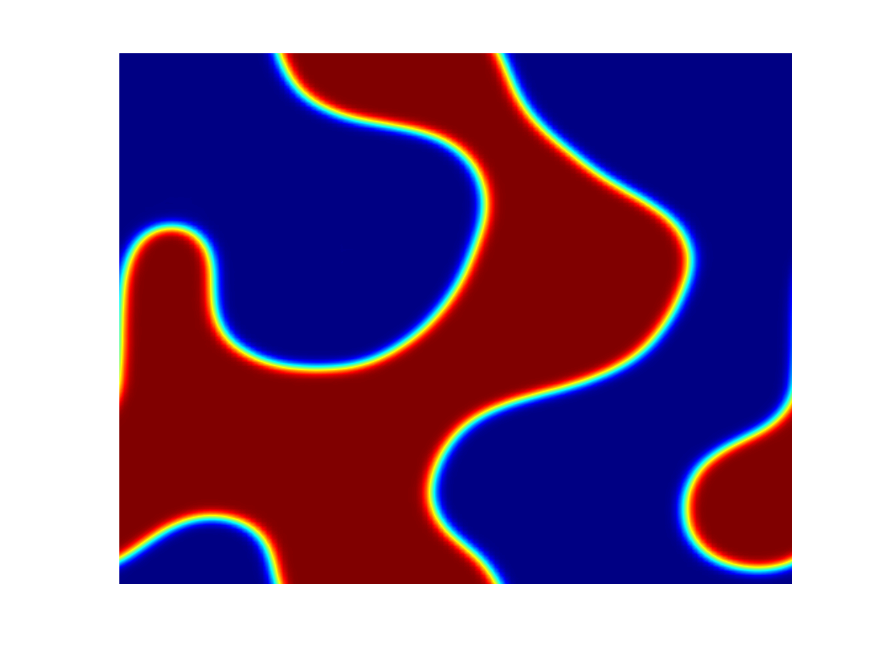}
			\includegraphics[width=1.5in]{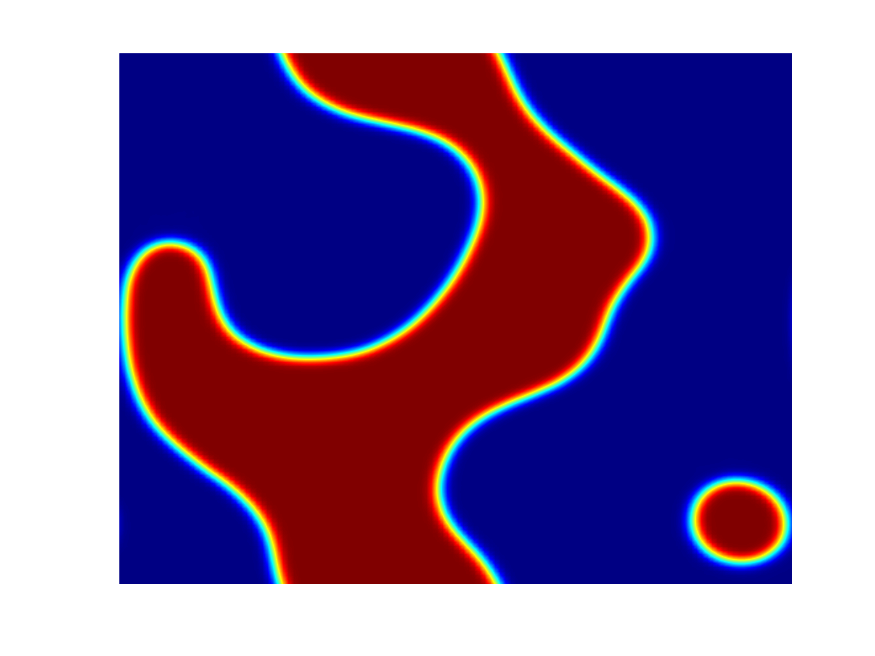}
			\includegraphics[width=1.5in]{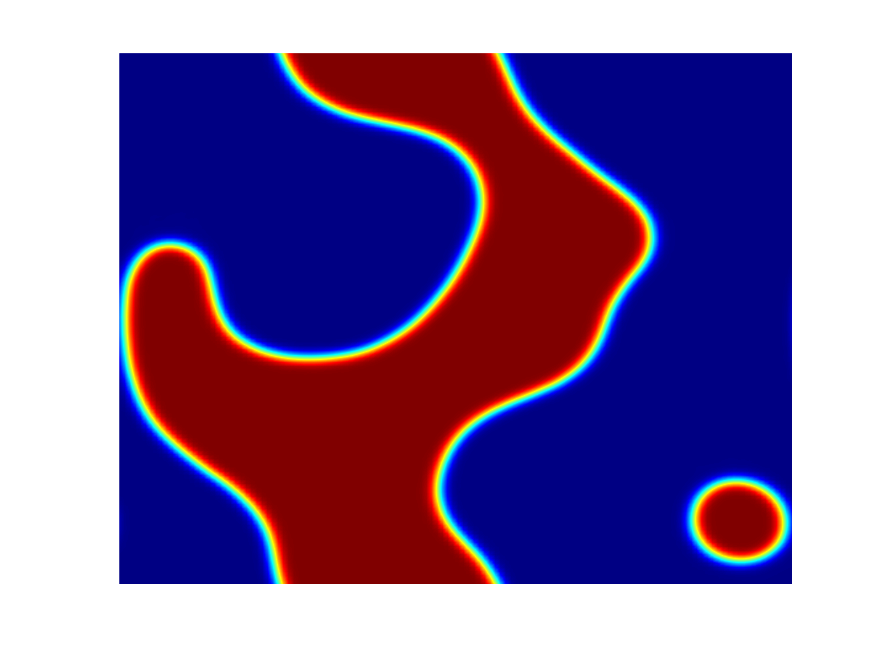}
		\end{minipage}%
	}%
	\setlength{\abovecaptionskip}{0.0cm} 
	\setlength{\belowcaptionskip}{0.0cm}
	\caption{The dynamic snapshots of the numerical solution $\phi$ obtained by the $L$1-sESAV scheme with the uniform  (top, $\tau = 2$; bottom, $\tau = 0.02$) and adaptive (middle) time steps: the Flory--Huggins potential}	\label{figEx3_3}
\end{figure}
\begin{figure}[!th]
	\vspace{-12pt}
	\centering
	\subfigure[maximum norm of $\phi$]
	{
		\includegraphics[width=0.32\textwidth]{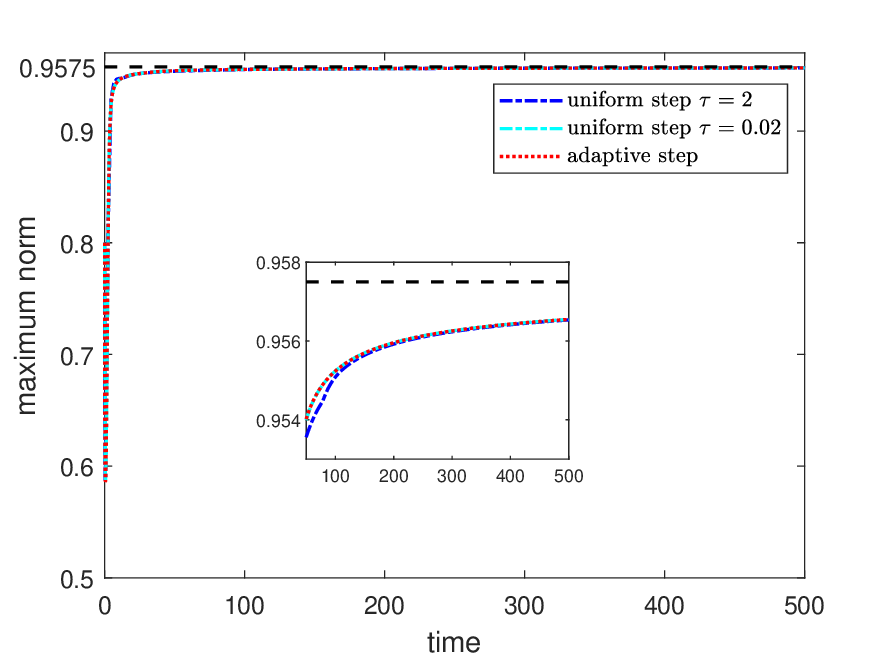}
		\label{figEx3_4a}
	}%
	\subfigure[energy]
	{
		\includegraphics[width=0.32\textwidth]{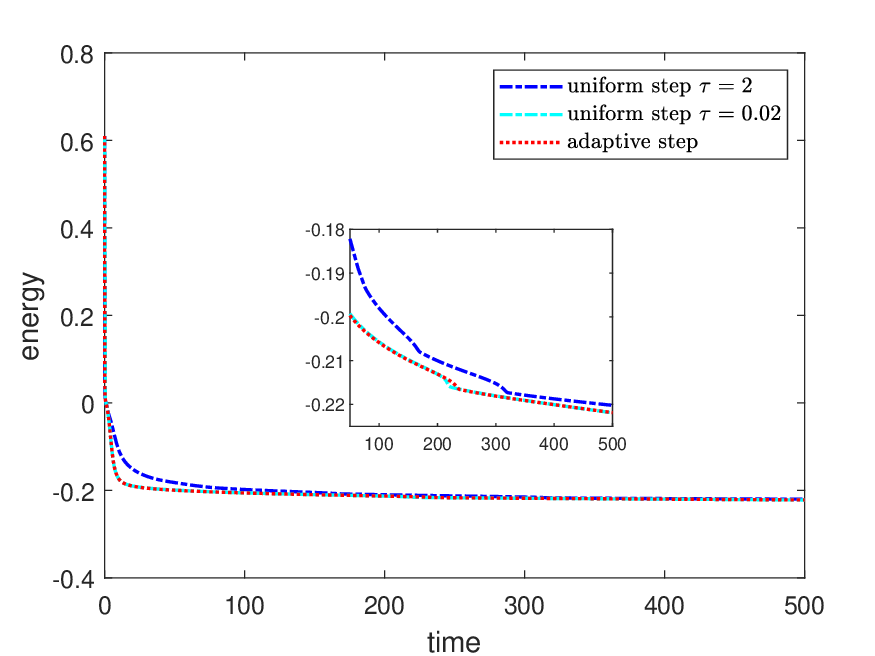}
		\label{figEx3_4b}
	}%
	\subfigure[time steps]
	{
		\includegraphics[width=0.32\textwidth]{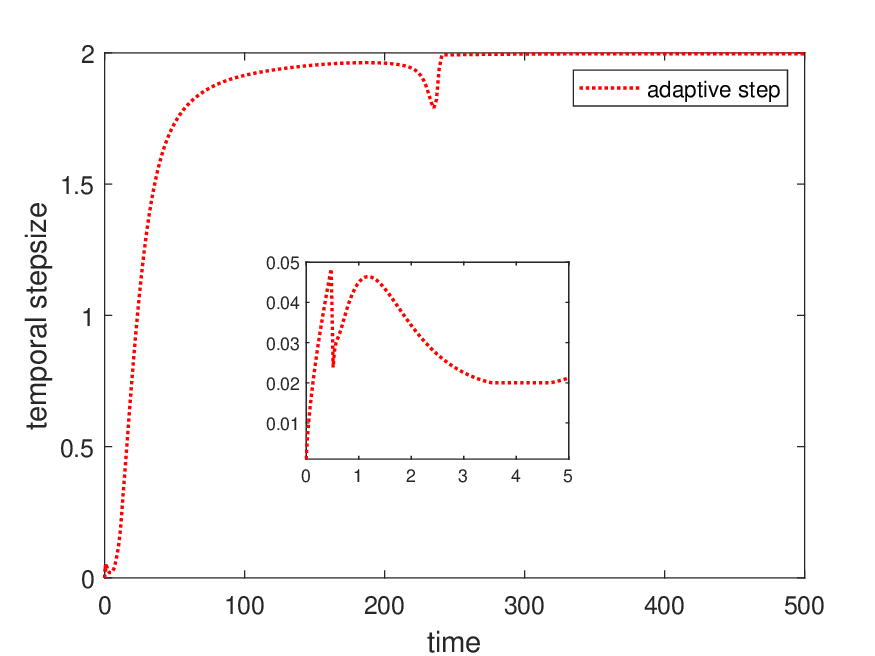}
		\label{figEx3_4c}
	}%
	\setlength{\abovecaptionskip}{0.0cm} 
	\setlength{\belowcaptionskip}{0.0cm}
	\caption{Time evolutions of the maximum norm (left), energy (middle), and time steps (right) for the $L$1-sESAV scheme: the Flory--Huggins potential}\label{figEx3_4}
  	\vspace{-5pt}
\end{figure}

\subsection{ Dynamical behaviour governed by the tFAC model}
In order to discover the influence of fractional order $\alpha$ on the dynamical behaviour, we consider the 2D tFAC model \eqref{Model:tAC} with $ \mm = 1 $, $ \varepsilon = 0.01 $, and double-well potential. The initial state is chosen as
$$
	\phi_{\text{init}}(x,y) = \tanh \f{ 1.5 + 1.2\cos( 6\vartheta ) - 2\pi r }{ \sqrt{2 \lambda } },~
	\vartheta = \arctan \f{ y - 0.5 }{ x - 0.5 }, ~ r = \sqrt{ \left( x - 0.5 \right)^2 + \left( y - 0.5 \right)^2 }.
$$
The computational domain $ \Omega = (0,1)^{2} $ is uniformly divided into 128 elements along each spatial direction.

In the adaptive time-stepping strategy \eqref{Alg1:adaptive}, we set $ \tau_{\max} = 2 $, $ \tau_{\min} = 0.02 $, and the parameter $ \eta = 10^{7} $. The time evolutions of the phase-field function with different fractional orders $ \alpha = 0.9, 0.7 $ and $0.4$ at different time instants are shown in Figure \ref{figEx4_1}. It is clearly seen that the dynamical behaviour is significantly affected by the fractional order: the bigger the $ \alpha $ is, the faster the dynamics evolutes. Thus, smaller $ \alpha $ requires a longer time to reach equilibrium. This observation is further confirmed by the energy curves shown in Figure \ref{figEx4_2}. In fact, for $ \alpha = 0.9 $, the dynamical process reaches the steady state around $ t = 600 $ and the corresponding energy is nearly zero; while for $\alpha=0.7$ and $ 0.4 $, even at $ t = 700 $, it has not reached the steady state.
\begin{figure}[!htbp]
  \vspace{-10pt}
	\centering
	\subfigure[$t=10$]
	{
		\begin{minipage}[t]{0.18\linewidth}
			\centering
			\includegraphics[width=1.2in]{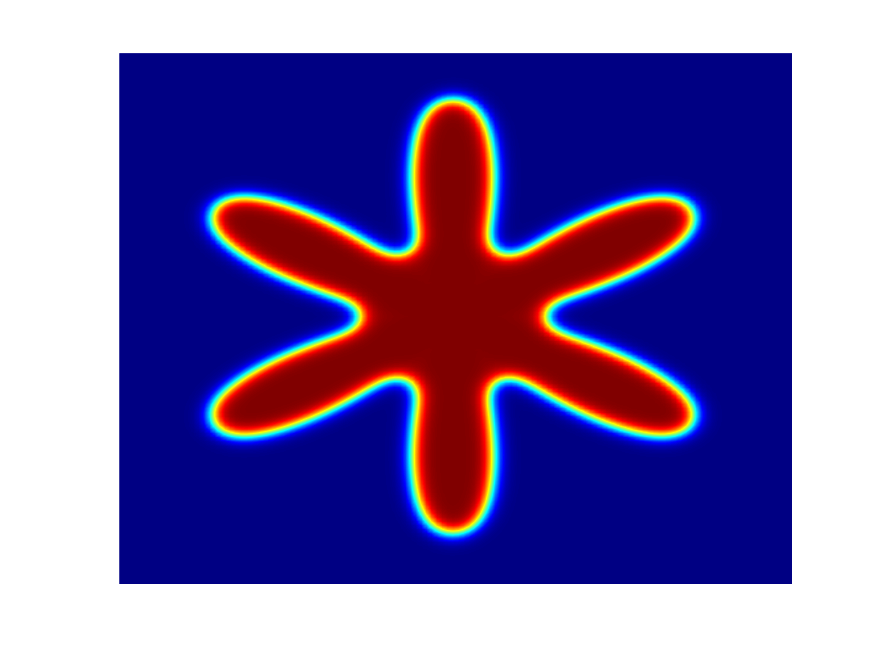}
			\includegraphics[width=1.2in]{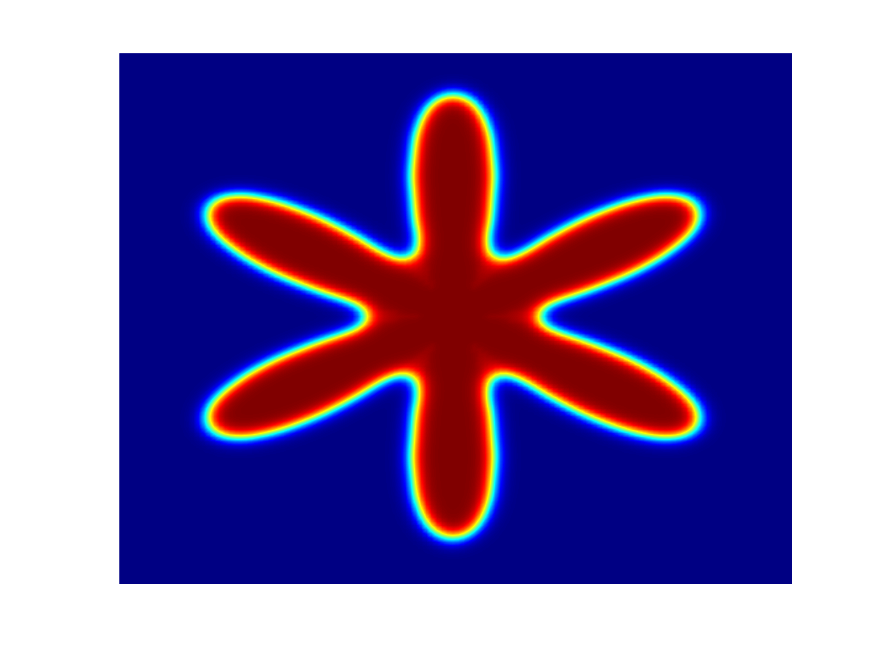}
			\includegraphics[width=1.2in]{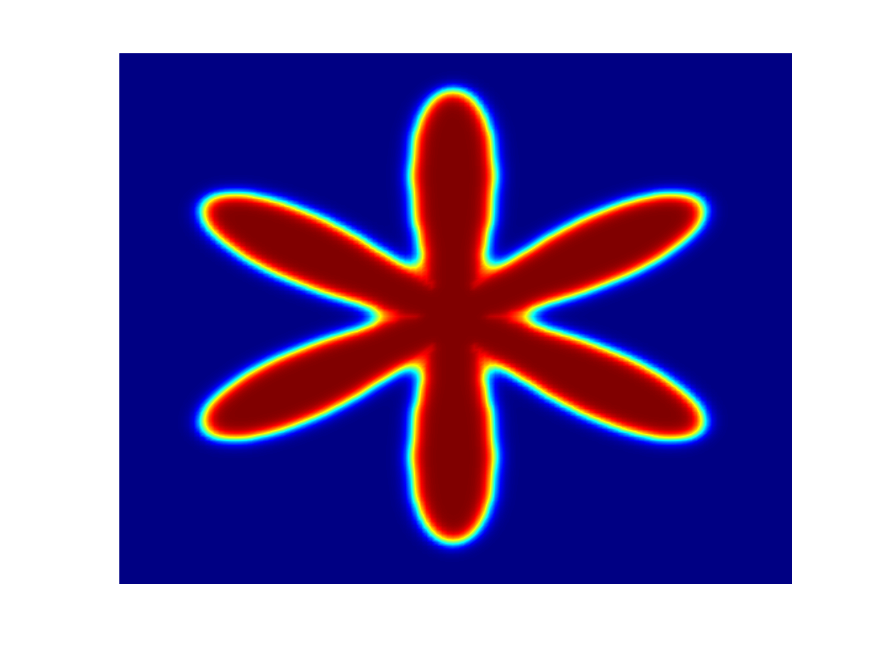}
		\end{minipage}%
	}%
	\subfigure[$t=50$]
	{
		\begin{minipage}[t]{0.18\linewidth}
			\centering
			\includegraphics[width=1.2in]{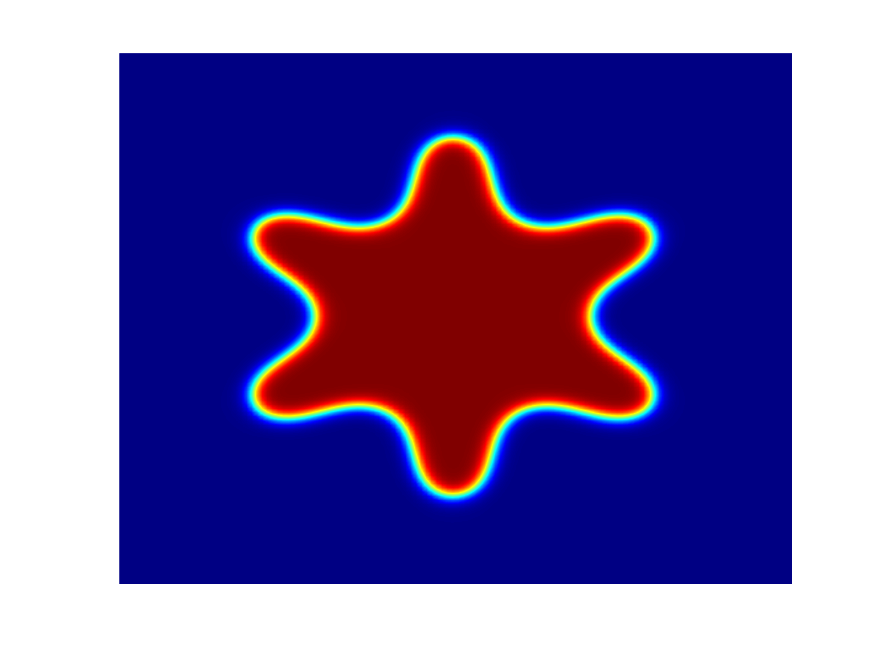}
			\includegraphics[width=1.2in]{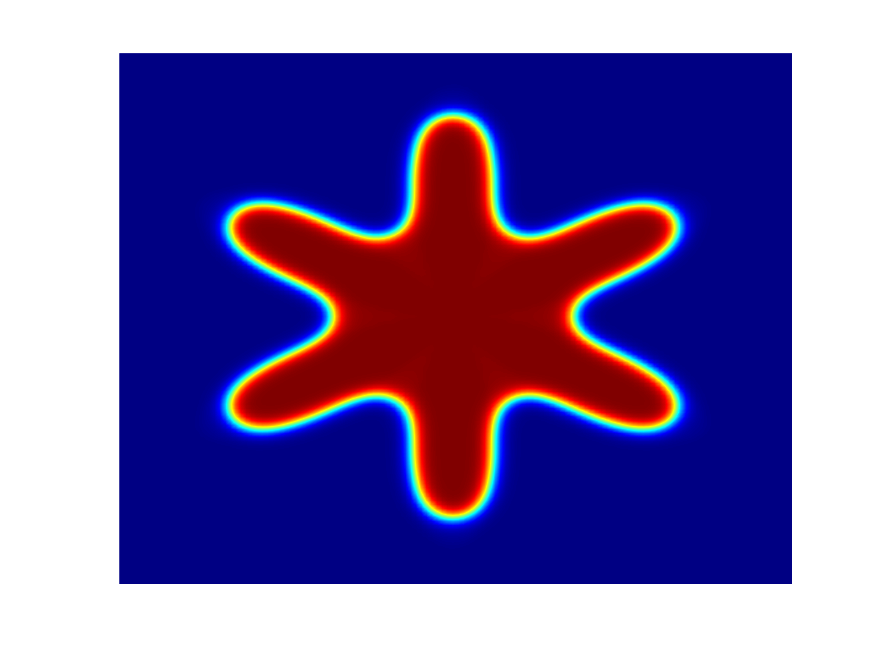}
			\includegraphics[width=1.2in]{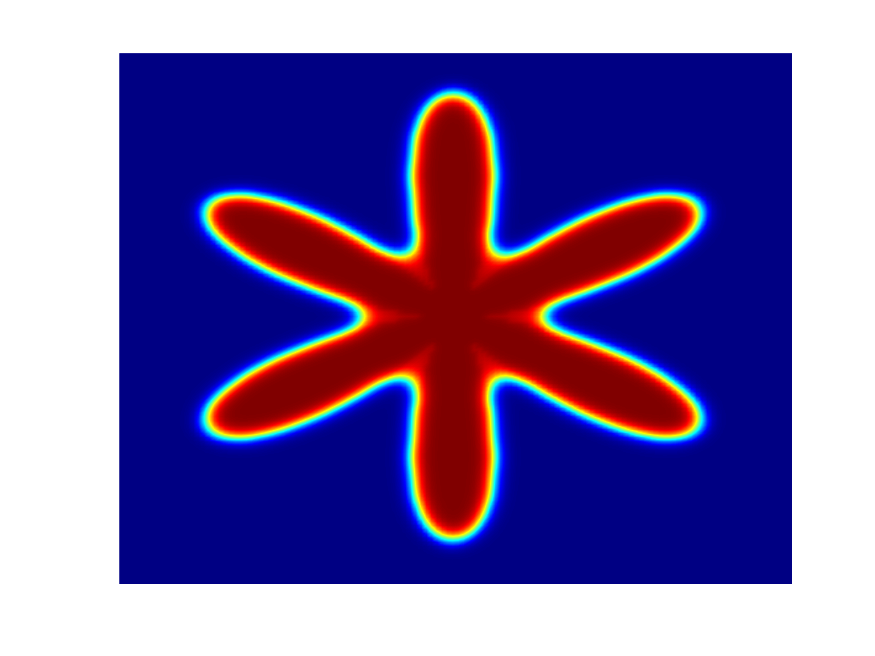}
		\end{minipage}%
	}%
	\subfigure[$t=100$]{
		\begin{minipage}[t]{0.18\linewidth}
			\centering
			\includegraphics[width=1.2in]{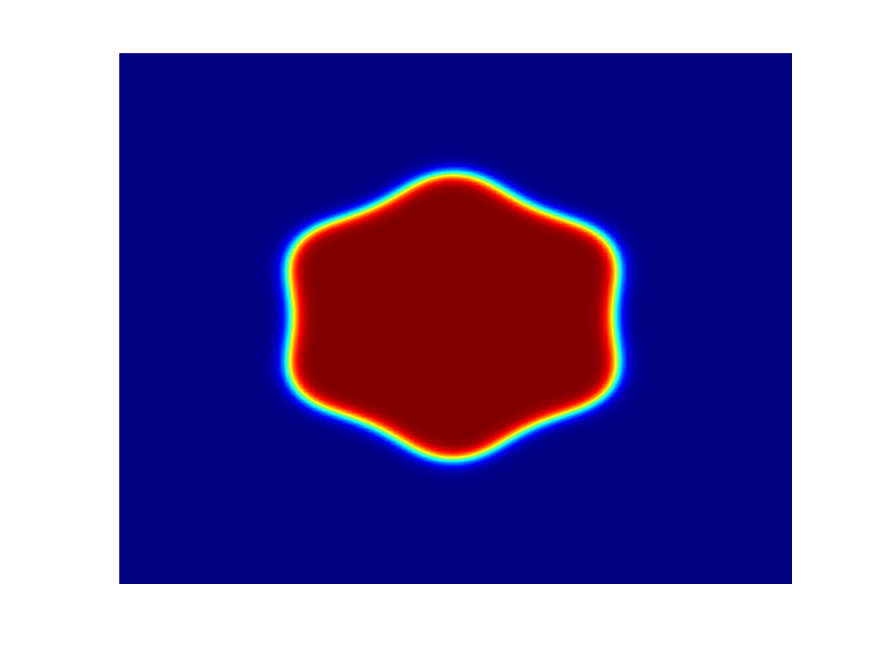}
			\includegraphics[width=1.2in]{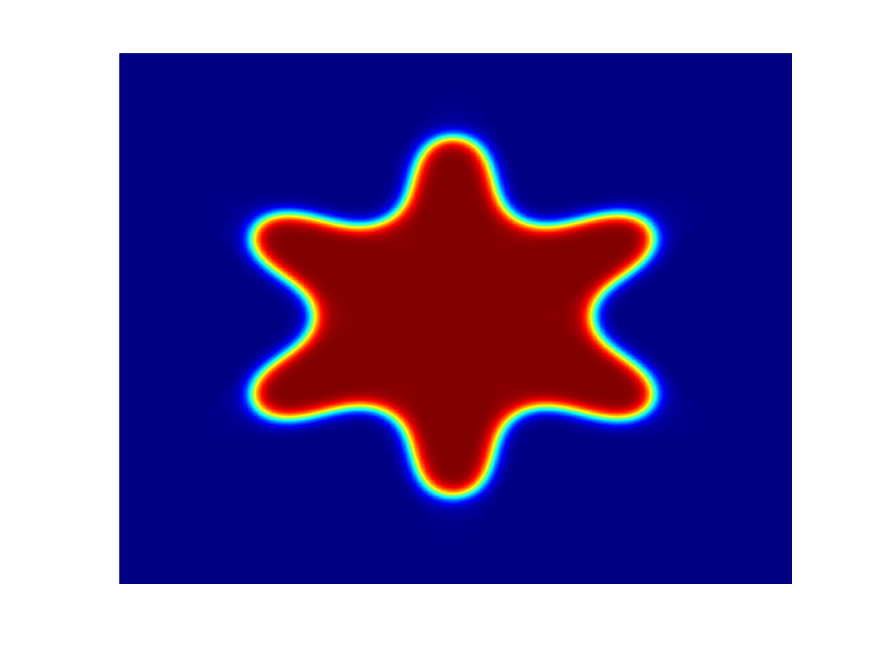}
			\includegraphics[width=1.2in]{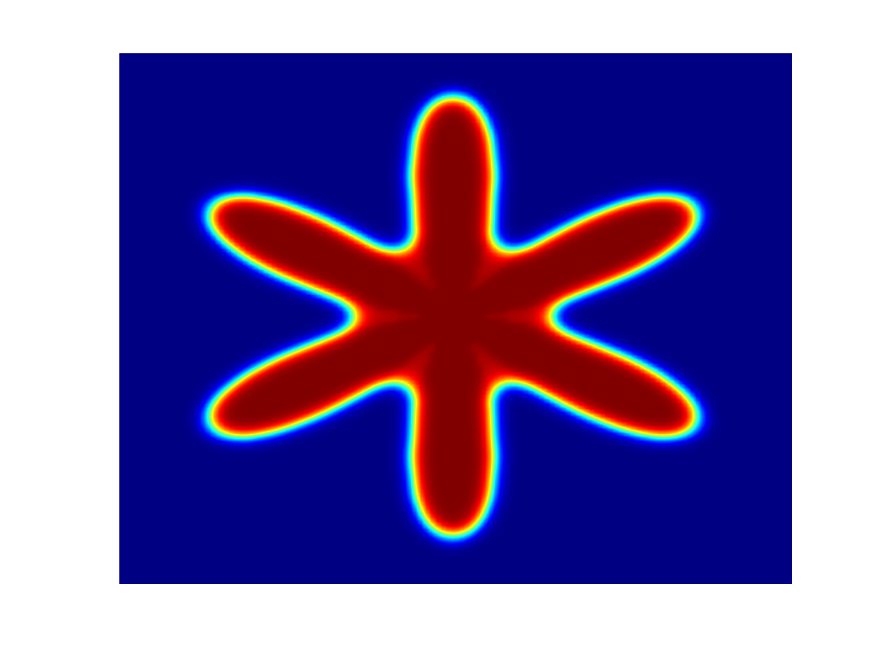}
		\end{minipage}%
	}%
	\subfigure[$t=300$]
	{
		\begin{minipage}[t]{0.18\linewidth}
			\centering
			\includegraphics[width=1.2in]{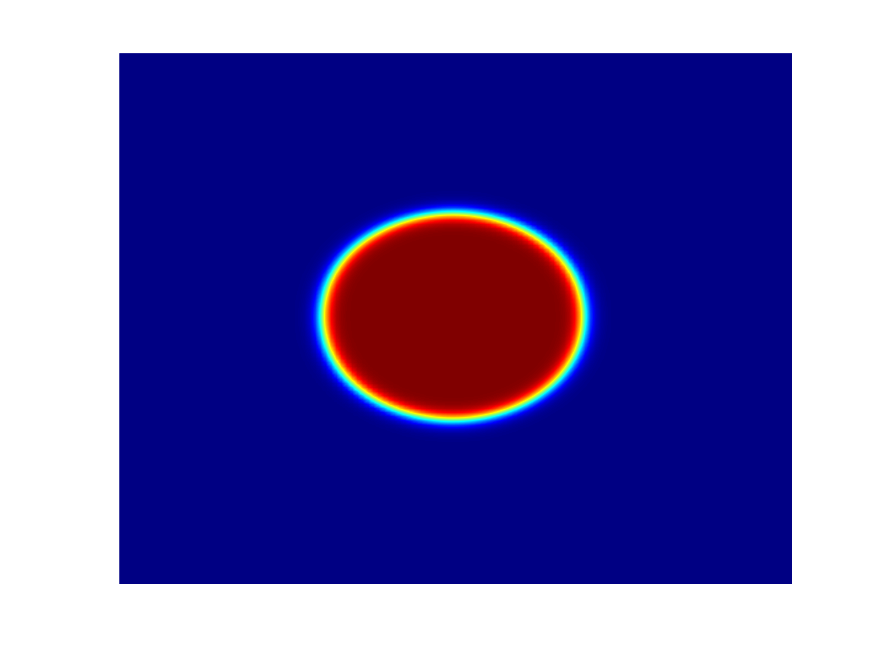}
			\includegraphics[width=1.2in]{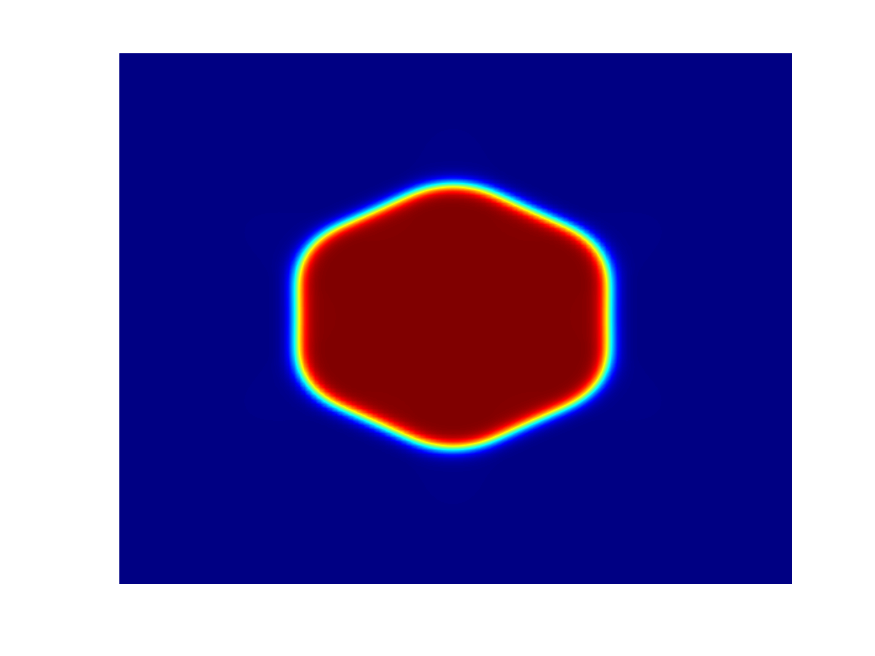}
			\includegraphics[width=1.2in]{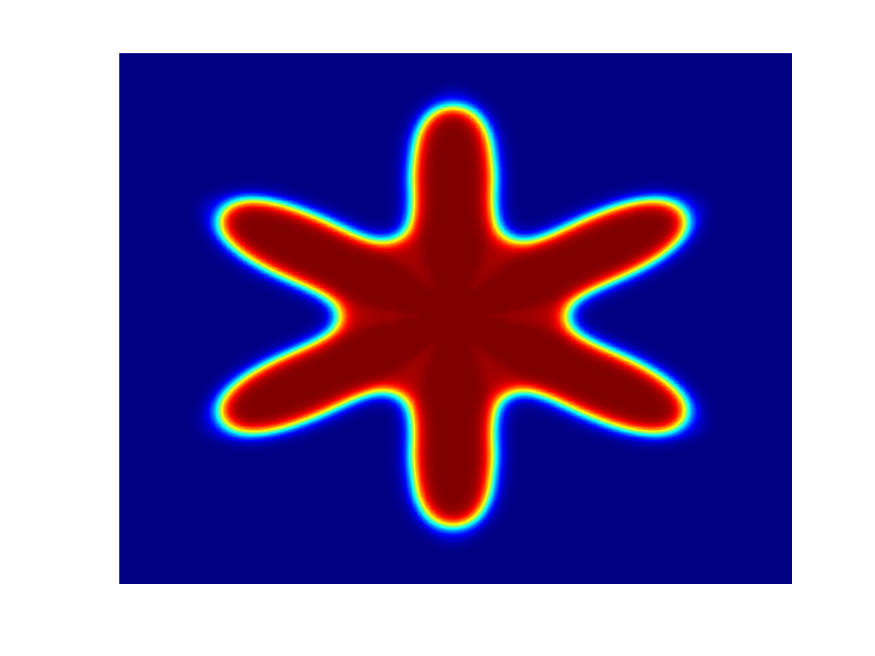}
		\end{minipage}%
	}%
	\subfigure[$t=700$]
	{
		\begin{minipage}[t]{0.18\linewidth}
			\centering
			\includegraphics[width=1.2in]{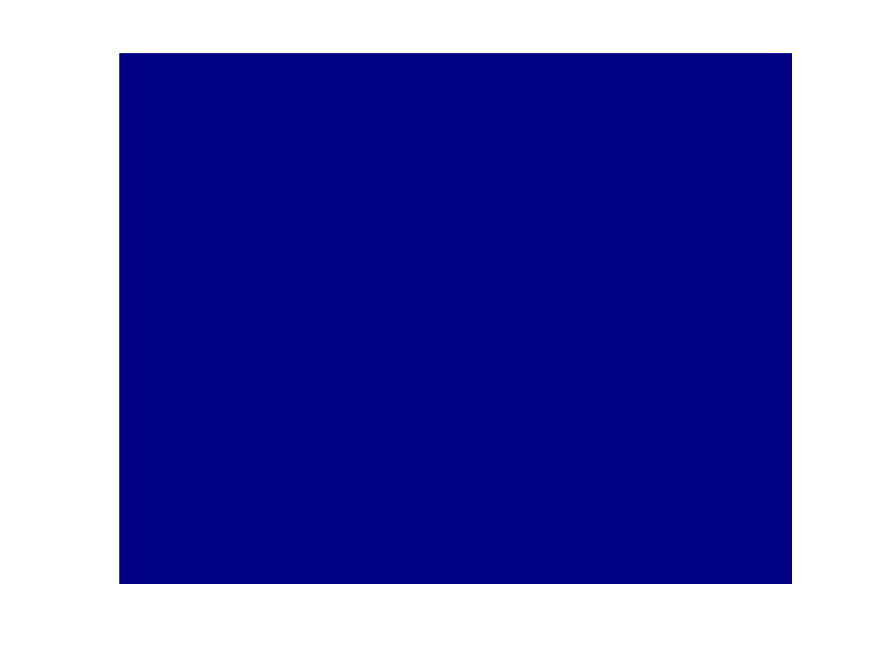}
			\includegraphics[width=1.2in]{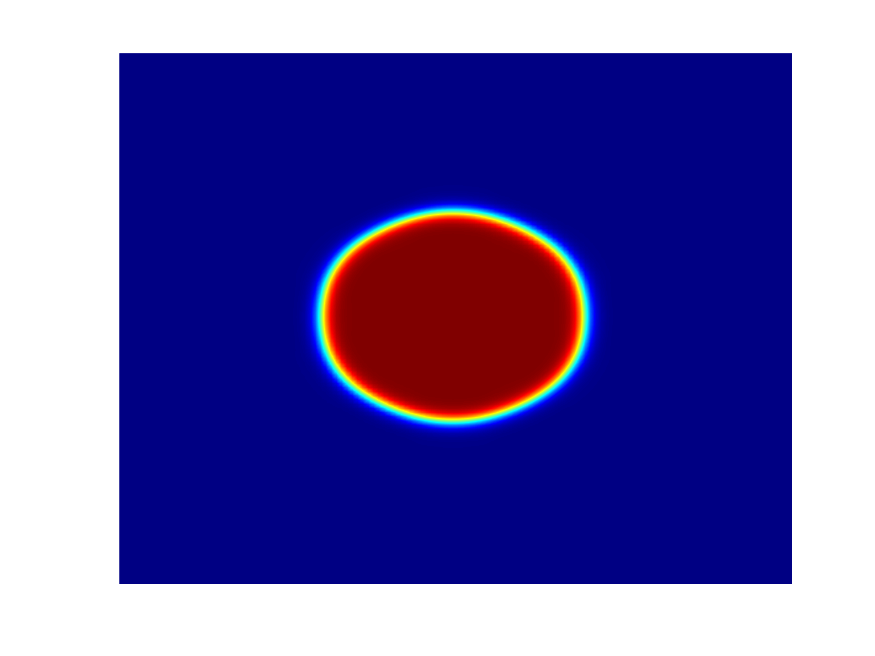}
			\includegraphics[width=1.2in]{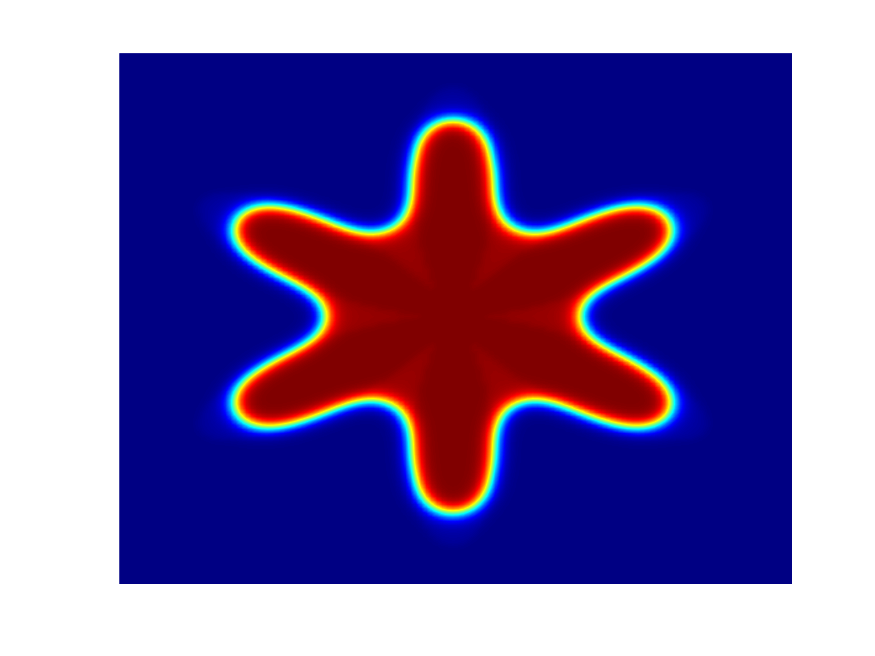}
		\end{minipage}%
	}%
	\setlength{\abovecaptionskip}{0.0cm} 
	\setlength{\belowcaptionskip}{0.0cm}
	\caption{The dynamic snapshots of the numerical solution $\phi$ obtained by the $L$1-sESAV scheme with $\alpha=0.9,0.7,0.4$ (from top to bottom, respectively)}	\label{figEx4_1}
\end{figure}
\begin{figure}[!ht]
	\vspace{-12pt}
	\centering
	\subfigure[energy]
	{
		\includegraphics[width=0.45\textwidth]{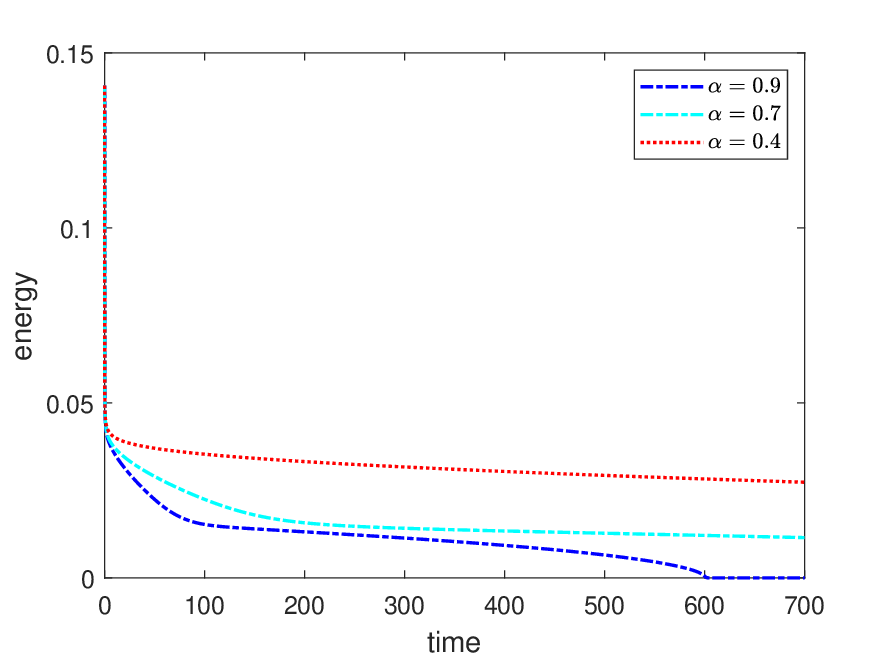}		\label{figEx4_2a}
	}%
	\subfigure[time steps]
	{
		\includegraphics[width=0.45\textwidth]{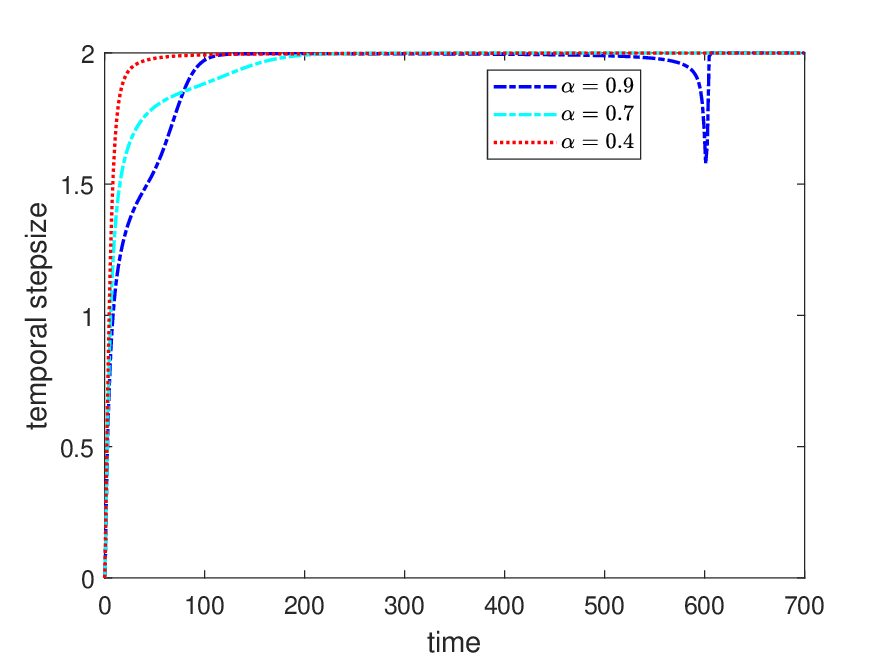}	\label{figEx4_2b}
	}
	\setlength{\abovecaptionskip}{0.0cm} 
	\setlength{\belowcaptionskip}{0.0cm}
	\caption{Time evolutions of the energy (left) and time steps (right) of the $L$1-sESAV scheme with $\alpha=0.9,0.7,0.4$}	\label{figEx4_2}
    	\vspace{-5pt}
\end{figure}

\subsection{3D bubble merging}
We finally consider the 3D tFAC model with $ \mm = 1 $, $ \varepsilon = 0.03 $, and $ \alpha = 0.5 $. 
In this example, the double-well potential \eqref{poten:dw} is considered again. 
The adaptive $L$1-sESAV scheme is employed to simulate the bubble merging with an initial condition $ \phi_{\text{init}}( x, y, z ) = \max\{ \phi_{1}, \phi_{2} \} $  such that
$$
\phi_{i} = \tanh\Bigl( \frac{ 0.2 - \sqrt{ ( x \pm 0.14 )^2 + y^2 + z^2 } }{ \varepsilon } \Bigr), ~ i=1,2.
$$

\begin{figure}[!htbp]
    \vspace{-10pt}
	\centering
	\subfigure[$t=0$]
	{
		\begin{minipage}[t]{0.24\linewidth}
			\centering
			\includegraphics[width=1.5in]{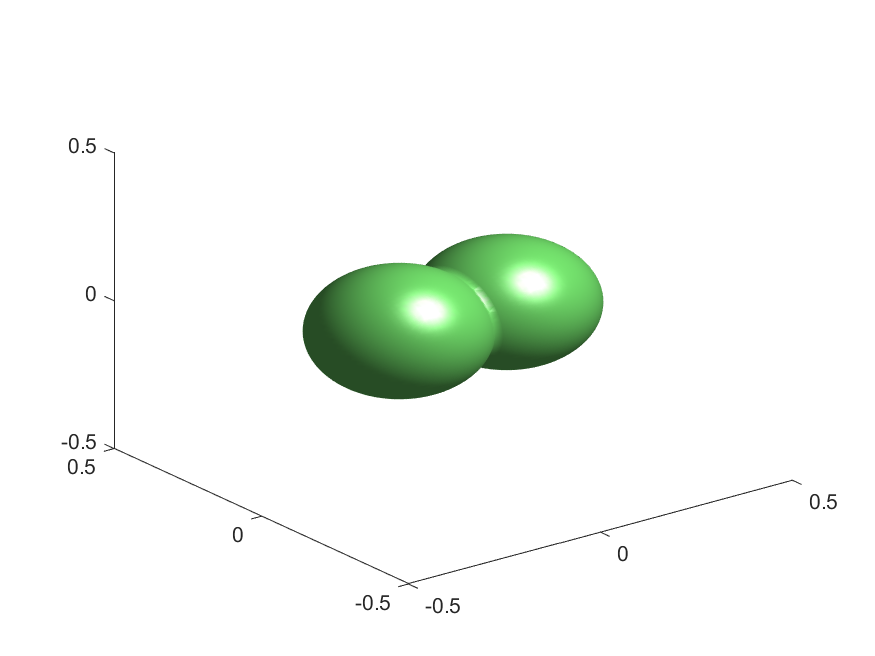}
		\end{minipage}%
	}%
	\subfigure[$t=10.71$]
	{
		\begin{minipage}[t]{0.24\linewidth}
			\centering
			\includegraphics[width=1.5in]{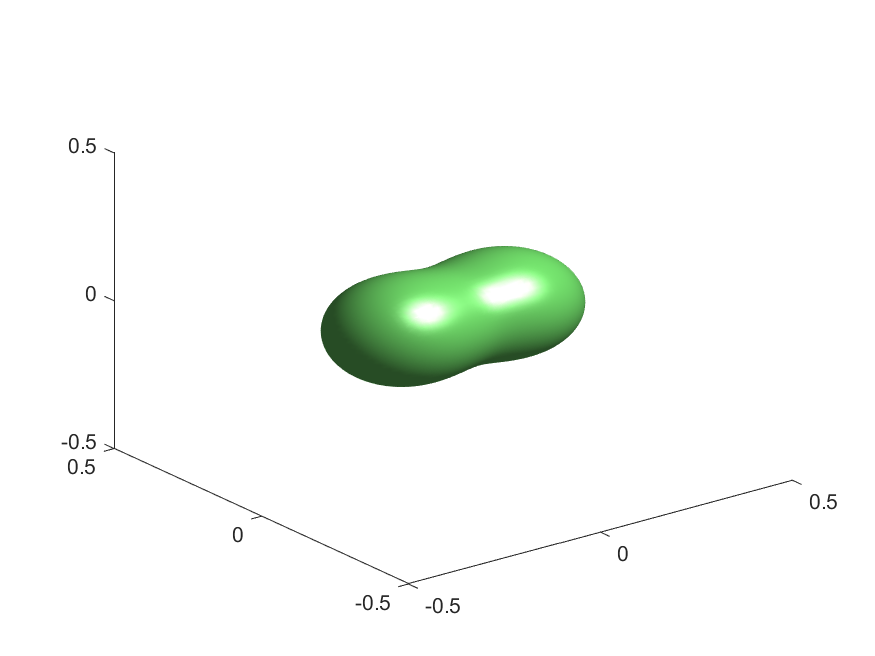}
		\end{minipage}%
	}%
	\subfigure[$t=30.80$]{
		\begin{minipage}[t]{0.24\linewidth}
			\centering
			\includegraphics[width=1.5in]{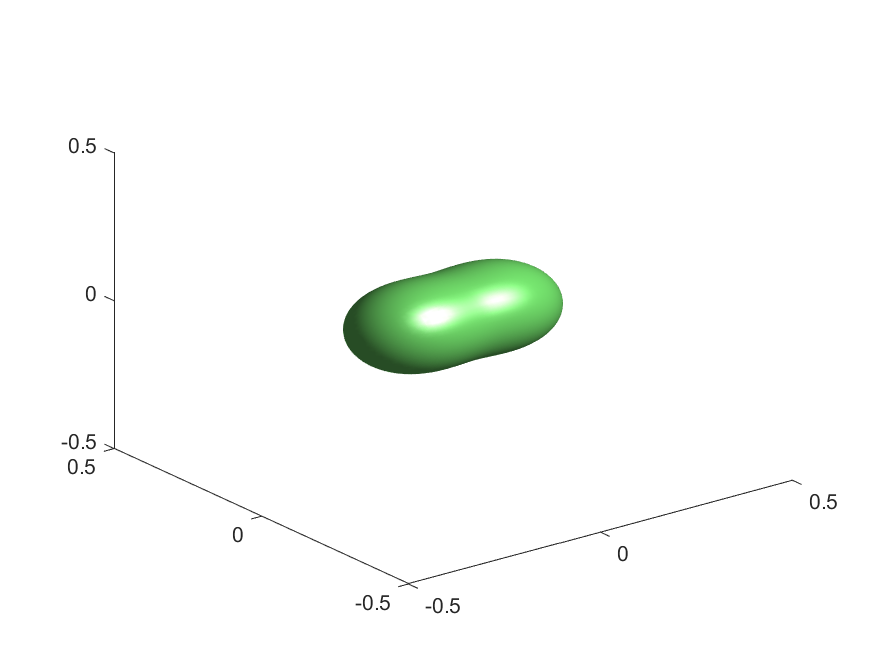}
		\end{minipage}%
	}%
	\subfigure[$t=50.48$]
	{
		\begin{minipage}[t]{0.24\linewidth}
			\centering
			\includegraphics[width=1.5in]{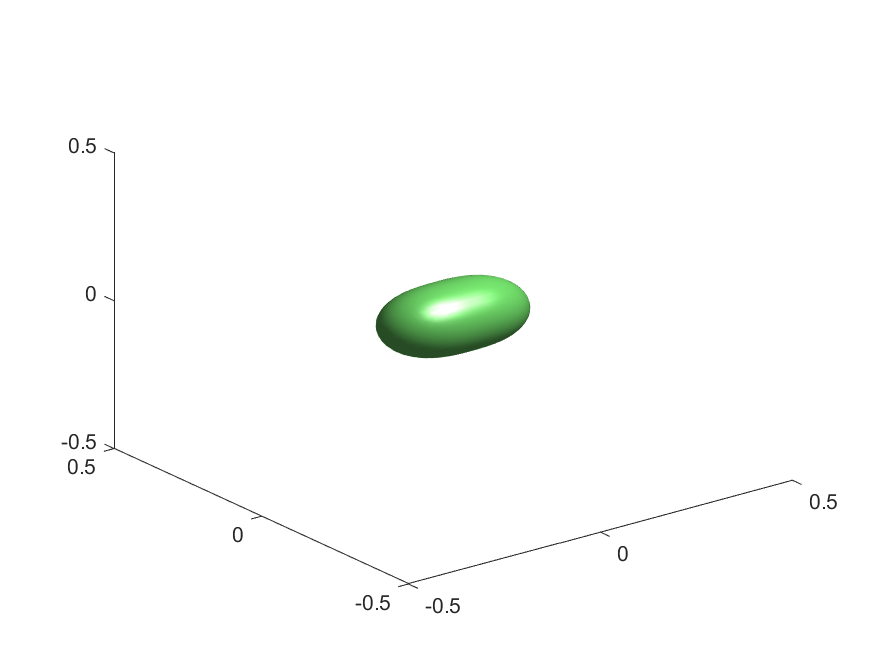}
		\end{minipage}%
	}%
	\setlength{\abovecaptionskip}{0.0cm} 
	\setlength{\belowcaptionskip}{0.0cm}
	\caption{Plots of the iso-surfaces (value $0$) of the numerical solution at different time instants obtained by the $L$1-sESAV scheme with $ \alpha = 0.5 $}	\label{figEx3D_1}
\end{figure}
In this test, we use $ 80 \times 80 \times 80 $ uniform meshes in space to discretize the computational domain $ \Omega =  ( -0.5, 0.5 )^3 $, and choose $ \tau_{\max} = 1 $, $ \tau_{\min} = 0.01 $, and $ \eta = 10^{7} $ in the adaptive time-stepping strategy \eqref{Alg1:adaptive}.
Figure \ref{figEx3D_1} shows the iso-surfaces (value $0$) of the numerical solution obtained by the adaptive $L$1-sESAV scheme at different time instants, clearly illustrating that the two balls gradually shrink and merge into one smaller ball. Moreover, the discrete energy stability and MBP are well preserved at all times, as shown in Figures \ref{figEx3D_2a}--\ref{figEx3D_2b}. Finally, Figure \ref{figEx3D_2c} displays the history curves of the adaptive time steps, demonstrating the high efficiency of the time-adaptivity technique.

\begin{figure}[!htbp]
	\vspace{-10pt}
	\centering
	\subfigure[maximum-norm of $\phi$]
	{
		\includegraphics[width=0.32\textwidth]{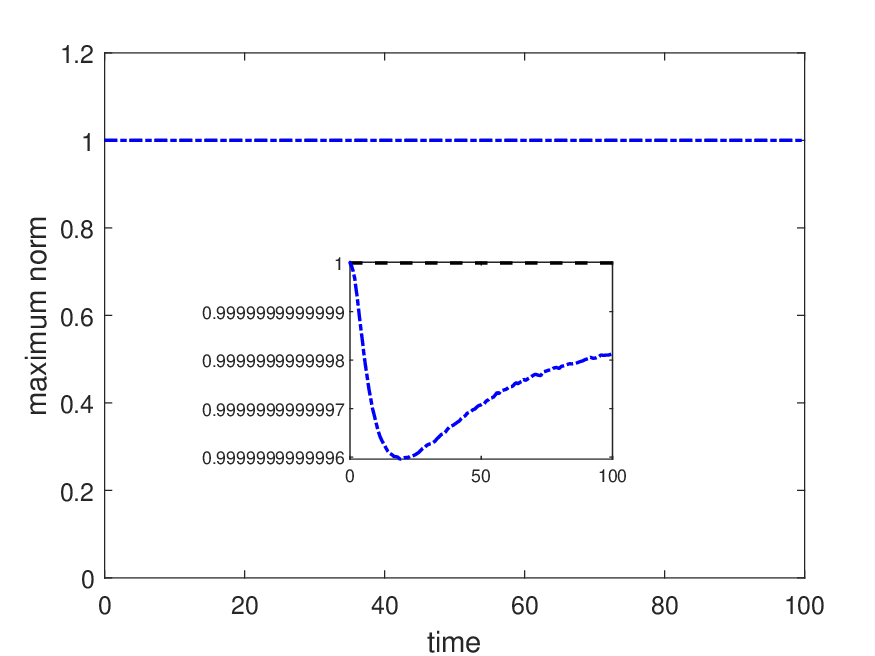}
		\label{figEx3D_2a}
	}%
	\subfigure[energy]
	{
		\includegraphics[width=0.32\textwidth]{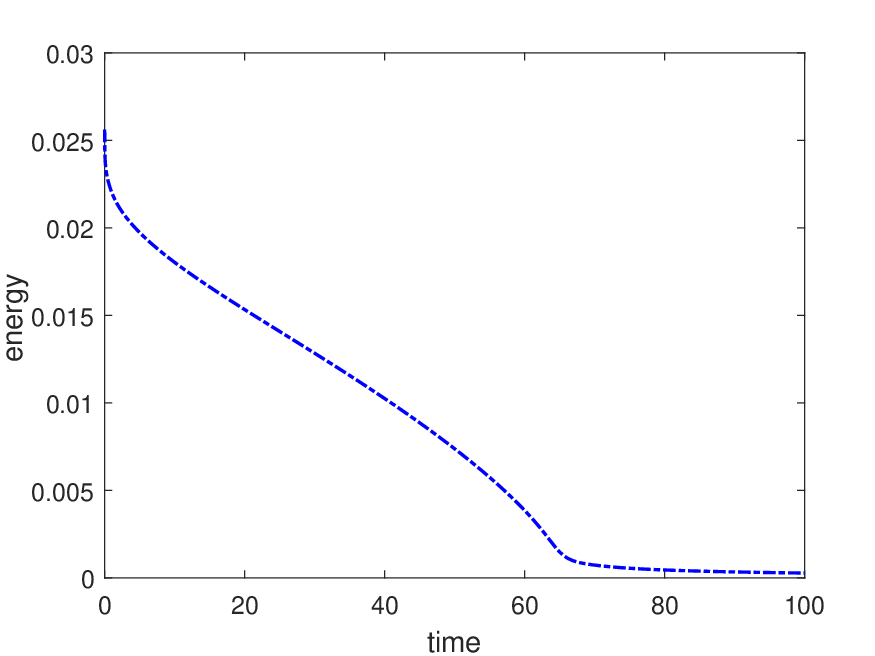}
		\label{figEx3D_2b}
	}%
	\subfigure[time steps]
	{
		\includegraphics[width=0.32\textwidth]{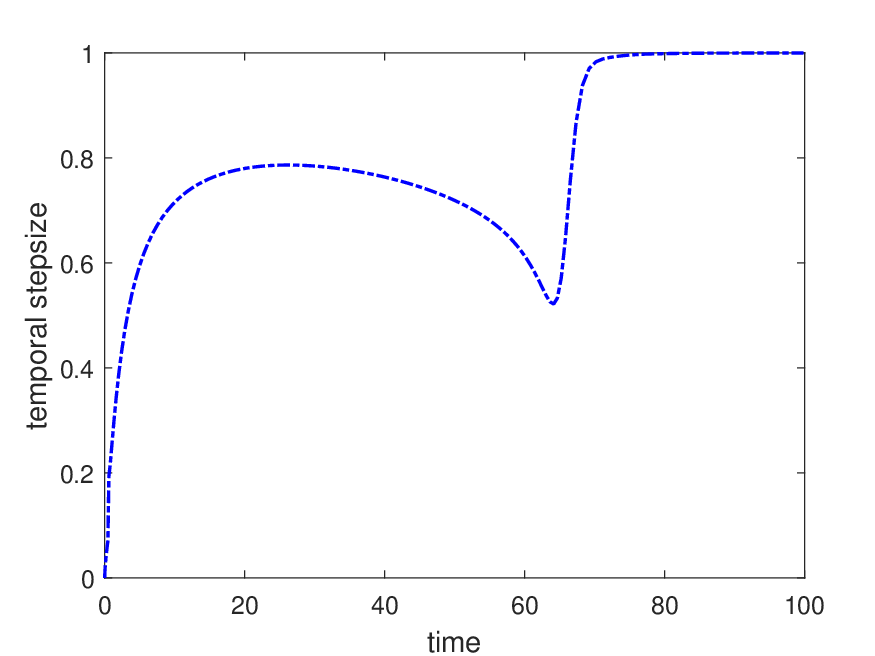}
		\label{figEx3D_2c}
	}%
	\setlength{\abovecaptionskip}{0.0cm} 
	\setlength{\belowcaptionskip}{0.0cm}
	\caption{Time evolutions of the maximum norm (left), energy (middle), and time steps (right) for the $L$1-sESAV scheme with $ \alpha = 0.5 $}
	\label{figEx3D_2}
    \vspace{-10pt}
\end{figure}

\section{ Stabilized $L2$-$1_{\sigma}$ type ESAV schemes }\label{Sec:L21}
In this section, we extend the ideas and derivations presented in Section \ref{Sec:L1} to develop a second-order in time MBP-preserving scheme using the nonuniform $L2$-$1_{\sigma}$ formula \cite{JCP_Alikhanov_2015} for the Caputo derivative. 
For a given grid function $\{ w^{k} \}$ with $ k \geq 1 $ and $ 0 < \varsigma < 1 $, we define the off-set time level $ t_{k-\varsigma} := ( 1 - \varsigma ) t_{k} + \varsigma t_{k-1} $ and the weighted operator $ w^{k-\varsigma} := ( 1 - \varsigma ) w^{k} + \varsigma w^{k-1} $. 
Following \cite{JCP_Alikhanov_2015,CiCP_Liao_2021}, we introduce the $L2$-$1_{\sigma}$ formula on a general nonuniform grid to approximate the Caputo derivative at $ t = t_{n- \varsigma}  $ with $ \varsigma = \alpha/2 $:
\begin{equation}\label{Formula:L21}
	\begin{aligned}
		{}_0^C D^{\alpha}_t w ( t_{n- \varsigma} ) \approx \widetilde{\mathbb{D}}_\tau^\alpha w^n 
		:=  \sum_{k=1}^n B_{n-k}^{(n)} \nabla_\tau w^k,
	\end{aligned}
\end{equation}
where the discrete convolution kernels $ \{B_{n-k}^{(n)}\} $ are defined as follows: $ B^{(1)}_{0} = a_{0}^{(1)} $ if $ n = 1 $, and 
\begin{equation}\label{L21:kernel}
	B^{(n)}_{n-k} := \left\{
	\begin{aligned}
		& a^{(n)}_{0} + b^{(n)}_{1}/r_{n}, \qquad \qquad \qquad \qquad \qquad \    k = n,\\
		& a^{(n)}_{n-k} + b^{(n)}_{n-k+1}/r_{k} - b^{(n)}_{n-k}, \qquad \quad   2 \leq k \leq n-1, \\
		& a^{(n)}_{n-1} - b^{(n)}_{n-1}, \qquad  \qquad \qquad \qquad \qquad  \  \   \  \ k = 1,
	\end{aligned}
	\right.
\end{equation}
for $ n \geq 2 $, with the discrete coefficients given by
\begin{equation*}
	\begin{aligned}
		a^{(n)}_{n-k} := \frac{1}{\tau_{k}} \int^{\min\{t_{k},t_{n- \varsigma}\}}_{t_{k-1}} \omega_{1-\alpha}(t_{n- \varsigma}-s) ds, \quad 1 \leq k \leq n,
	\end{aligned}
\end{equation*}
\begin{equation*}
	\begin{aligned}
		b^{(n)}_{n-k} := \frac{2}{\tau_{k}\left(\tau_{k}+\tau_{k+1}\right)} \int^{t_{k}}_{t_{k-1}} (s-t_{k-\frac{1}{2}})\omega_{1-\alpha}(t_{n- \varsigma}-s) ds, \quad 1 \leq k \leq n-1.
	\end{aligned}
\end{equation*}
It is easy to check that $a_{0}^{(n)} = \frac{1}{\tau_{n}} \int^{t_{n- \varsigma}}_{t_{n-1}} \omega_{1-\alpha}(t_{n- \varsigma}-s) ds = \frac{ (1-\varsigma)^{1-\alpha} }{ \Gamma(2-\alpha) \tau_{n}^{\alpha} } > 0 $.
Besides, it is shown in \cite{CiCP_Liao_2021,JCP_Liao_2020} that the discrete convolution kernels in \eqref{L21:kernel} satisfy the following properties.
\begin{lemma}[\cite{JCP_Liao_2020}]\label{lem:L21_DC}
	Assume that the time-step ratio $ r_k \geq 4/7 $ for $2 \leq k \leq n$ holds. Then  
	\begin{itemize}
		\item[(i)] the discrete kernels $B_{n-k}^{(n)}$ fulfill $B_0^{(n)} \leq \frac{24}{11 \tau_n} \int_{t_{n-1}}^{t_n} \omega_{1-\alpha}\left(t_n-s\right)ds $ and
		$$
		B_{n-k}^{(n)} \geq \frac{4}{11 \tau_n} \int_{t_{n-1}}^{t_n} \omega_{1-\alpha}\left(t_n-s\right) \mathrm{d} s = \frac{4}{11\Gamma(2-\alpha) \tau_{n}^{\alpha} },\quad   1 \leq k \leq n;
		$$
		
		\item[(ii)] the discrete kernels $B_{n-k}^{(n)}$ are monotone, i.e., $ B_{n-k+1}^{(n)} < B_{n-k}^{(n)} $ for $1 \leq k \leq n$;
		
		\item[(iii)] and the first kernel $B_0^{(n)}$ is appropriately larger than the second one, i.e.,
		$$
		\frac{1-2 \varsigma}{1-\varsigma} B_0^{(n)} - B_1^{(n)}>0, \quad   n \geq 2.
		$$	
	\end{itemize}
\end{lemma}

The following lemma is presented in \cite[Lemmas 2.1--2.2 \& Theorem 2.1]{JSC_Liao_2024}, which is used to prove the discrete energy stability of the scheme proposed in next subsection. 
\begin{lemma}[\cite{JSC_Liao_2024}]\label{lem:L21_DC_positive}
	 Let $ r_{*} = r_{*} ( \alpha ) $ be the unique positive root of the nonlinear equation
	\begin{equation*}\label{lemC:root}
		2\sqrt{ \frac{ 2 ( 1 - \alpha/2 ) r_{*} }{ 1 + \alpha + ( 1 - \alpha/2 ) r_{*} } + \frac{ r_{*} }{ 1 + r_{*} } } + 3 - \frac{1}{ r_{*}^{2} ( 1 + r_{*} ) } = 0 \quad \text{for} \  \alpha \in (0,1).
	\end{equation*}
	Then, it holds that $ 0.3865 \approx r_{*}(0) < r_{*}( \alpha ) < r_{*}(1) \approx 0.4037 $ for $ \alpha \in (0,1) $. Furthermore, assume that the adjacent time-step ratio $ r_k \geq r_{*} ( \alpha ) $ for $ 2 \leq k \leq n$, then it holds
$$
	2 (\nabla_\tau w^n)\, \widetilde{\mathbb{D}}_\tau^\alpha w^n 
	    \ge \mg[\nabla_\tau w^n ] - \mg[\nabla_\tau w^{n-1} ],   \quad    n \geq 1,
	$$
where the nonnegative functionals $\mg$ is defined by
	$$
	\mg[v^n ] := \sum_{j=1}^{n-1}\bigl(C_{n-j-1}^{(n)}-C_{n-j}^{(n)}\bigr) \Bigl(\sum_{\ell=j+1}^n v^{\ell}\Bigr)^2
	+   C_{n-1}^{(n)} \Bigl(\sum_{\ell=1}^n v^{\ell}\Bigr)^2, ~ n \geq 1;\quad \mg [v^0 ] = 0,
	$$
with $ C_{0}^{(1)} = 4(1-\alpha)a_{0}^{(1)}/(2-\alpha)$, $ C_{0}^{(n)} = 4(1-\alpha)a_{0}^{(n)}/(2-\alpha)  + 2 b_{1}^{(n)}/r_{n} $ for $n \ge 2$, and $ C_{n-k}^{(n)} = B_{n-k}^{(n)} $ for $ 1 \leq k \leq n-1 $.
\end{lemma}

\subsection{ Stabilized $L2$-$1_{\sigma}$-ESAV scheme } \label{subsec:L21-esav}
Following the concept of the $L$1-sESAV scheme developed in Section \ref{Sec:L1}, we apply the nonuniform $L2$-$1_{\sigma}$ formula \eqref{Formula:L21} for the approximation of the Caputo derivative in \eqref{sav_1} to derive the stabilized second-order ESAV scheme (denoted as $L2$-$1_{\sigma}$-sESAV):

\paragraph{\indent \bf Step 1}   Let $\hat{\phi}^{0} = \phi^{0}$.
For $ n = 1 $, first solve a predicted solution $ \hat{\phi}^{1} $ from the fully-implicit difference scheme:
\begin{equation}\label{sch:L21_non_3}
	\begin{aligned}
		\widetilde{ \mathbb{D} }_{\tau}^{\alpha} \hat{\phi}^{1} := B^{(1)}_{0} ( \hat{\phi}^{1} - \phi^{0} ) = \mm \bigl(  \varepsilon^2 \Delta_h \hat{\phi}^{1-\varsigma} + f(\hat{\phi}^{1-\varsigma}) \bigr),
	\end{aligned}
\end{equation}
where $ \hat{\phi}^{1-\varsigma} := ( 1-\varsigma ) \hat{\phi}^{1} + \varsigma \phi^{0} $, 
and let $ V^{1-\varsigma} := V( g_{h}(\hat{\phi}^{1-\varsigma}, R^{0}) ) $, then find $ \{ \phi^{1}, R^{1} \} \in \mathbb{V}_{h} \times \mathbb{R} $ by 
\begin{equation}\label{sch:L21_1_1}
	\widetilde{\mathbb{D}}^{\alpha}_{\tau} \phi^{1} = \mm \bigl( \varepsilon^2 \Delta_h \phi^{1-\varsigma} + V^{1-\varsigma} f(\hat{\phi}^{1-\varsigma}) - \kappa V^{1-\varsigma} ( \phi^{1-\varsigma} - \hat{\phi}^{1-\varsigma} ) \bigr),
\end{equation}
\begin{equation}\label{sch:L21_1_2}
	\begin{aligned}
		\mathbb{D}_{\tau} R^{1} = V^{1-\varsigma} \big\langle - f(\hat{\phi}^{1-\varsigma}) + \kappa( \phi^{1-\varsigma} - \hat{\phi}^{1-\varsigma} ), \mathbb{D}_{\tau} \phi^{1} \big\rangle.
	\end{aligned}
\end{equation}

\paragraph{\indent \bf Step 2} For $ n \geq 2 $,  given $ \phi^{n-1}$ and $ \phi^{n-2} $, first predict a MBP-preserving solution $ \hat{\phi}^{n} $ by \eqref{sch:L1_0} and let $ \hat{\phi}^{n-\varsigma} := ( 1-\varsigma ) \hat{\phi}^{n} + \varsigma \phi^{n-1} $, $ V^{n-\varsigma} := V( g_{h}(\hat{\phi}^{n-\varsigma}, R^{n-1}) ) $, then find $ \{ \phi^{n}, R^{n} \} \in \mathbb{V}_{h} \times \mathbb{R} $ such that
\begin{equation}\label{sch:L21_n_1}
	\begin{aligned}
		\widetilde{\mathbb{D}}^{\alpha}_{\tau} \phi^{n} = \mm  \bigl(  \varepsilon^2 \Delta_h \phi^{n-\varsigma} + V^{n-\varsigma} f(\hat{\phi}^{n-\varsigma}) - \kappa V^{n-\varsigma} ( \phi^{n-\varsigma} - \hat{\phi}^{n-\varsigma} ) \bigr),
	\end{aligned}
\end{equation}
\begin{equation}\label{sch:L21_n_2}
	\begin{aligned}
		\mathbb{D}_{\tau} R^{n} = V^{n-\varsigma} \big\langle - f(\hat{\phi}^{n-\varsigma}) + \kappa( \phi^{n-\varsigma} - \hat{\phi}^{n-\varsigma} ), \mathbb{D}_{\tau} \phi^{n} \big\rangle.
	\end{aligned}
\end{equation}

\begin{remark}
As mentioned in Remark \ref{rem:iter}, the nonlinear scheme \eqref{sch:L21_non_3} in \textbf{Step 1} of the $L2$-$1_{\sigma}$-sESAV scheme can be solved by performing the following simple iterative scheme with a given termination error $ tol $:
\begin{equation}\label{sch:L21_non_2}
	\begin{aligned}
		B^{(1)}_{0} ( \hat{\phi}^{1}_{(\mathfrak{s})} - \phi^{0} ) = \mm \bigl(  \varepsilon^2 \Delta_h \hat{\phi}^{1-\varsigma}_{(\mathfrak{s})} + f(\hat{\phi}^{1-\varsigma}_{(\mathfrak{s}-1)}) - \kappa ( \hat{\phi}^{1-\varsigma}_{(\mathfrak{s})} - \hat{\phi}^{1-\varsigma}_{(\mathfrak{s}-1)} ) \bigr) \quad \text{for} \  \mathfrak{s} \geq 1,
	\end{aligned}
\end{equation}
where $ \hat{\phi}^{1-\varsigma}_{(\ell)} = ( 1-\varsigma ) \hat{\phi}^{1}_{(\ell)} + \varsigma \phi^{0} $ with $ \ell = \mathfrak{s}-1, \mathfrak{s} $, and the initial iteration value is set as $ \hat{\phi}^{1}_{(0)} = \phi^{0} $. 
Therefore, the practical implementation of the $L2$-$1_{\sigma}$-sESAV scheme can refer to that of the $L1$-sESAV scheme described in Algorithm \ref{alg:L1}. The following Theorem \ref{lem:L21_iter} demonstrates that the nonlinear scheme \eqref{sch:L21_non_3} is MBP-preserving and uniquely solvable. Following similar mathematical induction argument as in Theorem \ref{thm:cover_L1_iter}, we can prove Theorem \ref{lem:L21_iter}, see \ref{App:C} for details.
\end{remark}

\begin{theorem}\label{lem:L21_iter}
 Assume that the condition in Lemma \ref{lem:L21_DC} holds and $\kappa \geq\left\|f^{\prime}\right\|_{C[-\beta, \beta]}$. If the first level time-step $ \tau_{1} $ is sufficiently small such that
	\begin{equation}\label{MBP:L21_inner_tau}
		\tau_{1} \leq \sqrt[\alpha]{4/(11 \varsigma \mm ( \kappa + 4 \varepsilon^{2}/h^{2} ) \Gamma( 2 - \alpha ) )},
	\end{equation}
	the iterative scheme \eqref{sch:L21_non_2} preserves the MBP for $\{ \hat{\phi}^1_{(\mathfrak{s})} \}$.
Furthermore, if $ \tau_{1} $ satisfies
	\begin{equation}\label{Cover:L21_tau}
		\tau_{1} \leq \sqrt[\alpha]{4/(11\kappa (1-\varsigma) \mm \Gamma( 2 - \alpha ) )},
	\end{equation}
	then the iterative scheme \eqref{sch:L21_non_2} converges to the unique solution $ \hat{\phi}^{1} $ of the nonlinear scheme \eqref{sch:L21_non_3} in the maximum norm satisfying $ \| \hat{\phi}^{1} \|_{\infty} \leq \beta $.
\end{theorem}

\begin{remark} The unique solvability of the $L2$-$1_{\sigma}$-sESAV scheme can be discussed similarly as Theorem \ref{thm:cover_L1_iter}, by equivalently rewriting \eqref{sch:L21_1_1}--\eqref{sch:L21_1_2} and \eqref{sch:L21_n_1}--\eqref{sch:L21_n_2} in a unified form:
\begin{equation}\label{sch:L21_3}
	\begin{aligned}
		 &\bigl[ ( B^{(n)}_{0} + \kappa ( 1 - \varsigma ) \mm V^{n-\varsigma} ) I - ( 1 - \varsigma ) \mm \varepsilon^2  \Delta_h \bigr] \phi^{n} \\
       &  = \bigl[ ( B^{(n)}_{0} - \kappa \varsigma \mm V^{n-\varsigma} ) I 
             + \varsigma \mm \varepsilon^2  \Delta_h \bigr] \phi^{n-1} + \sum^{n-1}_{k=1} B^{(n)}_{n-k} \nabla_{\tau} \phi^{k} 
             + \mm V^{n-\varsigma} \big( f(\hat{\phi}^{n-\varsigma}) + \kappa \hat{\phi}^{n-\varsigma} \big),
	\end{aligned}
\end{equation}
\begin{equation}\label{sch:L21_4}
		R^{n}  = R^{n-1} + V^{n-\varsigma} \big\langle - f(\hat{\phi}^{n-\varsigma}) + \kappa( \phi^{n-\varsigma} - \hat{\phi}^{n-\varsigma} ), \phi^{n} - \phi^{n-1} \big\rangle.
\end{equation}
Since $ B^{(n)}_{0} > 0 $, $ 0<\varsigma <1$ and assumption (\textbf{A2}), it can be verified that the coefficient matrix of \eqref{sch:L21_3} is symmetric and positive definite. In fact, it is also strictly diagonally dominant. 
\end{remark}

\subsection{ Discrete energy stability and MBP }
The following two theorems establish the discrete energy stability and MBP for the proposed $L2$-$1_{\sigma}$-sESAV scheme \eqref{sch:L21_non_3}--\eqref{sch:L21_n_2}.

\begin{theorem}\label{thm:energy_L21}
    Assume that the conditions in Lemma \ref{lem:L21_DC_positive} and Theorem \ref{lem:L21_iter} hold, the $L2$-$1_{\sigma}$-sESAV scheme \eqref{sch:L21_non_3}--\eqref{sch:L21_n_2} is unconditionally energy-stable in the sense that $\me_h[\phi^{n}, R^{n}] \leq \me_h[\phi^0, R^0]$ for $ n \geq 1 $.
\end{theorem}
\begin{proof}
Let  $ n = k $ in \eqref{sch:L21_n_1} and take the inner product with $\nabla_\tau \phi^k=\phi^{k}-\phi^{k-1}$. We then combine the resulting equation with \eqref{sch:L21_4} to derive
\begin{equation}\label{energy21_f1}
	\begin{aligned}
		R^{k} - R^{k-1} - \varepsilon^2 \big\langle \Delta_h \phi^{k-\varsigma}, \phi^{k}-\phi^{k-1} \big\rangle  = - \frac{1}{\mm} \big\langle \widetilde{\mathbb{D}}^{\alpha}_{\tau} \phi^{k}, \phi^{k}-\phi^{k-1} \big\rangle.
	\end{aligned}
\end{equation}
It is straightforward to verify that
\begin{equation}\label{identy:e1}
\left( (1-\sigma) a + \sigma b \right) ( a - b ) = a ( a - b ) - \sigma ( a - b )^2 = \frac{1}{2} a^2 - \frac{1}{2} b^2 + \frac{1 - 2\sigma}{2} ( a - b )^2,~ \forall \sigma\in \mathbb{R}.
\end{equation}
Thus, by taking $\sigma=\varsigma = \alpha/2 $ with $ \alpha \in (0,1) $ in \eqref{identy:e1}, the third left-hand side term of \eqref{energy21_f1} can be estimated below by
\begin{equation}\label{energy21_f2}
	\begin{aligned}
- \varepsilon^2 \big\langle \Delta_h \phi^{k-\varsigma}, \phi^{k}-\phi^{k-1} \big\rangle \geq \frac{\varepsilon^2}{2} \| \nabla_{h} \phi^{k} \|^2 - \frac{\varepsilon^2}{2} \| \nabla_{h} \phi^{k-1} \|^2.
	\end{aligned}
\end{equation}

Inserting \eqref{energy21_f2} into \eqref{energy21_f1}, we get
\begin{equation}\label{energy21_f3}
\begin{aligned}
	\me_h[\phi^{k}, R^{k}]-\me_h[\phi^{k-1}, R^{k-1}] \leq & - \frac{1}{\mm} \big\langle \widetilde{\mathbb{D}}^{\alpha}_{\tau} \phi^{k}, \phi^{k}-\phi^{k-1} \big\rangle, \quad k \geq 2.
\end{aligned}
\end{equation}
Similarly, it follows from \eqref{sch:L21_1_1} and \eqref{sch:L21_4} that \eqref{energy21_f3} also holds for $ k = 1 $.

Now, we sum the above inequality \eqref{energy21_f3} from $ k = 1 $ to $ n $ to derive 
\begin{equation*}
	\me_h[\phi^{n}, R^{n}]-\me_h[\phi^{0}, R^{0}] 
	\le  - \frac{1}{2\mm} \sum_{k=1}^{n} \left( \big\langle \mg[\nabla_\tau \phi^k], 1 \big\rangle - \big\langle \mg[\nabla_\tau \phi^{k-1}], 1 \big\rangle   \right) 
	 =  - \frac{1}{2\mm} \big\langle \mg[\nabla_\tau \phi^{n}], 1 \big\rangle \leq 0,
\end{equation*}
where Lemma \ref{lem:L21_DC_positive} has been used in the first inequality. The proof is completed.
\end{proof}

\begin{theorem}\label{thm:MBP_L21}
	Assume that the conditions in Theorem \ref{lem:L21_iter} hold. Furthermore, if the time-step satisfies
\begin{equation}\label{MBP:L21_tau}
	\tau_{n} \leq \sqrt[\alpha]{ 4/( 11 ( 1 - \varsigma ) \mm ( 4 \varepsilon^{2}/h^2 + \kappa K_{2} ) \Gamma( 2 - \alpha ) ) }, \quad n \geq 1,
 \end{equation}
    where $K_{2}$ is the upper bound of $ V(\cdot) $, the $L2$-$1_{\sigma}$-sESAV scheme \eqref{sch:L21_non_3}--\eqref{sch:L21_n_2} preserves the MBP for $\left\{\phi^n\right\}$, that is,
   \begin{equation*} 
	\text{if}~~\|\phi^{0} \|_{\infty} \leq \beta \quad  \Longrightarrow \quad  \|\phi^n \|_{\infty} \leq \beta, \quad \forall n \geq 1.
   \end{equation*}
\end{theorem}
\begin{proof}
For $ n = 1 $, it follows from \eqref{sch:L21_3} that
\begin{equation}\label{Formu:L21SAV_MBP_0}
	\bigl[ ( B^{(1)}_{0} + \kappa ( 1 - \varsigma )  \mm V^{1-\varsigma} ) I - ( 1 - \varsigma ) \mm  \varepsilon^{2} \Delta_{h} \bigr] \phi^{1} 
    = \mathbf{M}_{2}^{1} \phi^{0} + \mm V^{1-\varsigma} \big( \kappa \hat{\phi}^{1-\varsigma} + f( \hat{\phi}^{1-\varsigma} ) \big),
\end{equation}
with $ \mathbf{M}_{2}^{1} := ( B^{(1)}_{0} - \kappa \varsigma \mm V^{1-\varsigma} ) I + \varsigma \mm \varepsilon^{2} \Delta_{h} $. Theorem \ref{lem:L21_iter} demonstrates that  
$ \| \hat{\phi}^{1} \|_{\infty} \le \beta $ when $\tau_{1}$ satisfies \eqref{MBP:L21_inner_tau}--\eqref{Cover:L21_tau}, thereby ensuring that $ \| \hat{\phi}^{1-\varsigma} \|_{\infty} \leq \beta $. Then, the application of Lemmas \ref{lem:MBP_left}--\ref{lem:MBP_right} leads to
\begin{equation}\label{Formu:L21SAV_MBP_1}
	( B^{(1)}_{0} + \kappa ( 1 - \varsigma )  \mm V^{1-\varsigma} ) \| \phi^{1} \|_{\infty} 
     \le  \| \mathbf{M}_{2}^{1} \phi^{0}  \|_{\infty} + \kappa \mm V^{1-\varsigma} \beta.
\end{equation}
Moreover, if \eqref{MBP:L21_tau} holds for $n=1$, all elements of $\mathbf{M}_2^{1}$ are nonnegative, and thus
$$
 \|\mathbf{M}_2^{1} \|_{\infty} = B_0^{(1)} - \kappa \varsigma \mm V^{1-\varsigma},
$$
which yields $
 \|\mathbf{M}_2^{1} \phi^{0} \|_{\infty} 
   \le \|\mathbf{M}_2^{1} \|_{\infty} \|\phi^{0} \|_{\infty} 
   \le  B_0^{(1)} \beta - \kappa \varsigma \mm V^{1-\varsigma} \beta
$.
Inserting this estimate into \eqref{Formu:L21SAV_MBP_1}, we obtain
$ \| \phi^{1} \|_{\infty} \leq \beta $.

For $ 2 \leq n \leq N $, suppose that $ \|\phi^k\|_{\infty} \leq \beta $ for $ 0 \leq k \leq n-1 $.
It follows from \eqref{sch:L1_0} and the definition of $ \hat{\phi}^{n-\varsigma} $ that $ \| \hat{\phi}^{n-\varsigma} \|_{\infty} \leq \beta $. Thus, it remains to verify that $ \|\phi^n\|_{\infty} \leq \beta $. 
Noting that \eqref{Formula:L21} can be expressed as
\begin{equation*}\label{Formu:L21SAV_MBP_3}
	\begin{aligned}
		\widetilde{\mathbb{D}}^{\alpha}_{\tau} \phi^{n} 
        = B_{0}^{(n)} \phi^{n} - ( B_{0}^{(n)} - B_{1}^{(n)} ) \phi^{n-1} - \Xi^{n-2} (\phi), ~~\Xi^{n-2}(\phi) := \sum_{k=1}^{n-2} ( B_{n-k-1}^{(n)} - B_{n-k}^{(n)} ) \phi^k + B_{n-1}^{(n)} \phi^0.
	\end{aligned}
\end{equation*}
Then, equation \eqref{sch:L21_n_1} can be rewritten as
\begin{equation}\label{Formu:L21SAV_MBP_4}
	\begin{aligned}
		& \bigl[ ( B^{(n)}_{0} + \kappa ( 1 - \varsigma ) \mm V^{n-\varsigma} ) I - ( 1 - \varsigma ) \mm \varepsilon^2  \Delta_h \bigr] \phi^{n} \\
		& \quad = \mathbf{M}_{2}^{n} \phi^{n-1} + \Xi^{n-2} (\phi) + \mm V^{n-\varsigma} \bigl( f(\hat{\phi}^{n-\varsigma})+ \kappa \hat{\phi}^{n-\varsigma}  \bigr) ,
	\end{aligned}
\end{equation}
with
$
\mathbf{M}_{2}^{n} := ( B_{0}^{(n)} - B_{1}^{(n)} - \kappa \varsigma \mm V^{n-\varsigma} ) I + \varsigma \mm \varepsilon^2 \Delta_h.
$

Under the time-step condition \eqref{MBP:L21_tau} and Lemma \ref{lem:L21_DC} (i) and (iii), we see
$$
B_0^{(n)} - B_1^{(n)}  > \frac{ \varsigma }{ 1- \varsigma } B_0^{(n)} \geq \frac{ 4 \varsigma \tau_{n}^{-\alpha} }{ 11(1- \varsigma) \Gamma(2-\alpha) } \geq \kappa \varsigma \mm V^{n-\varsigma} + \frac{4 \varsigma \mm \varepsilon^2}{h^2},
$$
which further means that all elements of $\mathbf{M}_2^{n}$ are nonnegative and
$$
 \|\mathbf{M}_2^{n} \|_{\infty} = B_0^{(n)} - B_1^{(n)} - \kappa \varsigma \mm V^{n-\varsigma}.
$$
Consequently, the hypothesis $ \|\phi^{n-1}\|_{\infty} \leq \beta $ yields
\begin{equation}\label{Formu:L21SAV_MBP_6}
	 \|\mathbf{M}_2^{n} \phi^{n-1} \|_{\infty} \le \|\mathbf{M}_2^{n} \|_{\infty} \|\phi^{n-1} \|_{\infty} 
     \le ( B_0^{(n)} - B_1^{(n)} - \kappa \varsigma \mm V^{n-\varsigma} ) \beta.
\end{equation}
Moreover, the positivity and monotonicity of the discrete kernels $ \{B^{(n)}_{n-k}\} $, along with the assumption that $\|\phi^{k}\|_{\infty} \leq \beta$ for $0 \leq k \leq n-1$, lead to
\begin{equation}\label{Formu:L21SAV_MBP_7}
	\begin{aligned}
		 \| \Xi^{n-2} (\phi)  \|_{\infty} \leq \sum^{n-2}_{k=1} ( B^{(n)}_{n-k-1} - B^{(n)}_{n-k} ) \| \phi^{k} \|_{ \infty } + B^{(n)}_{n-1} \| \phi^{0} \|_{ \infty } \leq B_{1}^{(n)} \beta.
	\end{aligned}
\end{equation}

Therefore, collecting the estimates \eqref{Formu:L21SAV_MBP_6}--\eqref{Formu:L21SAV_MBP_7} and using Lemmas \ref{lem:MBP_left}--\ref{lem:MBP_right}, it follows from \eqref{Formu:L21SAV_MBP_4} that
\begin{equation}\label{Formu:L21SAV_MBP_8}
	\begin{aligned}
		 ( B^{(n)}_{0} + \kappa ( 1 - \varsigma ) \mm V^{n-\varsigma} ) \| \phi^{n} \|_{\infty} 
		& \leq \big\| \big[ ( B^{(n)}_{0} + \kappa ( 1 - \varsigma ) \mm V^{n-\varsigma} ) I - ( 1 - \varsigma ) \mm \varepsilon^2  \Delta_h \big] \phi^{n} \big\|_{\infty} \\
		& = \big\| \mathbf{M}_{2}^{n} \phi^{n-1} + \Xi^{n-2} (\phi) + \mm V^{n-\varsigma} \bigl(  f(\hat{\phi}^{n-\varsigma}) +\kappa \hat{\phi}^{n-\varsigma} \bigr) \big\|_{\infty} \\
		& \leq B_0^{(n)} \beta + \kappa ( 1 - \varsigma ) \mm V^{n-\varsigma} \beta.
	\end{aligned}
\end{equation}
This immediately implies $ \| \phi^{n} \|_{\infty} \leq \beta $ and completes the proof.
\end{proof}

\subsection{ Novel unbalanced stabilized $L2$-$1_{\sigma}$-sESAV scheme }\label{sub:unbalanced}
In cases that the auxiliary functional $ V(\cdot) $ satisfies the assumptions (\textbf{A1})--(\textbf{A3}) with $ K_{2} = 1 $,  we can further improve the $L2$-$1_{\sigma}$-sESAV scheme by introducing a novel unbalanced stabilization term 
\begin{equation}\label{def:stab_term}
	\kappa\, ( \phi^{n-\varsigma} - V^{n-\varsigma} \hat{\phi}^{n-\varsigma} ), \quad n \geq 1.
\end{equation}
The imbalance, reflected by $ V^{n-\varsigma} \neq 1 $ for $ n \geq 1 $ generally, is beneficial for the preservation of MBP especially for long-term simulations with a large time step, see Remark \ref{rem:MBP} and Figures \ref{figEx6_1}--\ref{figEx6_2} in Example \ref{subsub:unba}. Using the modified stabilization term \eqref{def:stab_term}, a novel unbalanced $L2$-$1_{\sigma}$-sESAV scheme is proposed as follows:

\paragraph{\indent \bf Step 1} Let $\hat{\phi}^{0} = \phi^{0}$.
For $ n = 1 $, first solve a predicted solution $ \hat{\phi}^{1} $ from the fully-implicit difference scheme \eqref{sch:L21_non_3}, then find $ \{ \phi^{1}, R^{1} \} \in \mathbb{V}_{h} \times \mathbb{R} $ by 
\begin{equation}\label{sch:L21_un_1_1}
	\begin{aligned}
		\widetilde{\mathbb{D}}^{\alpha}_{\tau} \phi^{1}  = \mm \bigl( \varepsilon^2 \Delta_h \phi^{1-\varsigma} + V^{1-\varsigma} f(\hat{\phi}^{1-\varsigma}) - \kappa ( \phi^{1-\varsigma} - V^{1-\varsigma} \hat{\phi}^{1-\varsigma} ) \bigr),
	\end{aligned}
\end{equation}
\begin{equation}\label{sch:L21_un_1_2}
	\begin{aligned}
		\mathbb{D}_{\tau} R^{1}  = - V^{1-\varsigma} \big\langle f(\hat{\phi}^{1-\varsigma}), \mathbb{D}_{\tau} \phi^{1} \rangle + \kappa \langle \phi^{1-\varsigma} - V^{1-\varsigma} \hat{\phi}^{1-\varsigma} , \mathbb{D}_{\tau} \phi^{1} \big\rangle.
	\end{aligned}
\end{equation}

\paragraph{\indent \bf Step 2} For $ n \geq 2 $,  given $ \phi^{n-1}$ and $ \phi^{n-2} $, first predict a MBP-preserving solution $ \hat{\phi}^{n} $ by \eqref{sch:L1_0}, then find $ \{ \phi^{n}, R^{n} \} \in \mathbb{V}_{h} \times \mathbb{R} $ such that
\begin{equation}\label{sch:L21_un_n_1}
	\begin{aligned}
		\widetilde{\mathbb{D}}^{\alpha}_{\tau} \phi^{n}  = \mm \bigl( \varepsilon^2 \Delta_h \phi^{n-\varsigma} + V^{n-\varsigma} f(\hat{\phi}^{n-\varsigma}) - \kappa ( \phi^{n-\varsigma} - V^{n-\varsigma} \hat{\phi}^{n-\varsigma} ) \bigr),
	\end{aligned}
\end{equation}
\begin{equation}\label{sch:L21_un_n_2}
	\begin{aligned}
		\mathbb{D}_{\tau} R^{n}  = - V^{n-\varsigma} \big\langle f(\hat{\phi}^{n-\varsigma}), \mathbb{D}_{\tau} \phi^{n} \big\rangle + \kappa \big\langle \phi^{n-\varsigma} - V^{n-\varsigma} \hat{\phi}^{n-\varsigma} , \mathbb{D}_{\tau} \phi^{n} \big\rangle.
	\end{aligned}
\end{equation}

Following the proofs of Theorems \ref{thm:energy_L21}--\ref{thm:MBP_L21}, together with Lemmas \ref{lem:L21_DC}--\ref{lem:L21_iter}, the preservation of discrete energy stability and MBP for the unbalanced $L2$-$1_{\sigma}$-sESAV scheme can be established in a similar manner.
\begin{theorem}\label{thm:energy_un_L21}
	 Assume that the conditions in Lemma \ref{lem:L21_DC_positive} and Theorem \ref{lem:L21_iter} hold, the unbalanced $L2$-$1_{\sigma}$-sESAV scheme \eqref{sch:L21_un_1_1}--\eqref{sch:L21_un_n_2} is unconditionally energy-stable in the sense that $\me_h[\phi^{n}, R^{n}] \leq \me_h[\phi^0, R^0]$ for $ n \geq 1 $.
\end{theorem}

\begin{theorem}\label{thm:MBP_un_L21} 
	Assume that the conditions in Theorem \ref{lem:L21_iter} hold. If the time-step satisfies
	\begin{equation}\label{MBP:un_L21_tau}
		\tau_{n} \leq \sqrt[\alpha]{ 4/( 11 ( 1 - \varsigma ) \mm ( 4 \varepsilon^{2}/h^2 + \kappa ) \Gamma( 2 - \alpha ) ) }, \quad n \geq 1,
	\end{equation}
	the unbalanced $L2$-$1_{\sigma}$-sESAV scheme \eqref{sch:L21_un_1_1}--\eqref{sch:L21_un_n_2} preserves the MBP for $\left\{\phi^n\right\}$, that is,
   \begin{equation*}\label{MBP:L1}
	\text{if}~~\|\phi^{0} \|_{\infty} \leq \beta \quad  \Longrightarrow \quad  \|\phi^n \|_{\infty} \leq \beta, \quad \forall n \geq 1.
   \end{equation*}
\end{theorem}
\begin{proof}
Together with a similar treatment to \eqref{Formu:L21SAV_MBP_0}--\eqref{Formu:L21SAV_MBP_8}, we conclude from  \eqref{sch:L21_un_1_1} and \eqref{sch:L21_un_n_1} that
\begin{equation}\label{Formu:unL21SAV_MBP_1}
	( B^{(n)}_{0} + \kappa ( 1 - \varsigma ) \mm ) \| \phi^{n} \|_{\infty}  \leq B_0^{(n)} \beta + \kappa ( 1 - \varsigma ) \mm  ) \beta + \mathcal{P}^{n}, \quad n\geq 1,
\end{equation}
under the ime-step condition \eqref{MBP:un_L21_tau},
where $ \mathcal{P}^{n} := \kappa \mm ( V^{n-\varsigma} - 1 ) \beta $. The assumption (\textbf{A2}) with $ K_{2} = 1 $ ensures that $ \mathcal{P}^{n} \leq 0 $, thereby further guarantees that the unbalanced $L2$-$1_{\sigma}$-sESAV scheme preserves the discrete MBP.
\end{proof}

\begin{remark}\label{rem:MBP} The result \eqref{Formu:unL21SAV_MBP_1} directly leads to
	\begin{equation*}
			\| \phi^n \|_{\infty} \leq \beta +  \frac{ \mathcal{P}^n }{ B^{(n)}_{0} + \kappa ( 1 - \varsigma ) \mm }.
	\end{equation*}
From assumption (\textbf{A3}), we know that $ g_h(\hat{\phi}^{n-\varsigma}, R^{n-1}) $ and $ V^{n-\varsigma} $ may deviate significantly from $ 1 $ for a relatively larger time step. In this sense, the negative term $ { \mathcal{P}^n }/{ (B^{(n)}_{0} + \kappa ( 1 - \varsigma ) \mm )} $ can be considered as a penalty term to generate solutions with smaller maximum-norm. However, no such term exists in the standard $L2$-$1_{\sigma}$-sESAV scheme, indicating that, compared to the standard one, the proposed unbalanced $L2$-$1_{\sigma}$-sESAV scheme is more effective in preserving the discrete MBP, see also the numerical results in Figures \ref{figEx6_1}--\ref{figEx6_2} for comparisons.
\end{remark}

\subsection{ Numerical experiments }
This subsection is devoted to numerical tests of both the $L2$-$1_{\sigma}$-sESAV and unbalanced $L2$-$1_{\sigma}$-sESAV schemes. Note that for the $L2$-$1_{\sigma}$ type methods, a time-step ratio constraint $ r_{k} \geq 4/7 $ is required to preserve the discrete energy stability (cf. Theorems \ref{thm:energy_L21} and \ref{thm:energy_un_L21}) and MBP (cf. Theorems \ref{thm:MBP_L21} and \ref{thm:MBP_un_L21}). Thus, in the numerical simulations, the following modified adaptive time-stepping strategy 
\begin{equation}\label{Alg2:adaptive}
		\tau_{k} = \max \Bigl\{ \max \Bigl\{ \tau_{\min} , \f{\tau_{\max}}{ \sqrt{ 1 + \eta \vert \p_{\tau} E^{k-1} \vert^{2} } } \Bigr\}, r_{\min} \tau_{k-1} \Bigr\} 
\end{equation}
is adopted, where $ r_{\min} \geq 4/7 $ is a user-defined constant. 

\subsubsection{Temporal convergence}\label{subsec:Ls_TimeCon}
First, we verify the temporal accuracy of both the balanced and unbalanced $L2$-$1_{\sigma}$-sESAV schemes, and  the numerical results are obtained by solving the exterior-forced problem \eqref{Model:tAC_source} with a fixed fractional order $ \alpha = 0.8 $ on $ 200 \times 200 $ uniform spatial grids. The exact solution is prescribed as $ \phi( \mathbf{x}, t ) = 0.5\, \omega_{1+\iota} (t) \sin x \sin y$ through an appropriate choice of the linear term $ g( \mathbf{x}, t ) $. For different regularity parameters $ \iota = 0.3, 0.5 $ and $ 0.8 $, the mesh grading parameter is set to be $ \gamma = 2/\iota $ such that the theoretical convergence order is $ \min\{ 2, \gamma \iota \}= 2$. The other setup is the same as in Example \ref{subsec:L1_TimeCon}. Numerical results for different $ \iota$ and potentials are summarized in Tables \ref{Ex5:tab2}--\ref{Ex5:tab3}, from which we can clearly observe that (i) both methods exhibit second-order temporal accuracy as expected; and (ii) they generate numerical solutions with the same magnitude accuracy. 
\begin{table}[!ht]
	\vspace{-12pt}
	\caption{Time accuracy of the two $L2$-$1_{\sigma}$-sESAV schemes with $ \alpha = 0.8 $ and $ \gamma = 2/\iota $: the double-well potential} \label{Ex5:tab2}
	{\footnotesize\begin{tabular*}{\columnwidth}{@{\extracolsep\fill}cccccccc@{\extracolsep\fill}}
			\toprule
			\multicolumn{1}{c}{\multirow{2}{*}{method}} & \multicolumn{1}{c}{\multirow{2}{*}{$ N $}} & \multicolumn{2}{c}{$\iota=0.3$} & \multicolumn{2}{c}{$\iota=0.5$} & \multicolumn{2}{c}{$\iota=0.8$} \\
			\cmidrule{3-4} 
			\cmidrule{5-6}
			\cmidrule{7-8}
			&      & $ e(N) $ & Order & $ e(N) $ & Order & $ e(N) $ & Order \\
			\midrule
            &  20     & $9.74 \times 10^{-3}$   &  ---   &  $2.67 \times 10^{-3}$ & ---   & $3.35 \times 10^{-4}$   &  ---   \\
			  \multicolumn{1}{c}{\multirow{1}{*}{$L2$-$1_{\sigma}$-sESAV}} 
            &  40     & $2.49 \times 10^{-3}$   &  1.97  &  $6.88 \times 10^{-4}$ & 1.96   & $8.46 \times 10^{-5}$   & 1.98   \\ 
			\multicolumn{1}{c}{\multirow{1}{*}{scheme}} 
            &  80     & $6.25 \times 10^{-4}$   &  2.00   &  $1.73 \times 10^{-4}$ & 1.99   & $2.21 \times 10^{-5}$   &  1.94   \\ 
			&  160    & $1.59 \times 10^{-4}$   &  1.98   &  $4.33 \times 10^{-5}$ & 2.00   & $5.54 \times 10^{-6}$   & 1.99   \\
			\midrule 
            &  20     & $9.74 \times 10^{-3}$   &  ---   &  $2.67 \times 10^{-3}$ & ---      & $3.54 \times 10^{-4}$   &  ---   \\
			\multicolumn{1}{c}{\multirow{1}{*}{unbalanced $L2$-$1_{\sigma}$}}  
            &  40     & $2.49 \times 10^{-3}$   & 1.97   &  $6.88 \times 10^{-4}$ & 1.96      & $8.46 \times 10^{-5}$   &  2.06   \\
			\multicolumn{1}{c}{\multirow{1}{*}{sESAV scheme}} 
            &  80     & $6.25 \times 10^{-4}$   &  2.00   &  $1.73 \times 10^{-4}$ & 1.99      & $2.21 \times 10^{-5}$   &  1.94  \\
			&  160    & $1.59 \times 10^{-4}$   &  1.98   &  $4.33 \times 10^{-5}$ & 2.00     & $5.54 \times 10^{-6}$   &  1.99   \\
			\midrule
			\multicolumn{2}{c}{$ \min\{ 2, \gamma \iota \} $ }   
			&       & 2.00       &    & 2.00 &   & 2.00 \\
			\bottomrule
	\end{tabular*}}
\end{table}
\begin{table}[!ht]
	\vspace{-10pt}
	\caption{Time accuracy of the two $L2$-$1_{\sigma}$-sESAV schemes with $ \alpha = 0.8 $ and $ \gamma = 2/\iota $: the Flory--Huggins potential} \label{Ex5:tab3}
	{\footnotesize\begin{tabular*}{\columnwidth}{@{\extracolsep\fill}cccccccc@{\extracolsep\fill}}
			\toprule
			\multicolumn{1}{c}{\multirow{2}{*}{method}} & \multicolumn{1}{c}{\multirow{2}{*}{$ N $}} & \multicolumn{2}{c}{$\iota=0.3$} & \multicolumn{2}{c}{$\iota=0.5$} & \multicolumn{2}{c}{$\iota=0.8$} \\
			\cmidrule{3-4} 
			\cmidrule{5-6}
			\cmidrule{7-8}
			&      & $ e(N) $ & Order & $ e(N) $ & Order & $ e(N) $ & Order \\
			\midrule
            &  20     & $9.59 \times 10^{-3}$   &  ---   &  $2.59 \times 10^{-3}$ & ---   & $2.95 \times 10^{-4}$   &  ---   \\
			\multicolumn{1}{c}{\multirow{1}{*}{$L2$-$1_{\sigma}$-sESAV}} 
            &  40     & $2.49 \times 10^{-3}$   &  1.95   &  $6.82 \times 10^{-4}$ & 1.93   & $8.26 \times 10^{-5}$   &  1.84   \\ 
			\multicolumn{1}{c}{\multirow{1}{*}{scheme}} 
            &  80     & $6.25 \times 10^{-4}$   &  1.99   &  $1.73 \times 10^{-4}$ & 1.98   & $2.18 \times 10^{-5}$   &  1.92   \\ 
			&  160    & $1.59 \times 10^{-4}$   &  1.98   &  $4.33 \times 10^{-5}$ & 2.00   & $5.53 \times 10^{-6}$   &  1.98   \\
			\midrule 
            &  20     & $9.64 \times 10^{-3}$   &  ---   &  $2.59 \times 10^{-3}$ & ---      & $2.95 \times 10^{-4}$   &  ---   \\
			\multicolumn{1}{c}{\multirow{1}{*}{unbalanced $L2$-$1_{\sigma}$}} 
            &  40     & $2.49 \times 10^{-3}$   &  1.95   &  $6.82 \times 10^{-4}$ & 1.93      & $8.26 \times 10^{-5}$   &  1.84   \\
			\multicolumn{1}{c}{\multirow{1}{*}{sESAV scheme}} 
            &  80     & $6.25 \times 10^{-4}$   &  1.99   &  $1.73 \times 10^{-4}$ & 1.98      & $2.18 \times 10^{-5}$   &  1.92   \\
			&  160    & $1.59 \times 10^{-4}$   &  1.98   &  $4.33 \times 10^{-5}$ & 2.00     & $5.53 \times 10^{-6}$   &  1.98   \\
			\midrule
			\multicolumn{2}{c}{$ \min\{ 2, \gamma \iota \} $ }   
			&       & 2.00       &    & 2.00 &   & 2.00 \\
			\bottomrule
	\end{tabular*}}
\end{table}

\subsubsection{Effectiveness of the unbalanced stabilization term}\label{subsub:unba}
The coarsening dynamics of the tFAC model \eqref{Model:tAC} is examined with $ \alpha = 0.9$, $\mm = 1 $, and $ \varepsilon = 0.01 $. We adopt a $ 128\times 128 $ uniform spatial mesh to discretize the spatial domain $ \Omega = (0,1)^{2} $, and give the initial value by random numbers uniformly distributed between $ -0.8 $ and $ 0.8 $ at each grid point, see Figure \ref{figEx2_0}.

For special choice $V(\cdot)$ in \eqref{f_cut}, it deduces from Theorems \ref{thm:MBP_L21} and \ref{thm:MBP_un_L21} that $ \tau \leq 0.1044 $ for the double-well potential and $ \tau \leq 0.0646 $ for the Flory--Huggins potential are required to preserve the discrete MBP on uniform temporal grids. Note that these conditions are only sufficient but not necessary. 
Therefore, in the following tests, three different types of temporal girds are employed for each potential, two of which violate the time-step restriction while the other satisfies it. Specifically, we take $ \tau = 1, 0.5, 0.1 $
for the double-well potential and $ \tau = 1/3, 1/6, 1/30 $ for the Flory–Huggins potential. The coarsening dynamics are tested using the two $L2$-$1_{\sigma}$-sESAV schemes with different stabilization terms. Numerical results for the double-well and Flory–Huggins potentials are depicted in Figures \ref{figEx6_1} and \ref{figEx6_2}, respectively. From these results, we observe that: (i) when the time steps satisfy the restriction \eqref{MBP:un_L21_tau} (or \eqref{MBP:L21_tau} with $K_2=1$), both schemes preserve the discrete MBP, see Figures \ref{figEx6_1c} and \ref{figEx6_2c}; (ii) when the time steps moderately violate the constraint, both schemes can still preserve the discrete MBP, indicating that the conditions in Theorems \ref{thm:MBP_L21} and \ref{thm:MBP_un_L21} are only sufficient but not necessary, see Figures \ref{figEx6_1b} and \ref{figEx6_2b}; (iii) however, when the given largest time steps are used for both schemes, the unbalanced $L2$-$1_{\sigma}$-sESAV scheme can still preserve the discrete MBP, whereas the standard $L2$-$1_{\sigma}$-sESAV scheme yields solutions exceeding the corresponding maximum-bound $\beta$, see Figures \ref{figEx6_1a} and \ref{figEx6_2a}. This demonstrates the effectiveness and excellence of the proposed unbalanced stabilization term in preserving the discrete MBP.  
\begin{figure}[!htbp]
	\vspace{-12pt}
	\centering
	\subfigure[$\tau = 1$]
	{
		\includegraphics[width=0.32\textwidth]{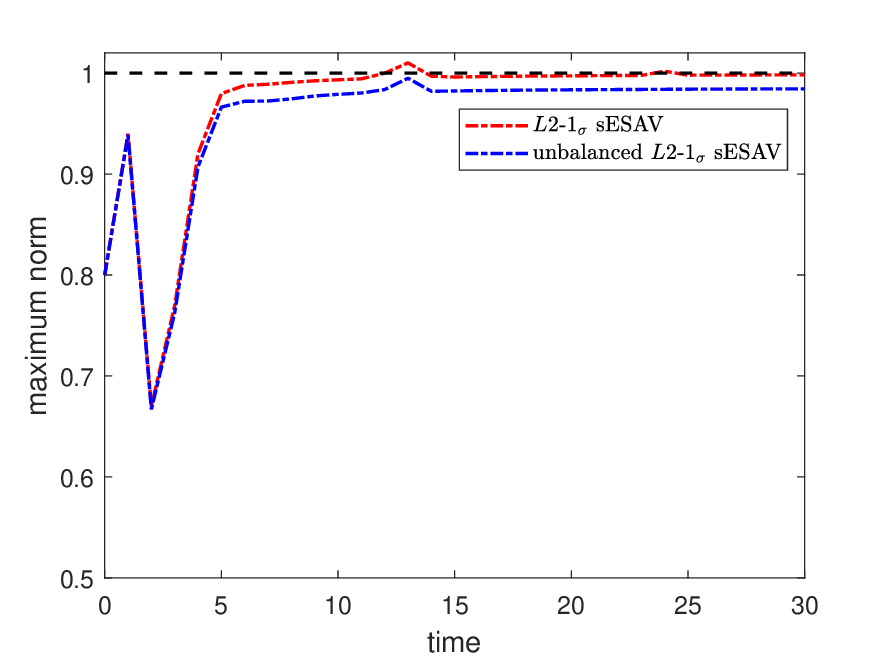}		\label{figEx6_1a}
	}%
	\subfigure[$\tau = 0.5$]
	{
		\includegraphics[width=0.32\textwidth]{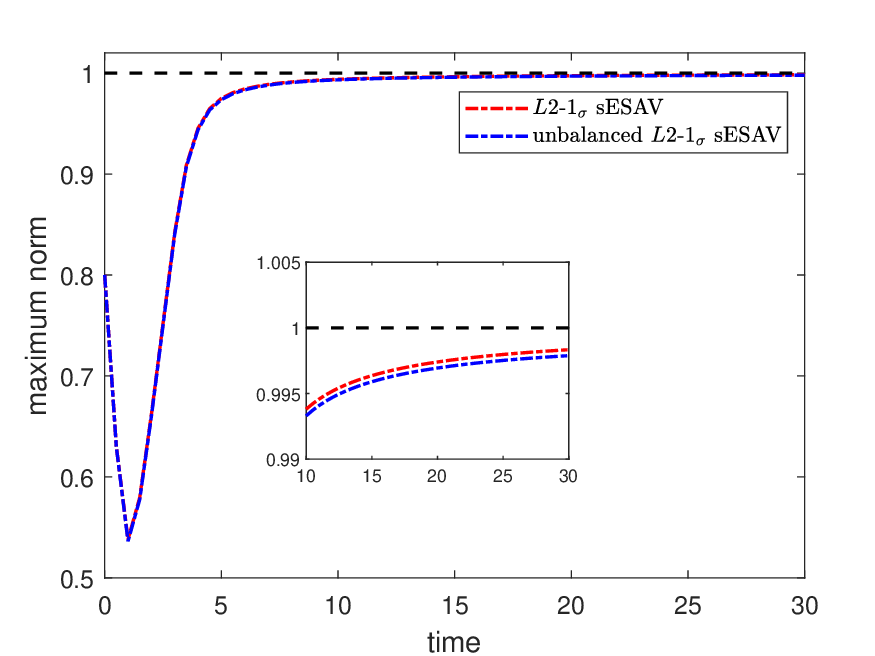}		\label{figEx6_1b}
	}%
	\subfigure[$\tau = 0.1$]
	{
		\includegraphics[width=0.32\textwidth]{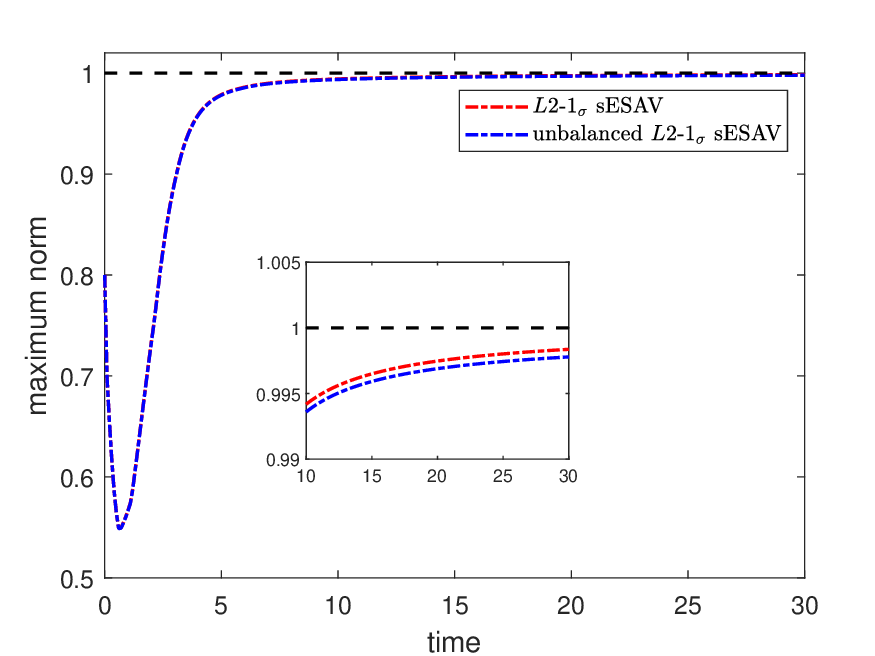}	\label{figEx6_1c}
	}%
	\setlength{\abovecaptionskip}{0.0cm} 
	\setlength{\belowcaptionskip}{0.0cm}
	\caption{Time evolutions of the maximum norm and energy of simulated solutions computed by the $L2$-$1_{\sigma}$-sESAV scheme with different stabilization terms: the double-well potential}	
	\label{figEx6_1}
\end{figure}
\begin{figure}[!ht]
	\vspace{-12pt}
	\centering
	\subfigure[$\tau = 1/3$]
	{
		\includegraphics[width=0.32\textwidth]{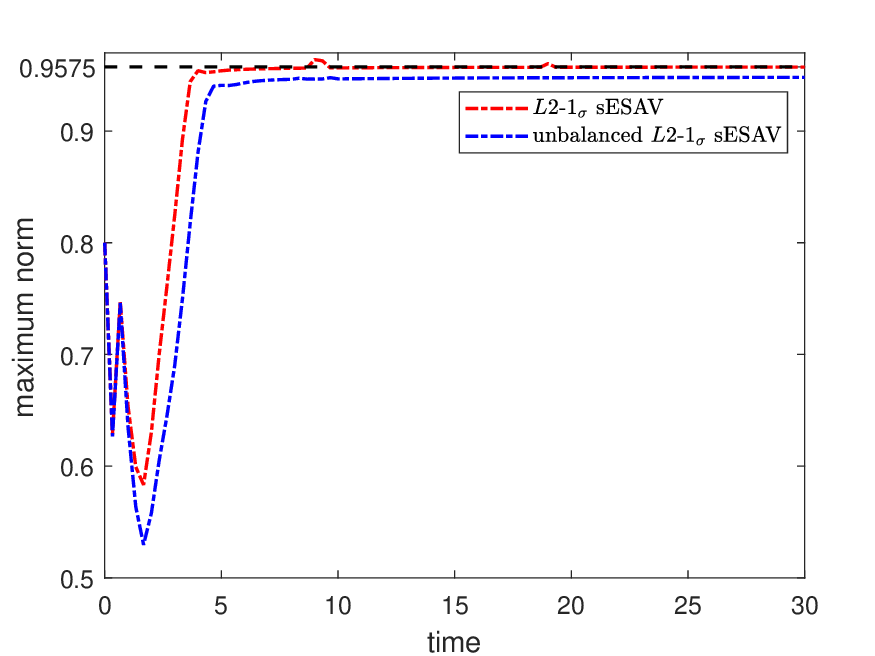}		\label{figEx6_2a}
	}%
	\subfigure[$\tau = 1/6$]
	{
		\includegraphics[width=0.32\textwidth]{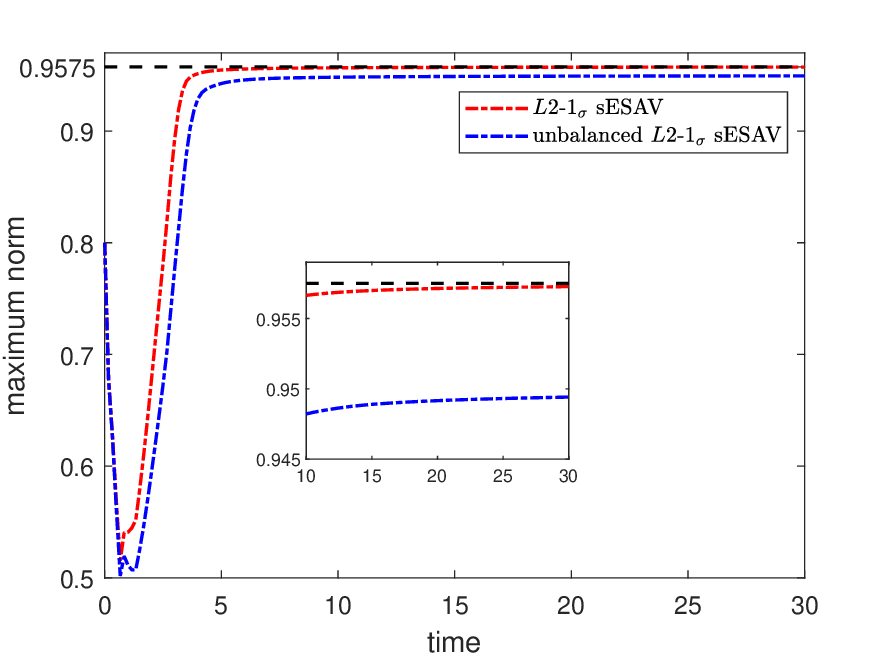}		\label{figEx6_2b}
	}%
	\subfigure[$\tau = 1/30$]
	{
		\includegraphics[width=0.32\textwidth]{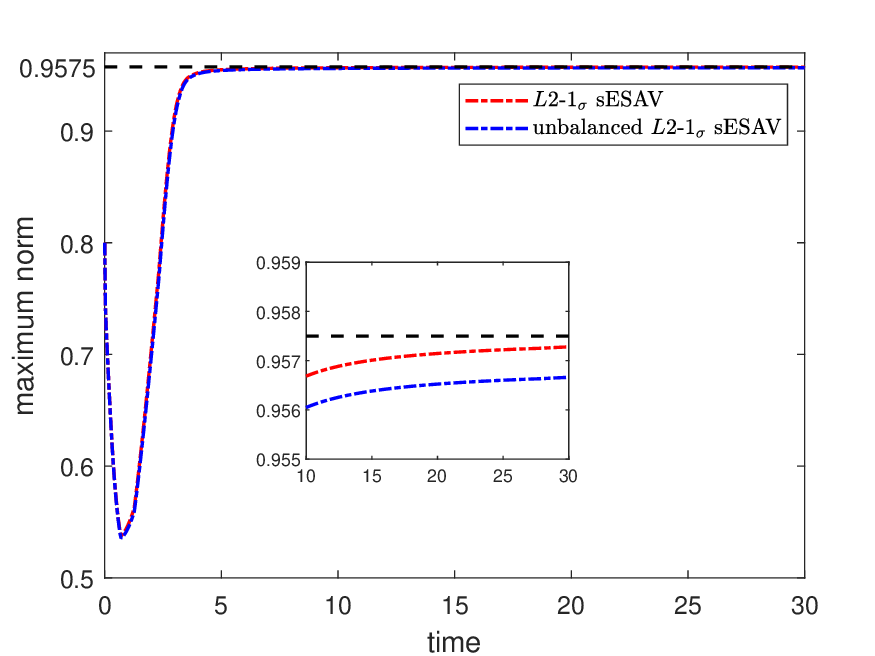}	\label{figEx6_2c}
	}%
	\setlength{\abovecaptionskip}{0.0cm} 
	\setlength{\belowcaptionskip}{0.0cm}
	\caption{Time evolutions of the maximum norm and energy of simulated solutions computed by the $L2$-$1_{\sigma}$-sESAV scheme with different stabilization terms: the Flory--Huggins potential}		\label{figEx6_2}
    \vspace{-10pt}
\end{figure}

\subsubsection{Long-term coarsening dynamics simulations}
The unbalanced $L2$-$1_{\sigma}$-sESAV scheme has demonstrated superior reliability in preserving the MBP compared to the standard scheme, making it the preferred choice for simulating long-term coarsening dynamics governed by the tFAC model \eqref{Model:tAC}, particularly for extended time scales up to $T = 2000$ in the following test. Moreover, to improve the computational efficiency without sacrificing accuracy in the simulation, a mixed adaptive time-stepping strategy \eqref{mesh:graded} and \eqref{Alg2:adaptive} with $ \hat{T} = 0.5 $ and $ \hat{N} = 30 $ is employed. Here we set $\mm= 1$, $\varepsilon = 0.01$, $ r_{\min} = 4/7 $, $ \eta = 10^{7} $, and $ ( \tau_{\min}, \tau_{\max} ) = ( 0.1, 1 ) $ for the double-well potential and $ ( \tau_{\min}, \tau_{\max} ) = ( 1/30, 1/3 ) $ for the Flory–Huggins potential. The random initial values and spatial partition are the same as aforementioned.

\begin{figure}[!htb]
	\vspace{-10pt}
	\centering
	\subfigure[$t=20$]
	{
		\begin{minipage}[t]{0.24\linewidth}
			\centering
			\includegraphics[width=1.5in]{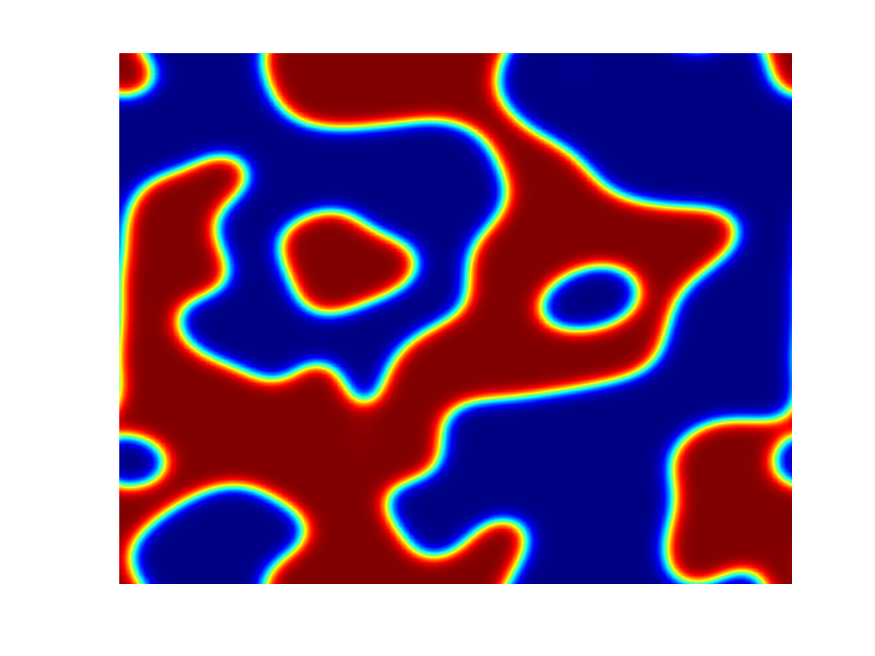}
			\includegraphics[width=1.5in]{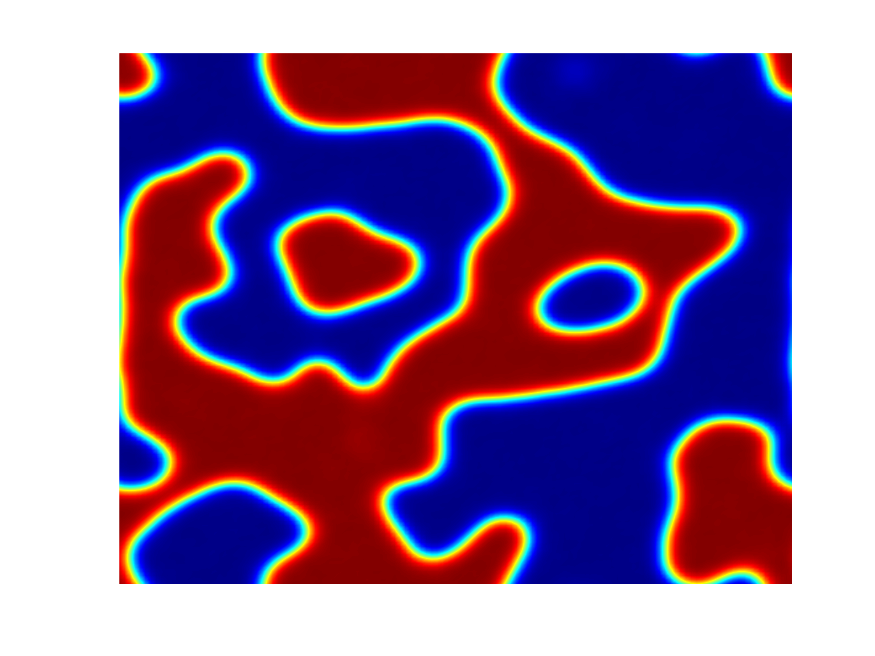}
			\includegraphics[width=1.5in]{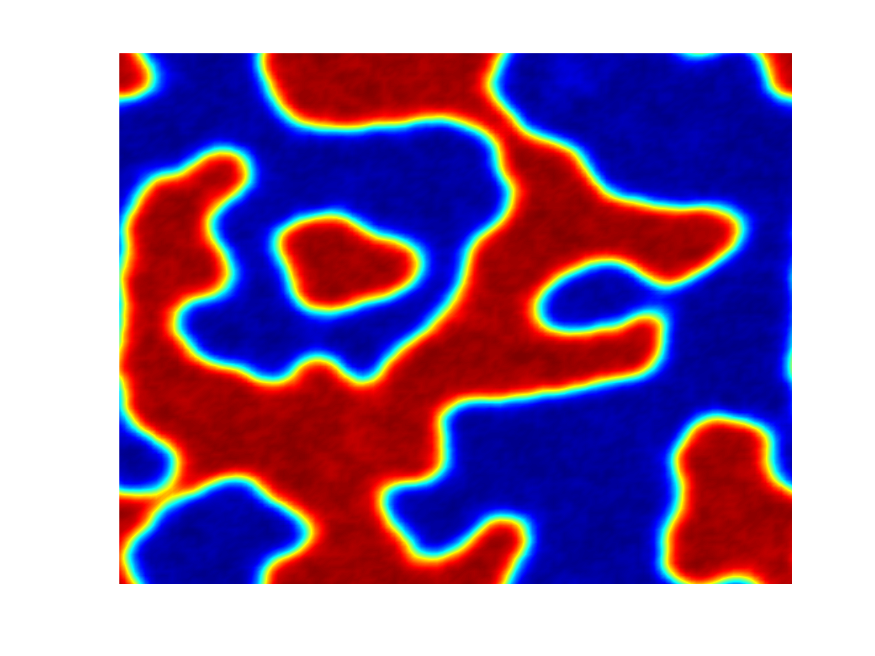}
		\end{minipage}%
	}%
	\subfigure[$t=100$]
	{
		\begin{minipage}[t]{0.24\linewidth}
			\centering
			\includegraphics[width=1.5in]{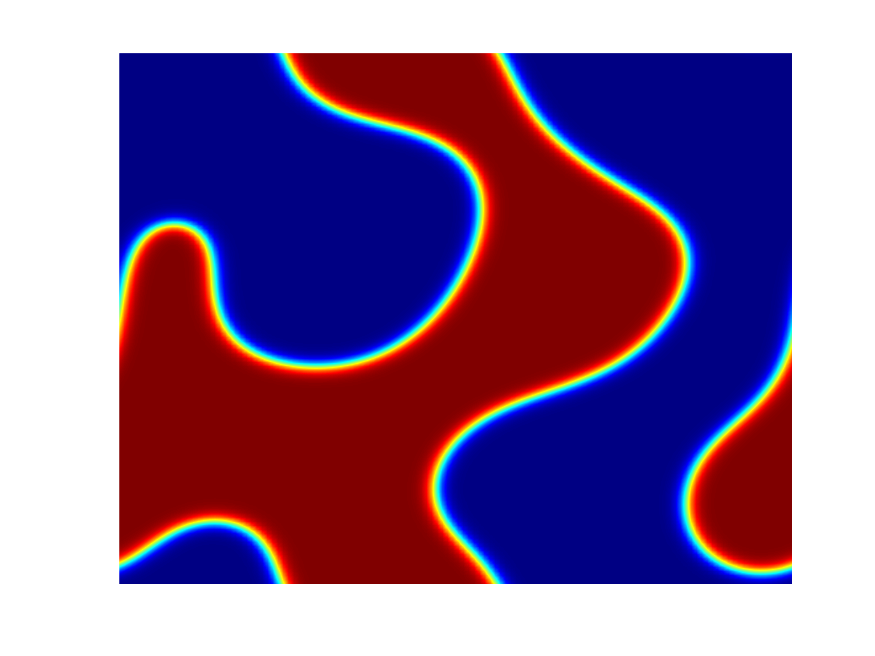}
			\includegraphics[width=1.5in]{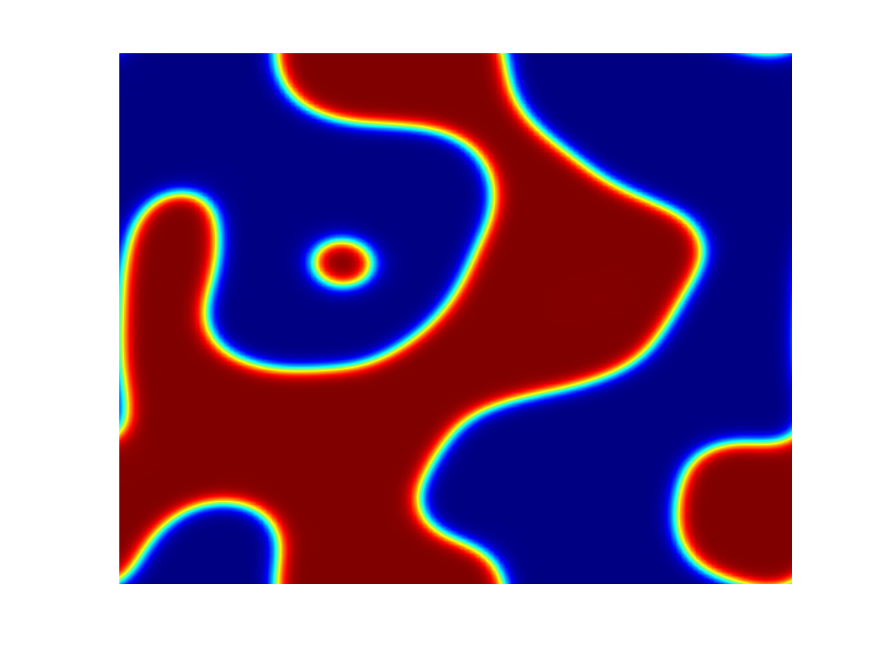}
			\includegraphics[width=1.5in]{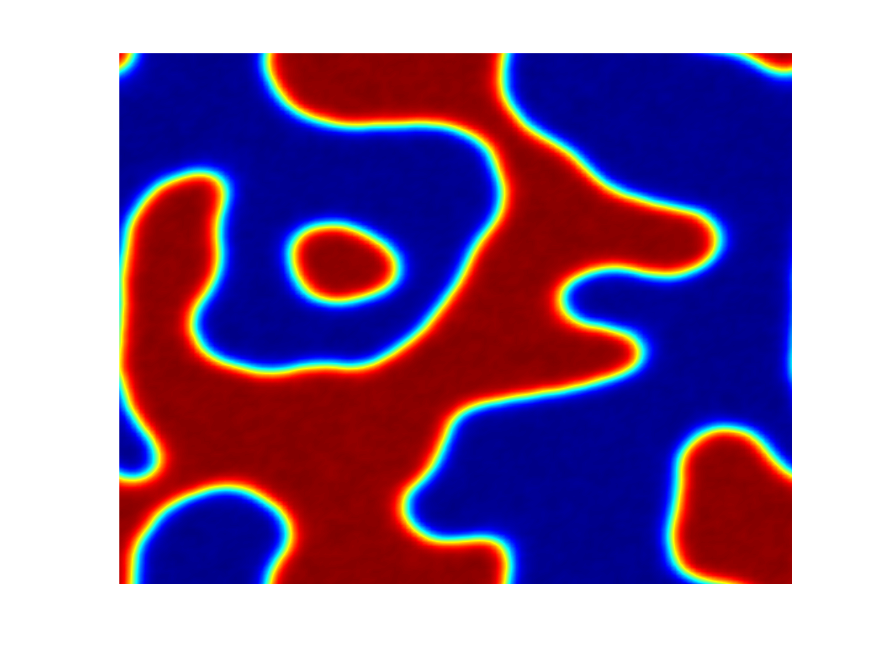}
		\end{minipage}%
	}%
	\subfigure[$t=500$]{
		\begin{minipage}[t]{0.24\linewidth}
			\centering
			\includegraphics[width=1.5in]{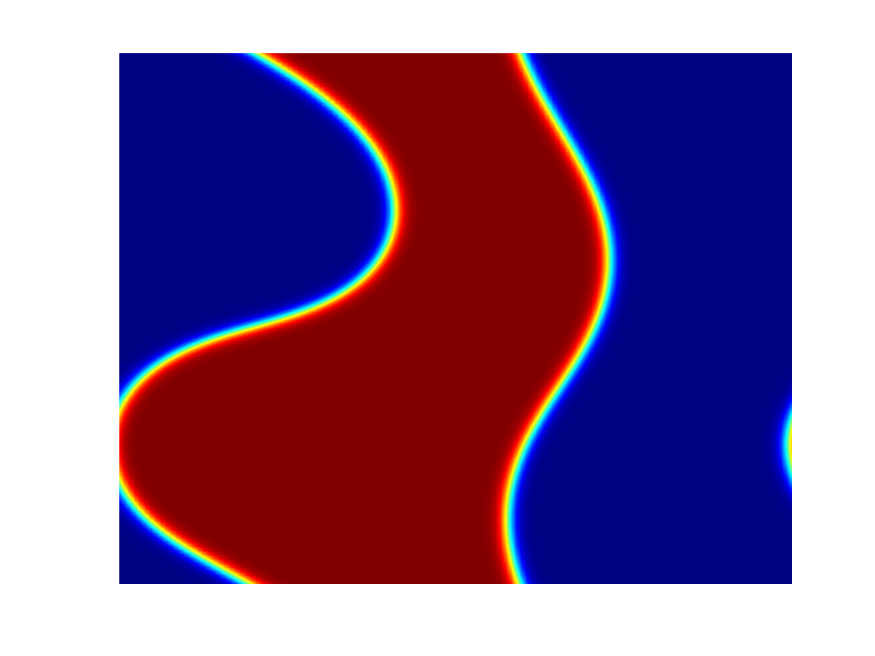}
			\includegraphics[width=1.5in]{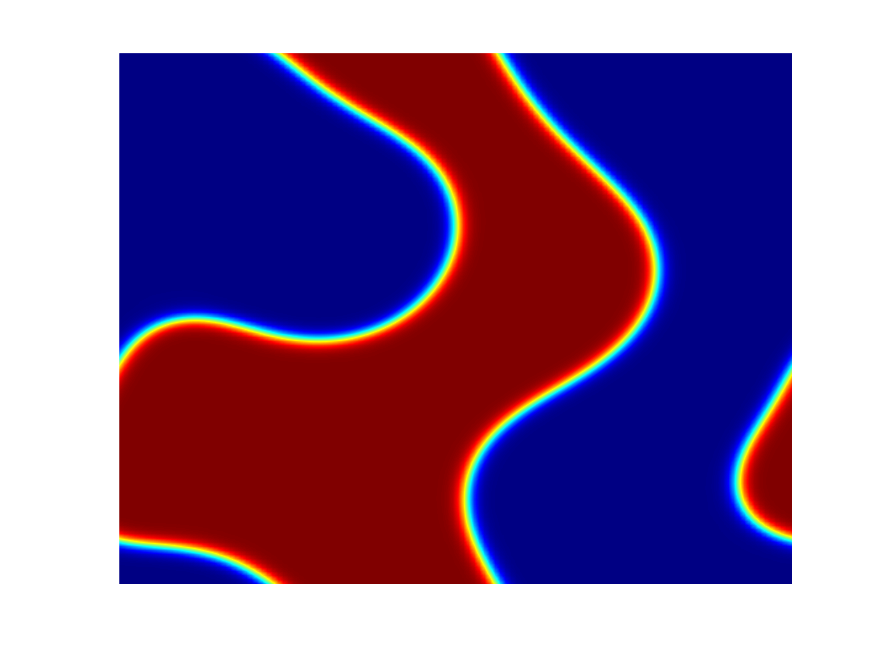}
			\includegraphics[width=1.5in]{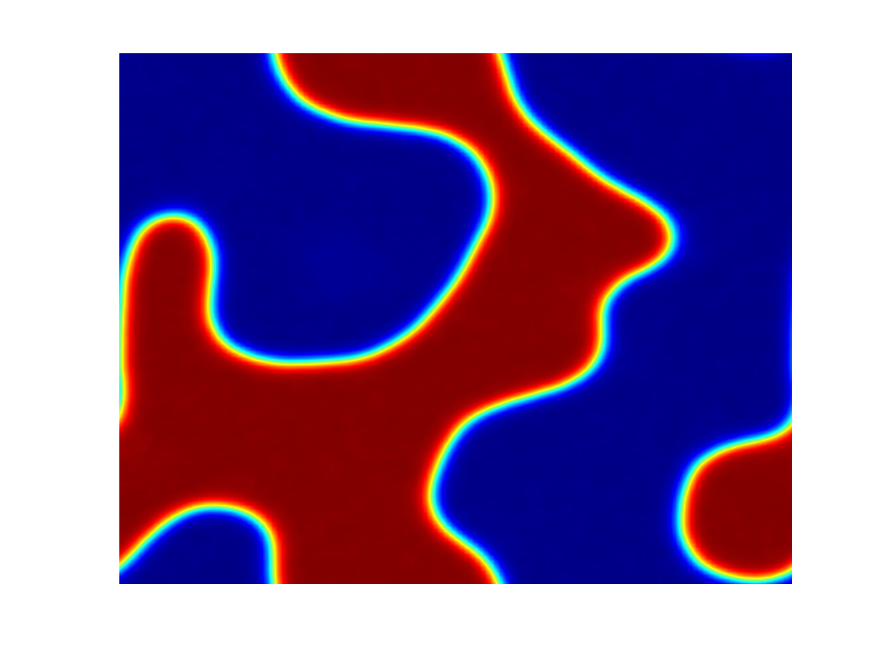}
		\end{minipage}%
	}%
	\subfigure[$t=2000$]
	{
		\begin{minipage}[t]{0.24\linewidth}
			\centering
			\includegraphics[width=1.5in]{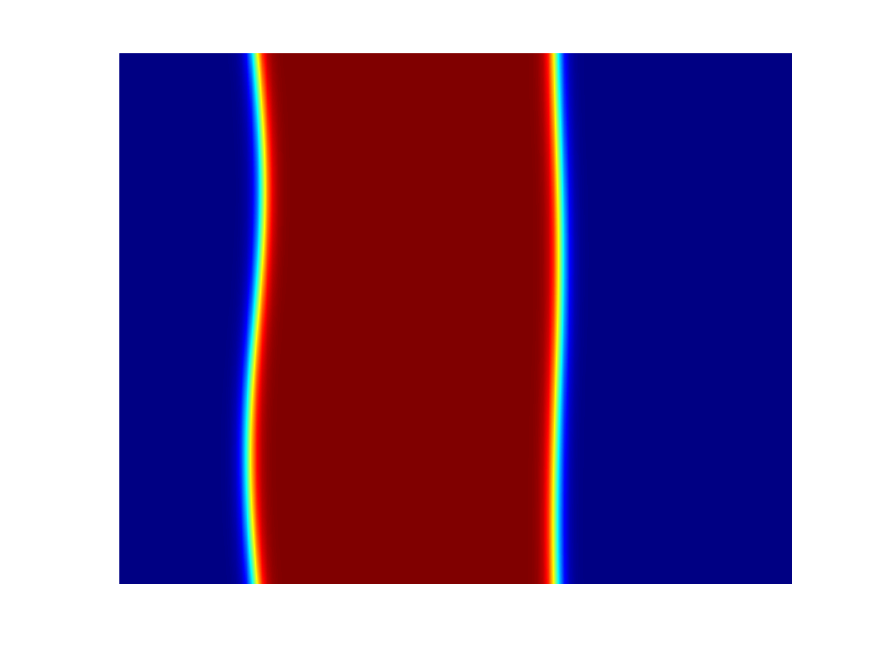}
			\includegraphics[width=1.5in]{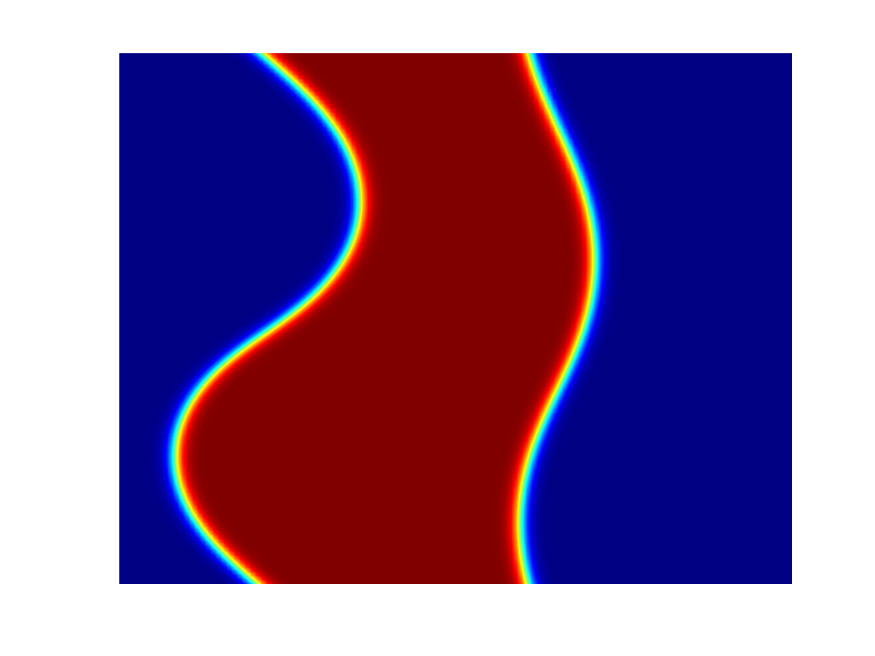}
			\includegraphics[width=1.5in]{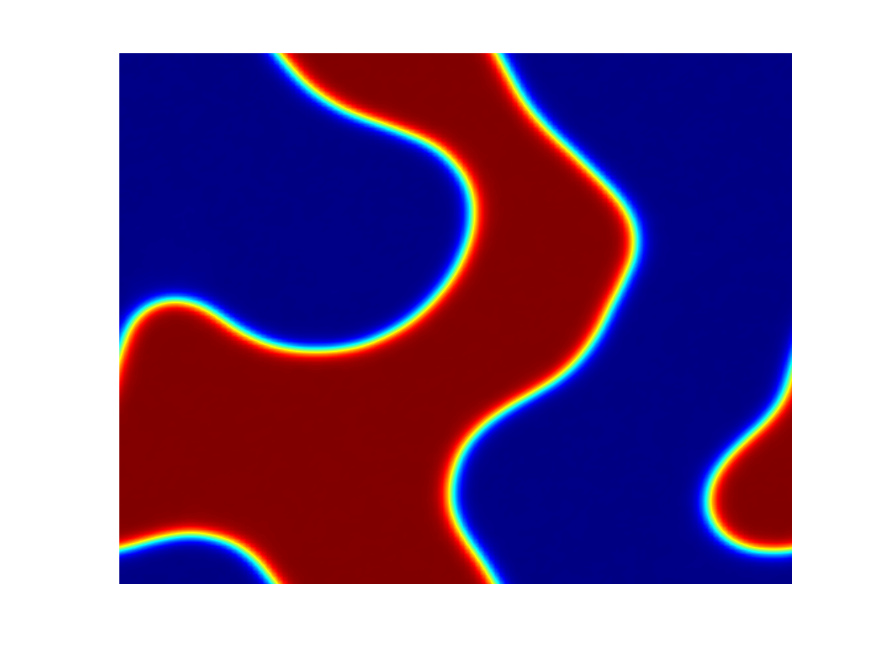}
		\end{minipage}%
	}%
	\setlength{\abovecaptionskip}{0.0cm} 
	\setlength{\belowcaptionskip}{0.0cm}
	\caption{The dynamic snapshots of the numerical solution $\phi$ obtained by the unbalanced $L2$-$1_{\sigma}$-sESAV scheme with $\alpha=0.9,0.7,0.4$ (from top to bottom, respectively): the double-well potential}	\label{figEx7_1}
\end{figure}
\begin{figure}[!htb]
	\vspace{-12pt}
	\centering
	\subfigure[maximum norm of $\phi$]
	{
		\includegraphics[width=0.32\textwidth]{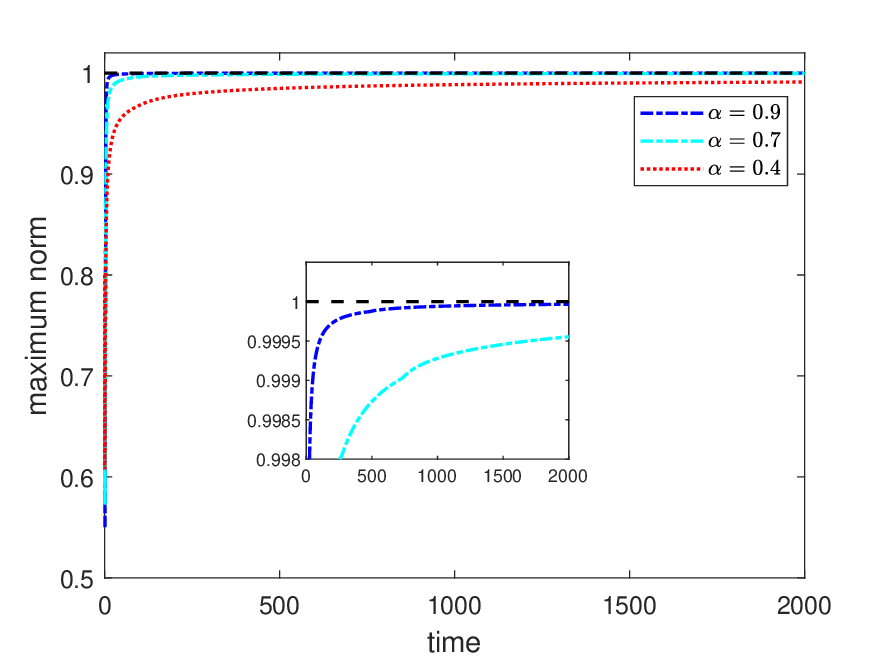}
		\label{figEx7_2a}
	}%
	\subfigure[energy]
	{
		\includegraphics[width=0.32\textwidth]{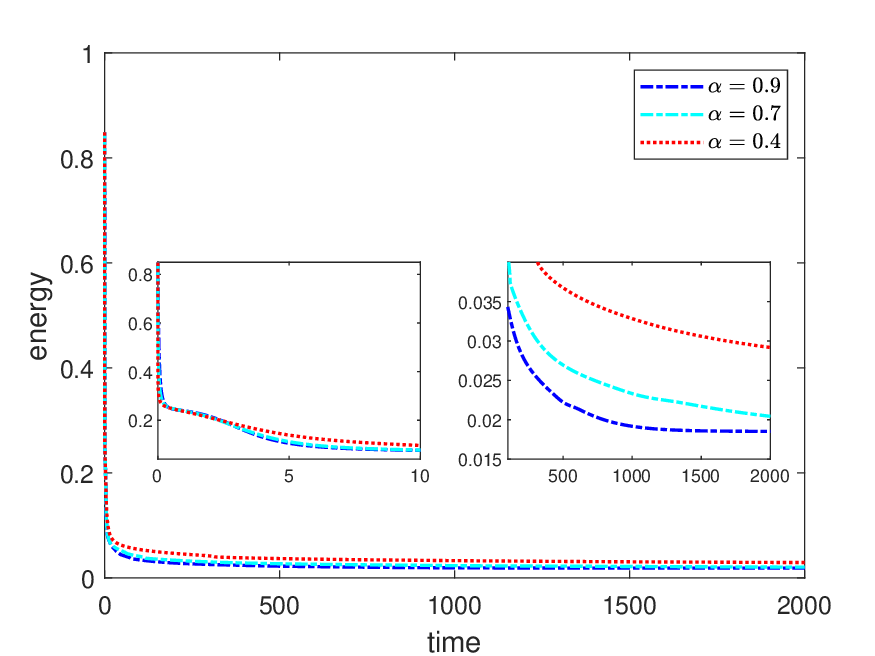}
		\label{figEx7_2b}
	}%
	\subfigure[time steps]
	{
		\includegraphics[width=0.32\textwidth]{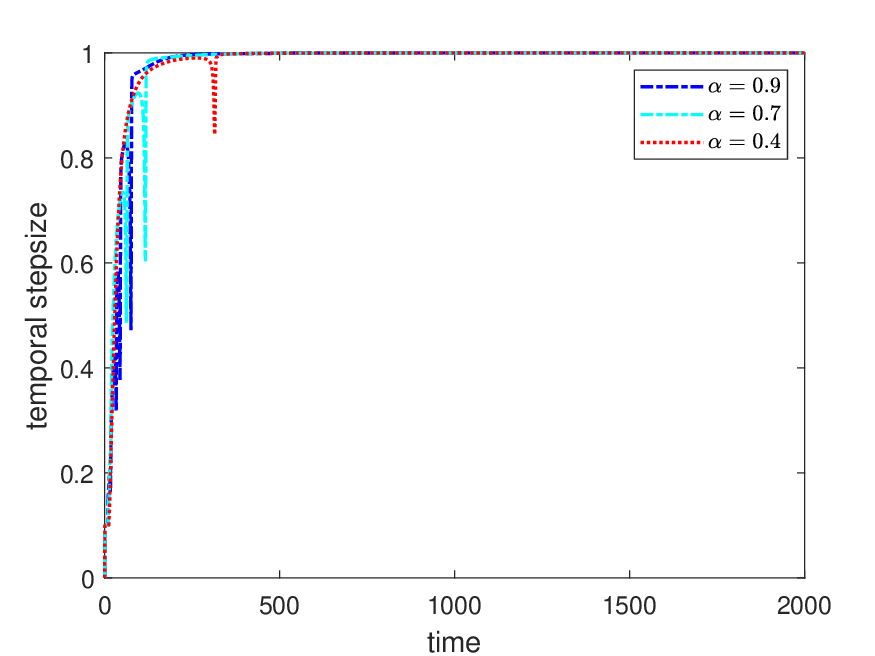}
		\label{figEx7_2c}
	}%
	\setlength{\abovecaptionskip}{0.0cm} 
	\setlength{\belowcaptionskip}{0.0cm}
	\caption{Time evolutions of the maximum norm (left), energy (middle), and time steps (right) for the unbalanced $L2$-$1_{\sigma}$-sESAV scheme with $\alpha=0.9,0.7,0.4$: the double-well potential}	\label{figEx7_2}
\end{figure}
For the double-well potential, the profiles of coarsening dynamics with different fractional orders $ \alpha = 0.9, 0.7 $ and $ 0.4 $ are depicted in Figure \ref{figEx7_1}, where snapshots are taken at times $ t = 20, 100, 500 $, and $2000$, respectively. It is clearly observed that the coarsening speed governed by the tFAC equation depends heavily on the fractional order $\alpha$ and smaller values of $\alpha$ require much longer time to reach the steady state. 
Moreover, Figure \ref{figEx7_2a}--\ref{figEx7_2b} demonstrate that the proposed unbalanced $L2$-$1_{\sigma}$-sESAV scheme with the adaptive time-stepping strategy successfully preserves both the discrete energy stability and MBP, and Figure \ref{figEx7_2c} further shows that the time steps are accurately and effectively selected to capture the energy variations in the adaptive method. For the Flory--Huggins potential, the corresponding numerical results are presented in Figures \ref{figEx7_3}--\ref{figEx7_4}, and similar conclusions can also be drawn.

\begin{figure}[!ht]
	\vspace{-10pt}
	\centering
	\subfigure[$t=20$]
	{
		\begin{minipage}[t]{0.24\linewidth}
			\centering
			\includegraphics[width=1.5in]{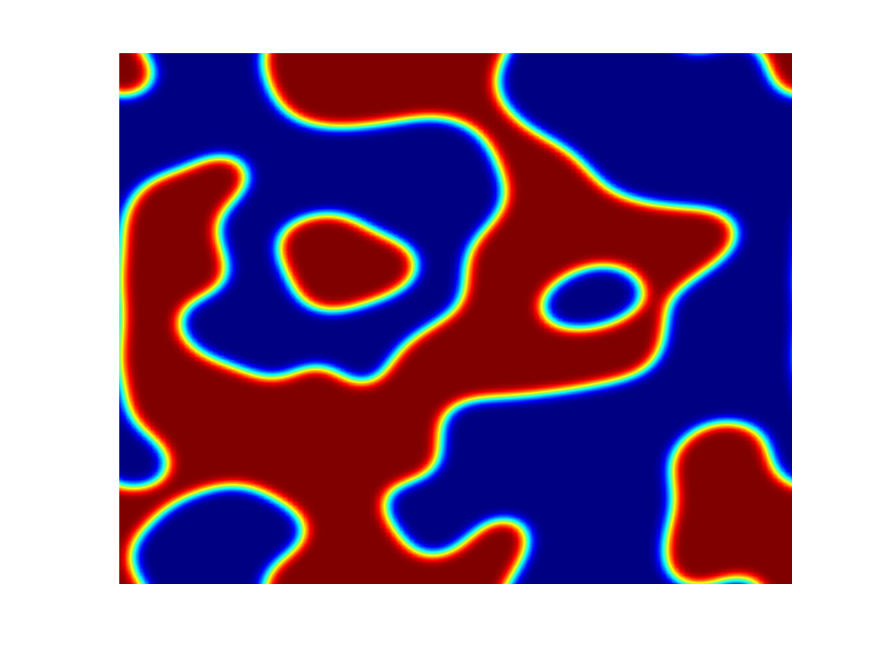}
			\includegraphics[width=1.5in]{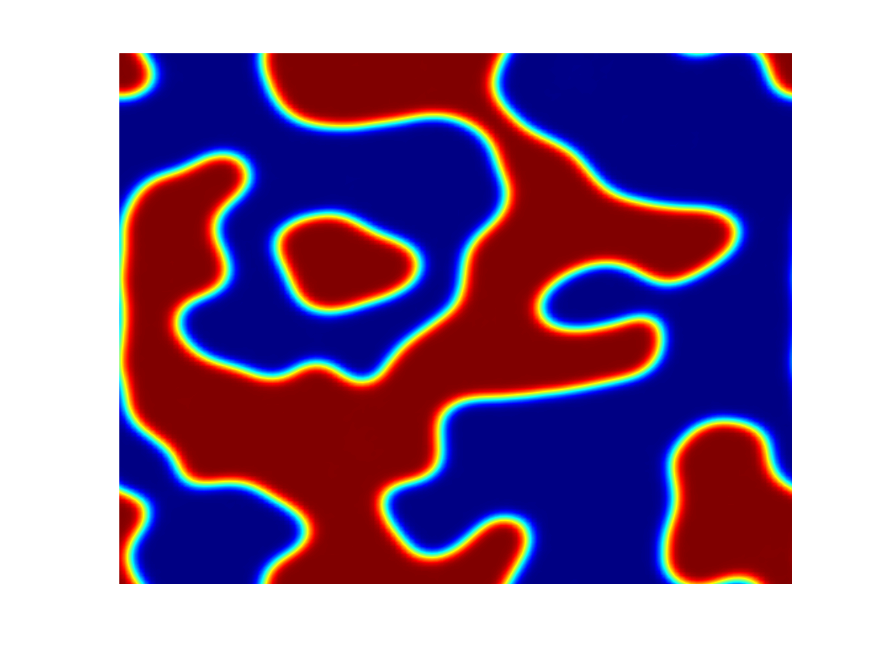}
			\includegraphics[width=1.5in]{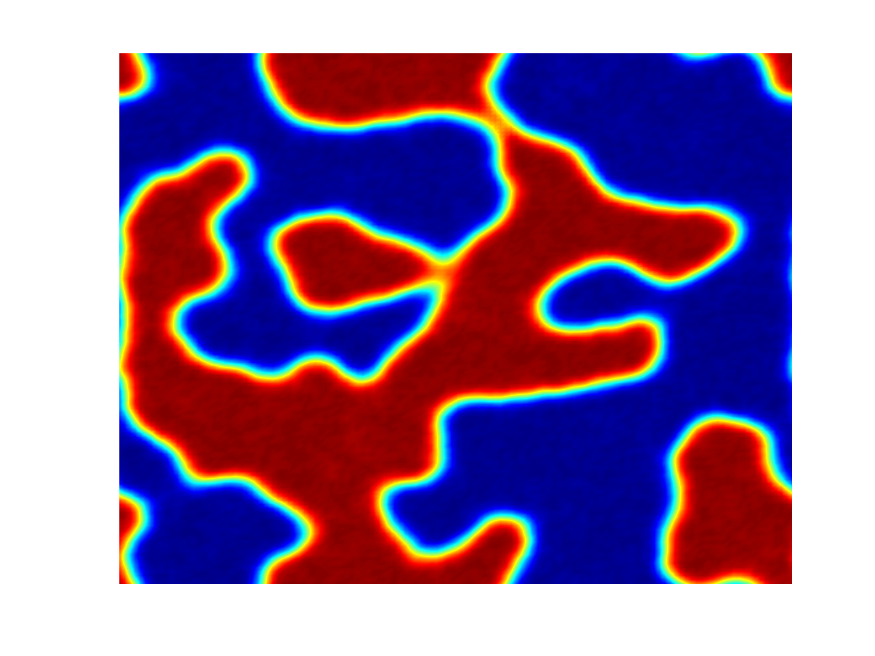}
		\end{minipage}%
	}%
	\subfigure[$t=100$]
	{
		\begin{minipage}[t]{0.24\linewidth}
			\centering
			\includegraphics[width=1.5in]{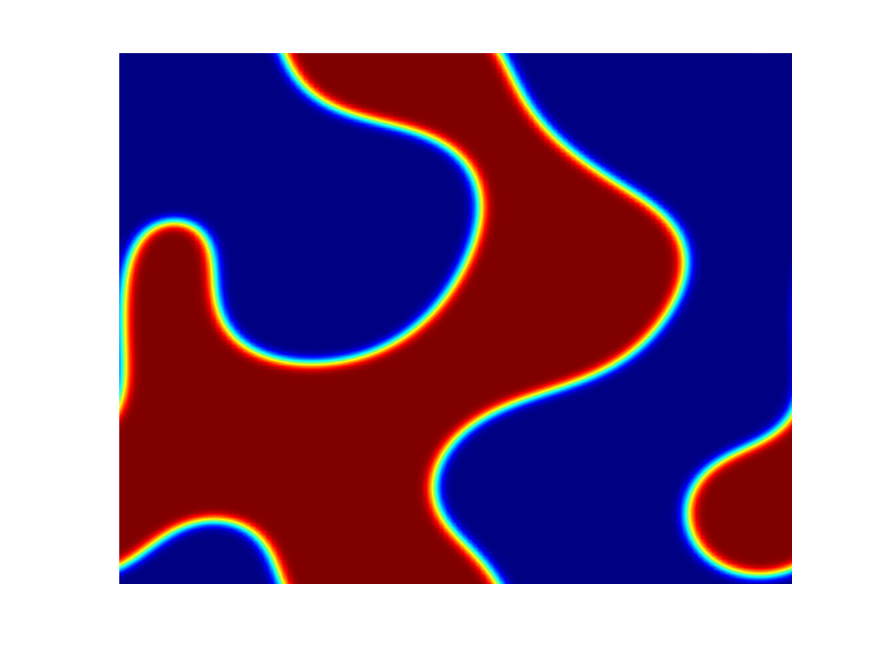}
			\includegraphics[width=1.5in]{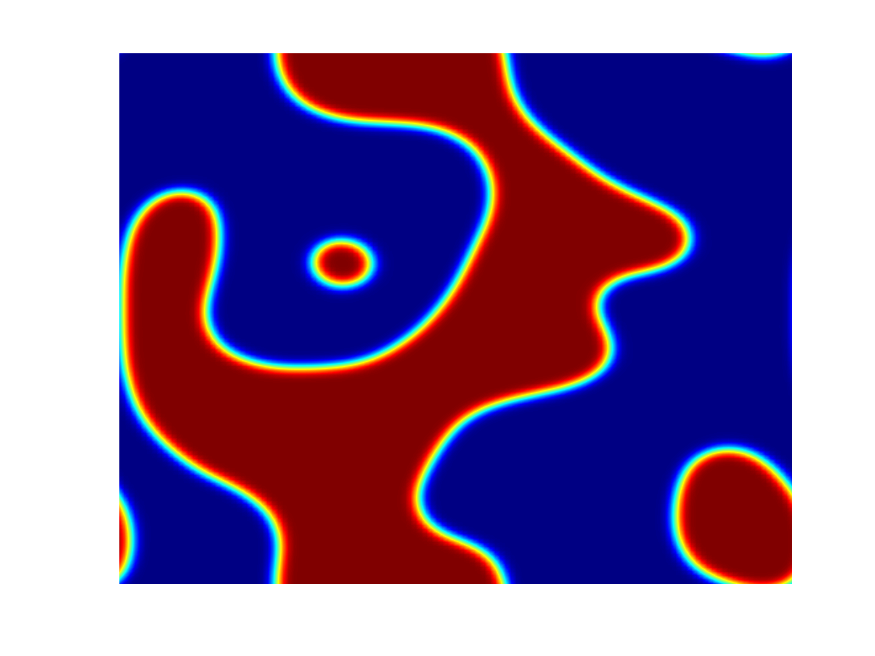}
			\includegraphics[width=1.5in]{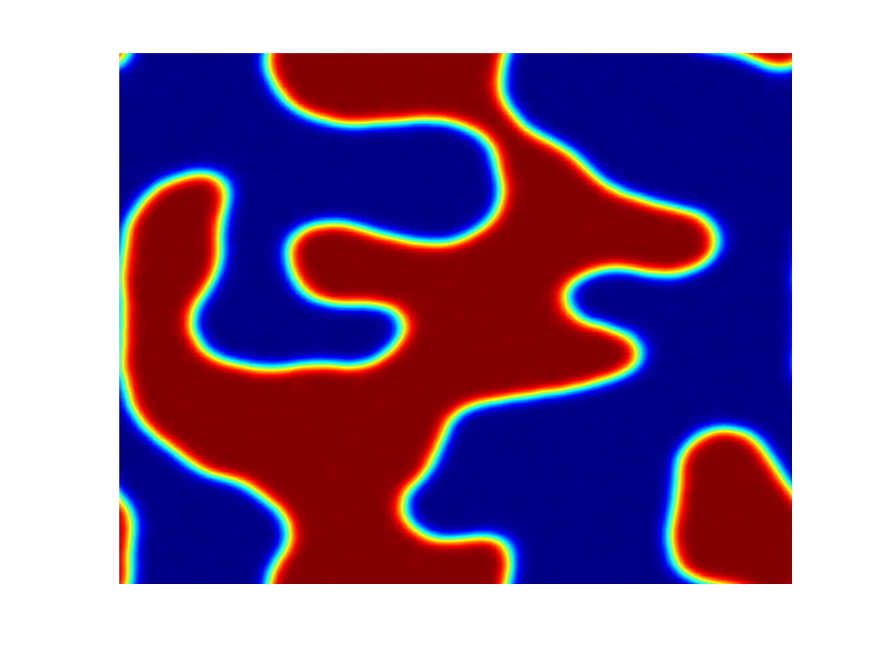}
		\end{minipage}%
	}%
	\subfigure[$t=500$]{
		\begin{minipage}[t]{0.24\linewidth}
			\centering
			\includegraphics[width=1.5in]{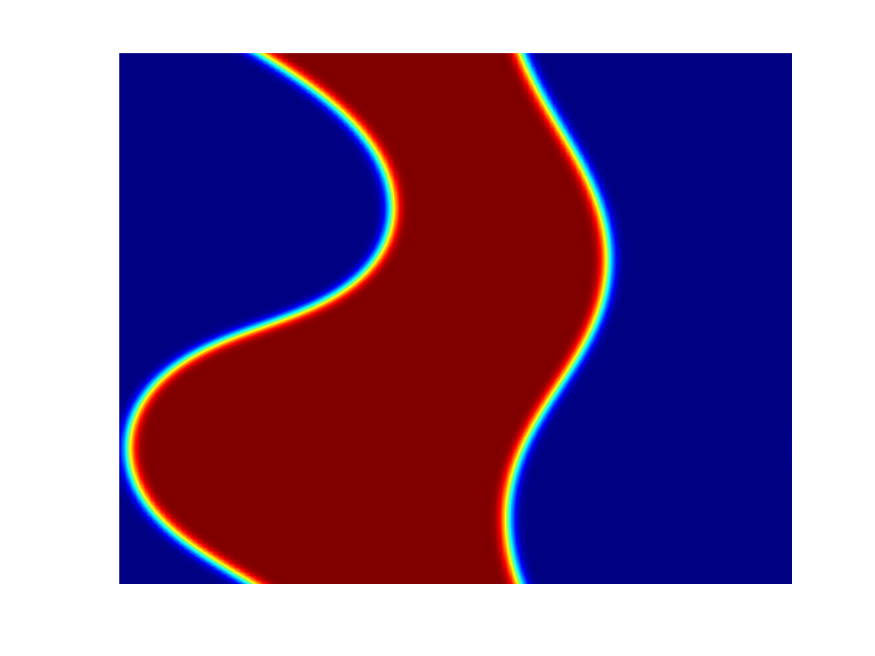}
			\includegraphics[width=1.5in]{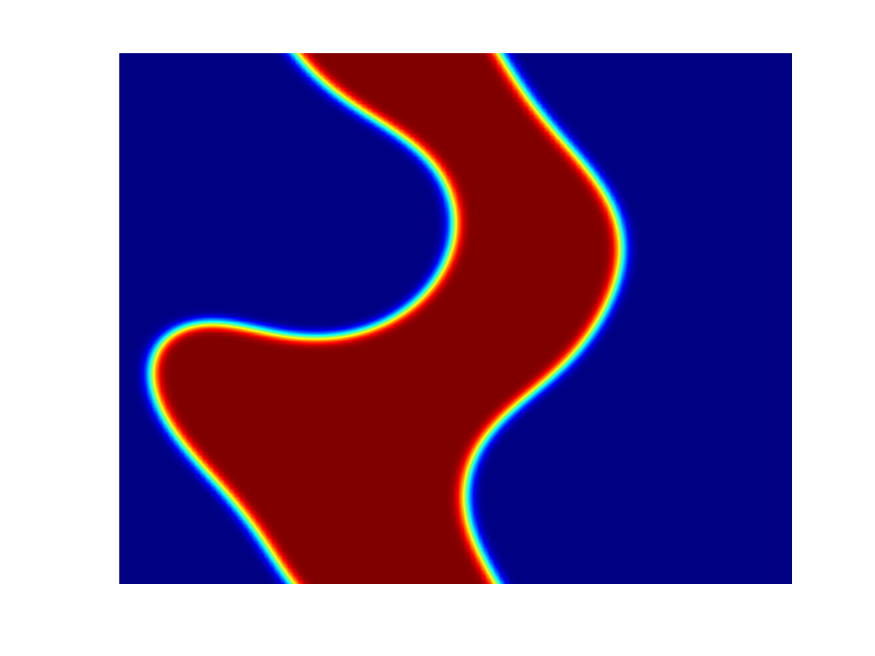}
			\includegraphics[width=1.5in]{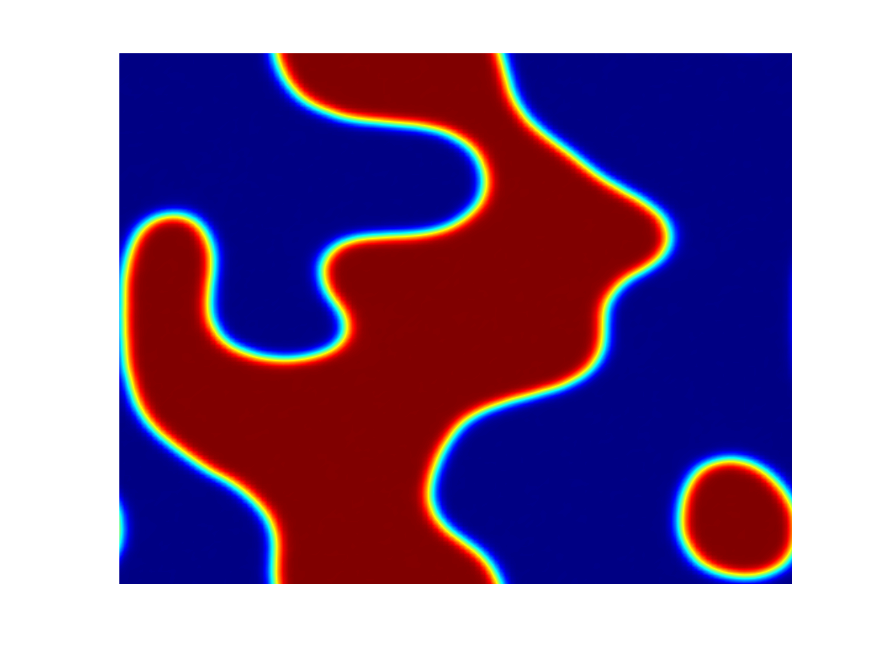}
		\end{minipage}%
	}%
	\subfigure[$t=2000$]
	{
		\begin{minipage}[t]{0.24\linewidth}
			\centering
			\includegraphics[width=1.5in]{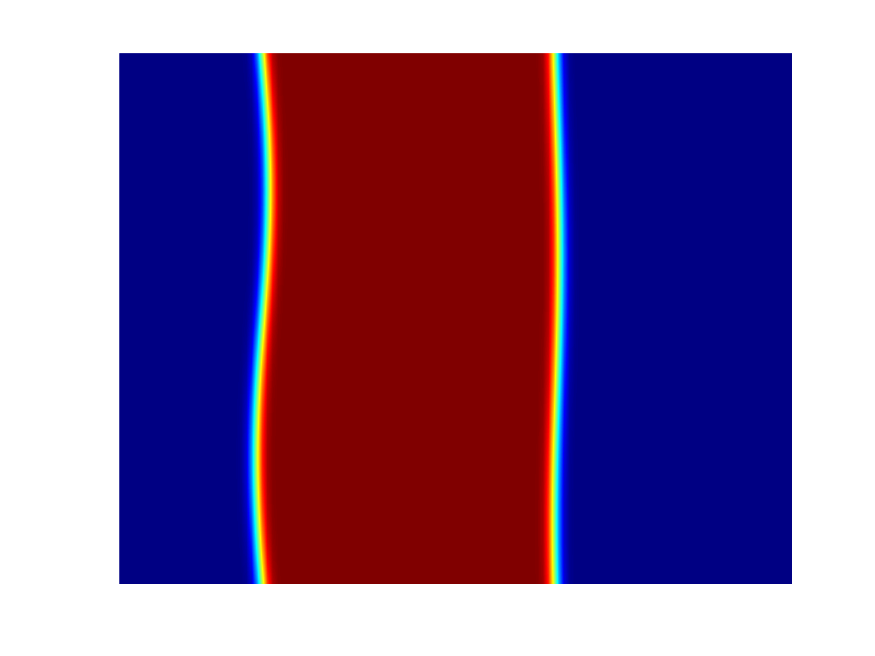}
			\includegraphics[width=1.5in]{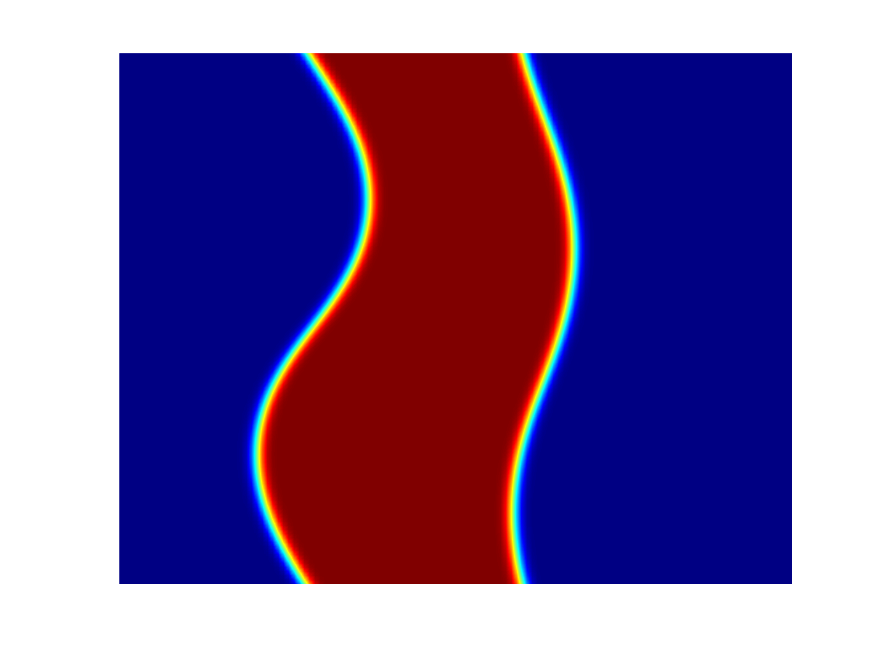}
			\includegraphics[width=1.5in]{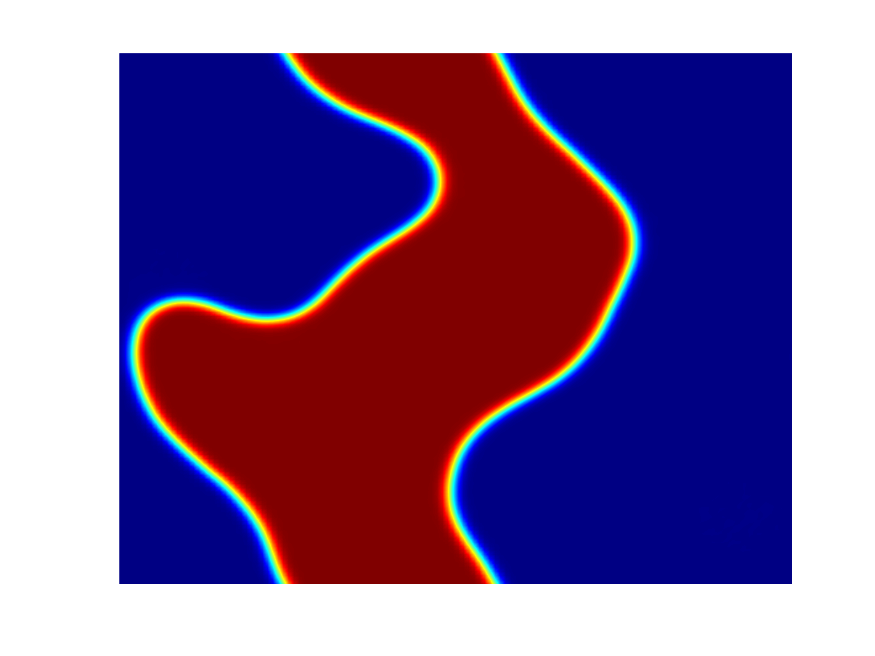}
		\end{minipage}%
	}%
	\setlength{\abovecaptionskip}{0.0cm} 
	\setlength{\belowcaptionskip}{0.0cm}
	\caption{The dynamic snapshots of the numerical solution $\phi$ obtained by the unbalanced $L2$-$1_{\sigma}$-sESAV scheme with $\alpha=0.9,0.7,0.4$ (from top to bottom, respectively): the Flory--Huggins potential}	\label{figEx7_3}
\end{figure}
\begin{figure}[!ht]
	\vspace{-12pt}
	\centering
	\subfigure[maximum norm of $\phi$]
	{
		\includegraphics[width=0.32\textwidth]{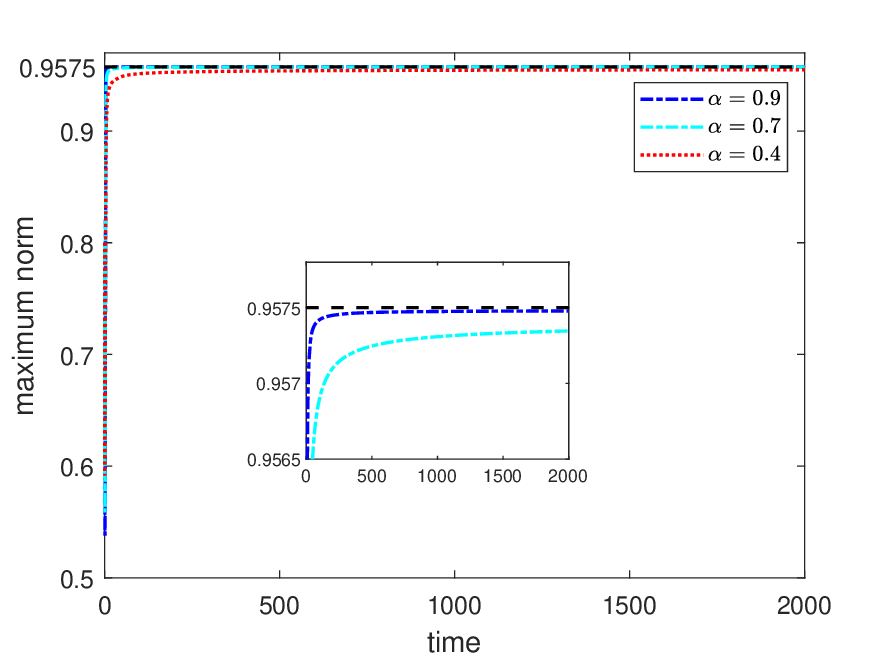}
		\label{figEx7_4a}
	}%
	\subfigure[energy]
	{
		\includegraphics[width=0.32\textwidth]{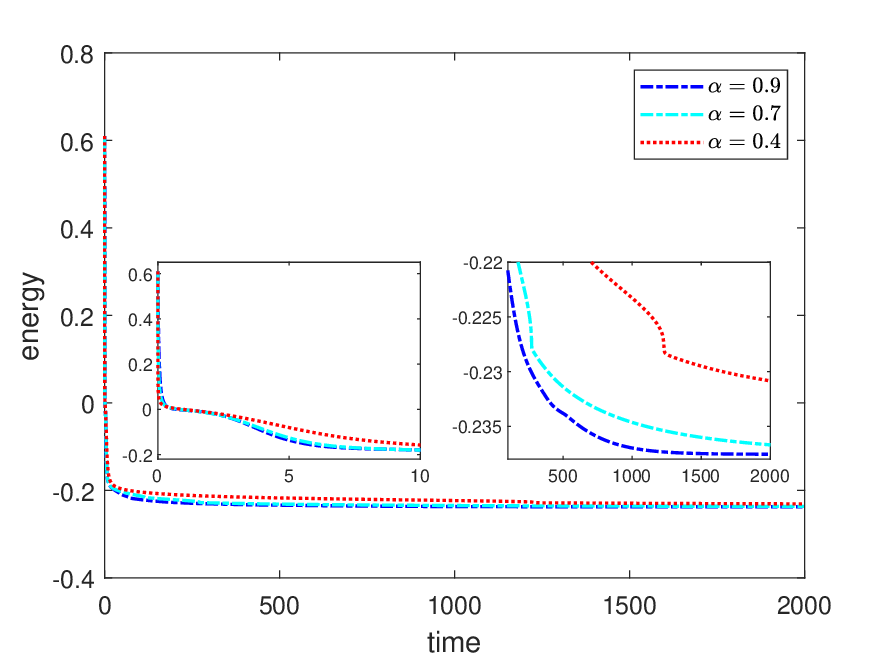}
		\label{figEx7_4b}
	}%
	\subfigure[time steps]
	{
		\includegraphics[width=0.32\textwidth]{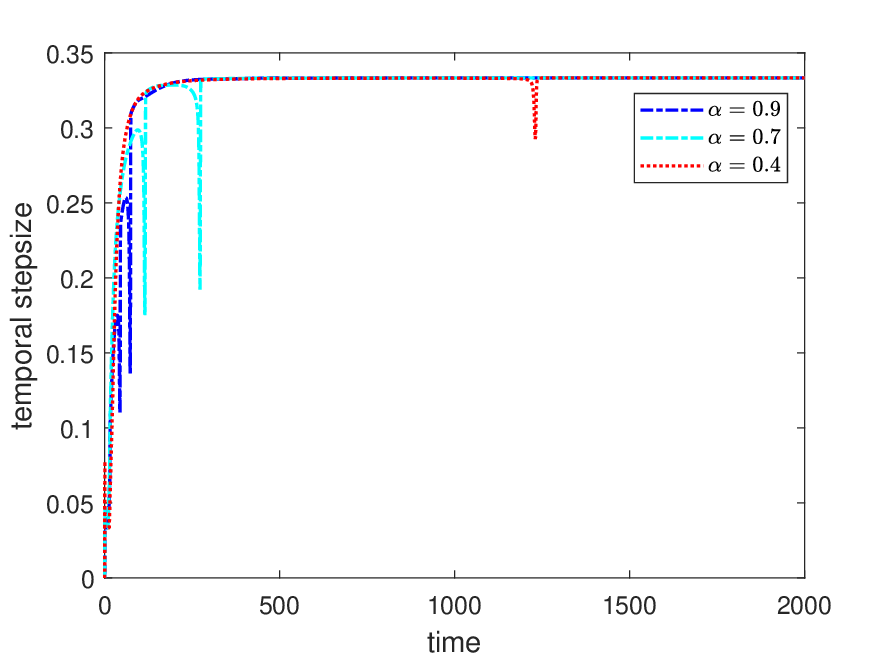}
		\label{figEx7_4c}
	}%
	\setlength{\abovecaptionskip}{0.0cm} 
	\setlength{\belowcaptionskip}{0.0cm}
	\caption{Time evolutions of the maximum norm (left), energy (middle), and time steps (right) for the unbalanced $L2$-$1_{\sigma}$-sESAV scheme with $\alpha=0.9,0.7,0.4$: the Flory--Huggins potential}	\label{figEx7_4}
        \vspace{-10pt}
\end{figure}

\section{Concluding remarks }
In this paper, we first established the MBP for the tFAC model \eqref{Model:tAC} with two different nonlinear potentials. Subsequently, we developed and analyzed a linear stabilized nonuniform time-stepping scheme, referred to as the $L1$-sESAV scheme, in which an essential auxiliary functional was introduced to ensure that the first-order approximation of the SAV does not compromise the temporal accuracy of the phase function. This treatment is critical for deriving unconditional energy stability of nonuniform time-stepping linear scheme. Moreover, the non-negativity of this novel auxiliary functional also allows for the design of a stabilization term that effectively controls the nonlinear term, which together with a newly developed prediction strategy guarantees the discrete MBP-preservation. To the best of our knowledge, this is the first $(2-\alpha)$th order $L1$-type linear scheme with unconditional preservation of the energy stability and MBP. Building on these ideas, we further established a linear second-order scheme, termed the $L2$-$1_{\sigma}$-sESAV scheme, which can preserve the energy stability unconditionally and  MBP conditionally. To enhance the discrete MBP preservation for large time steps, we also introduced an improved unbalanced stabilization term by imposing appropriate boundedness and monotonicity assumptions on the auxiliary functional, and then an unbalanced $L2$-$1_{\sigma}$-sESAV scheme was developed and discussed. Numerical simulations for the 2D and 3D tFAC models with either the double-well potential or the Flory--Huggins potential suggest the satisfactory and high effectiveness of the proposed methods. Moreover, the ideas and derivations presented in this paper enable the development of MBP-preserving numerical methods based on other discrete Caputo derivatives, such as the $L1_{R}$ formula \cite{SISC_Liao_2021} and the fast versions of $L1$ \cite{ACM_Ji_2020} and $L2$-$1_{\sigma}$ \cite{JCP_Liao_2020} formulas. Currently, rigorous error analyses of SAV-type methods for time-fractional phase-field models under weak singularity assumptions remain scarce,  representing an important direction for future research.

\section*{CRediT authorship contribution statement}
\textbf{Bingyin Zhang:} Methodology, Formal analysis, Software, Writing-Original draft, Funding acquisition.
\textbf{Hongfei Fu:} Methodology, Conceptualization, Supervision, Writing-Reviewing and Editing, Funding acquisition.

\section*{Data availability}  A free github repository contains the Matlab code employed, which can be accessed through the following link: \url{https://github.com/HFu20/time-Fractional-Allen-Cahn/tree/sESAV-MBP}.

\section*{Declaration of competing interest} The authors declare that they have no known competing financial interests or personal relationships that could have appeared to influence the work reported in this paper.

\section*{Acknowledgements}
B. Zhang was supported in part by the Fundamental Research Funds for the Central Universities (No. 202461103). H. Fu was supported in part by the National Natural Science Foundation of China (Nos. 11971482, 12131014), by the Shandong Provincial Natural Science Foundation (No. ZR2024MA023), by the Fundamental Research Funds for the Central Universities (No. 202264006), by the OUC Scientific Research Program for Young Talented Professionals.

\appendix
\section{Proof of Lemma \ref{lem:frac}}\label{App:A}
\setcounter{equation}{0}
\renewcommand\theequation{A.\arabic{equation}}
For convenience, we set $ w(t) = v(t) - v( t_{*} ) $. Then, we have $ \partial_t w(t) = \partial_t v(t) $ and $ w( t_{*} ) = 0 $, which further gives
\begin{equation*}
	\begin{aligned}
		{}_{0}^{C}D^{\alpha}_{t} v ( t_{*} ) = \frac{1}{ \Gamma(1-\alpha) } \int^{t_{*}}_{0} ( t_{*} - s )^{-\alpha} \partial_s w (s) ds.
	\end{aligned}
\end{equation*}
Moreover, since $ v \in C[0,T] \cap C^{1} ( 0, T ]  $, one can easily verify that $ \lim\limits_{s\rightarrow t_{*} } ( t_{*} - s )^{-\alpha} w(s) = 0$. Thus, applying integration by parts yields
\begin{equation*}
	\begin{aligned}
		{}_{0}^{C}D^{\alpha}_{t} v (t_{*})    = - \frac{1}{ \Gamma(1-\alpha) } \left( t_{*}^{-\alpha} w(0) + \alpha \int^{t_{*}}_{0} \frac{w (s)}{ ( t_{*} - s )^{1+\alpha} } ds \right).
	\end{aligned}
\end{equation*}
If $v$ attains its minimum at $ t_{*} \in ( 0, T ] $, then $ w(s) \geq 0 $ for all $ s \in [ 0, T ] $. Consequently, we obtain
$
{}_{0}^{C}D^{\alpha}_{t} v (t_{*}) \leq 0 .
$
An analogous argument yields the corresponding inequality when $v$ attains its maximum.

\section{Proof of Lemma \ref{lem:Trun_L1}}\label{App:B}
\setcounter{equation}{0}
\renewcommand\theequation{B.\arabic{equation}}
The integral version of Taylor’s theorem gives
\begin{equation}\label{TrunL1:1}
	\mathcal{R}^{n} [v] =  ( 1 + r_{n} ) \int^{ t_{n} }_{ t_{n-1} } v^{\prime\prime}(s) ( t_{n-1} - s ) ds + r_{n} \int^{ t_{n} }_{ t_{n-2} } v^{\prime\prime}(s) ( s- t_{n-2} ) ds.
\end{equation}
Under the regularity assumption, one has
\begin{equation*}
	\begin{aligned}
		\Big\vert ( 1 + r_{n} ) \int^{ t_{n} }_{ t_{n-1} } v^{\prime\prime}(s) ( t_{n-1} - s ) ds \Big\vert \leq C_{v} t_{n-1}^{\iota-2} \tau_{n}^2, \quad 2 \leq n \leq N,\\
		\Big\vert r_{n} \int^{ t_{n} }_{ t_{n-2} } v^{\prime\prime}(s) ( s - t_{n-2} ) ds \Big\vert \leq C_{v} t_{n-2}^{\iota-2} ( \tau_{n-1} + \tau_{n} )^2, \quad 3 \leq n \leq N.
	\end{aligned}
\end{equation*}
Specifically, when $n=2$, the second right-hand side term of \eqref{TrunL1:1} can be estimated by
\begin{equation*}
	\Big\vert r_{2} \int^{ t_{2} }_{ t_{0} } s v^{\prime\prime}(s) ds \Big\vert \leq C_{v} \int^{ t_{2} }_{ t_{0} } ( s + s^{\iota-1}) ds \leq C_{v}  ( \tau_{1} + \tau_{2} )^{\iota} / \iota.
\end{equation*}
Inserting the above estimates into \eqref{TrunL1:1} completes the proof of \eqref{TrunL1:00}.

In particular, when the graded temporal grids are used, we recall the conclusion in \cite{SINUM_Mustapha_2014} that there exists a constant $ C_{\gamma} > 0 $ such that $ \tau_{k} \leq C_{\gamma} t_{k}^{1-1/\gamma} N^{-1} $ and $ t_{k} \leq C_{\gamma} t_{k-1} $ for $2 \leq k \leq N  $,
which further implies $ \tau_{1} + \tau_{2} \leq 2 \tau_{2} \leq C_{\gamma}' N^{-\gamma} $ and thus $ ( \tau_{1} + \tau_{2} )^{\iota} \leq C_{\gamma}' N^{-\gamma\iota} $. Denote $ \zeta = \min\{2,\gamma\iota\} $, then we have $ \tau_{\ell}^{\zeta} \leq C_{\gamma} t_{\ell}^{\zeta-\zeta/\gamma} N^{-\zeta} $ with $ \ell = n, n-1 $. Thus, it holds that
$$
\begin{aligned}
	t_{n-2}^{\iota-2} ( \tau_{n-1} + \tau_{n} )^{2} 
	& \leq C_{\gamma} \bigl( t_{n-1}^{\iota-2+\zeta-\zeta/\gamma} \tau_{n-1}^{2-\zeta} N^{-\zeta} + t_{n}^{\iota-2+\zeta-\zeta/\gamma} \tau_{n}^{2-\zeta} N^{-\zeta} \bigr) \\
	& = C_{\gamma} \bigl( t_{n-1}^{\iota-\zeta/\gamma} ( \tau_{n-1}/t_{n-1} )^{2-\zeta} N^{-\zeta} + t_{n}^{\iota-\zeta/\gamma} ( \tau_{n}/t_{n} )^{2-\zeta} N^{-\zeta} \bigr) \\
	& \leq C_{\gamma} t_{n}^{\max\{0,\iota-2/\gamma\}} N^{-\zeta}.
\end{aligned}
$$
Finally, by inserting the above estimates into \eqref{TrunL1:00}, we complete the proof of \eqref{TrunL1:02}.

\section{ Proof of Theorem \ref{lem:L21_iter} }\label{App:C}
\setcounter{equation}{0}
\renewcommand\theequation{C.\arabic{equation}}

For any $ \mathfrak{s} \geq 1 $, we assume that $ \| \hat{\phi}^{1}_{ (\mathfrak{s} -1) } \|_{\infty} \leq \beta $, which further implies $ \| \hat{\phi}^{1-\varsigma}_{ (\mathfrak{s}-1) } \|_{\infty} \leq \beta $. Equivalently, equation \eqref{sch:L21_non_2} can be rewritten as
\begin{equation*}\label{MBP:2L1_inner_2}
	\begin{aligned}
		\big( ( B^{(1)}_{0} + \kappa ( 1 - \varsigma ) \mm  ) I - (1-\varsigma) \mm \varepsilon^{2} \Delta_h \big) \hat{\phi}^{1}_{(\mathfrak{s})} & = \mathbf{M}_{1} \phi^{0} + \mm \bigl(  f(\hat{\phi}^{1-\varsigma}_{(\mathfrak{s}-1)}) +\kappa \hat{\phi}^{1-\varsigma}_{(\mathfrak{s}-1)} \bigr),
	\end{aligned}
\end{equation*}
where $ \mathbf{M}_{1} := ( B^{(1)}_{0} - \kappa \varsigma \mm  ) I + \varsigma \mm \varepsilon^{2} \Delta_h $. Due to $B_{0}^{(1)} \geq \frac{4}{11} \int_{t_{0}}^{t_1} \frac{\omega_{1-\alpha}\left(t_1-s\right)}{\tau_1} \mathrm{~d} s = \frac{ 4 \tau_{1}^{-\alpha} }{ 11 \Gamma( 2 - \alpha ) }$ (see Lemma \ref{lem:L21_DC} (i)), it then follows from \eqref{MBP:L21_inner_tau} that all elements of $ \mathbf{M}_{1} $ are nonnegative, and thus, we have
\begin{equation*}\label{MBP:2L1_inner_3}
		\| \mathbf{M}_{1} \|_{\infty} \leq B^{(1)}_{0} - \kappa \varsigma \mm,
\end{equation*}
which together with Lemmas \ref{lem:MBP_left}--\ref{lem:MBP_right} gives us
\begin{equation*}
	( B^{(1)}_{0} + \kappa ( 1 - \varsigma ) \mm  ) \| \hat{\phi}^{1}_{(\mathfrak{s})} \|_{\infty} \leq ( B^{(1)}_{0} - \kappa \varsigma \mm ) \| \phi^{0} \|_{\infty} + \kappa \mm \beta
    = ( B^{(1)}_{0} + \kappa ( 1 - \varsigma ) \mm  ) \beta.
\end{equation*}
Consequently, $ \| \hat{\phi}^{1}_{(\mathfrak{s})} \|_{\infty} \leq \beta $ holds for any $ \mathfrak{s} \geq 1 $.

Furthermore, equation \eqref{sch:L21_non_2} is also equivalent to
\begin{equation*}\label{Cover:L21_0}
	\begin{aligned}
		\frac{B^{(1)}_{0}}{ 1 - \varsigma } ( \hat{\phi}^{1-\varsigma}_{(\mathfrak{s})} - \phi^{0} )  = -\mm  \bigl( - \varepsilon^2 \Delta_h \hat{\phi}^{1-\varsigma}_{(\mathfrak{s})} - f(\hat{\phi}^{1-\varsigma}_{(\mathfrak{s}-1)}) + \kappa ( \hat{\phi}^{1-\varsigma}_{(\mathfrak{s})} - \hat{\phi}^{1-\varsigma}_{(\mathfrak{s}-1)} ) \bigr), \quad \mathfrak{s} \geq 1,
	\end{aligned}
\end{equation*}
which defines a mapping $ \mt [v] = w $ from $ \mathbb{V}_{\beta} $ to $\mathbb{V}_{\beta} $, i.e.,
\begin{equation*}
	\begin{aligned}
		\frac{B^{(1)}_{0}}{ 1 - \varsigma } ( w - \phi^{0} ) = -\mm  ( - \varepsilon^2 \Delta_h w - f(v) + \kappa ( w - v ) ), \quad v \in \mathbb{V}_{\beta}.
	\end{aligned}
\end{equation*}
Similar as in Theorem \ref{thm:cover_L1_iter}, it is straightforward to verify that $ \mt $ is a contractive mapping under the time-step condition \eqref{Cover:L21_tau}. Consequently, the iterative scheme \eqref{sch:L21_non_2} admits a unique fixed point in $ \mathbb{V}_{\beta} $, denoted by $ \hat{\phi}^{1-\varsigma} $, 
and $\hat{\phi}^{1} = \frac{\phi^{1-\varsigma} - \varsigma \phi^{0}}{1-\varsigma}$ is just the unique solution of the nonlinear scheme \eqref{sch:L21_non_3}. The proof of Theorem \ref{lem:L21_iter} is completed.  


\begin{thebibliography}{10}
	\bibitem{Acta_Allen_1979}
	S.~Allen, J.~Cahn, A microscopic theory for antiphase boundary motion and its
	application to antiphase domain coarsening, Acta Metall. 27 (1979)
	1085--1095.
	
	\bibitem{JCP_CH_1958}
	J.~Cahn, J.~Hilliard, Free energy of a nonuniform system {I}: {I}nterfacial
	free energy, J. Chem. Phys. 28 (1958) 258--267.
	
	\bibitem{ARFM_Anderson_1998}
	D.~Anderson, G.~McFadden, A.~Wheeler, Diffuse-interface methods in fluid
	mechanics, Annu. Rev. Fluid Mech. 30 (1998) 139--165.
	
	\bibitem{PRE_Qian_2003}
	T.~Qian, X.~Wang, P.~Sheng, Molecular scale contact line hydrodynamics of
	immiscible flows, Phys. Rev. E 68 (2003) 016306.
	
	\bibitem{JCP_Du_2004}
	Q.~Du, C.~Liu, X.~Wang, A phase field approach in the numerical study of the
	elastic bending energy for vesicle membranes, J. Comput. Phys. 198 (2004)
	450--468.
	
	\bibitem{IJNMBE_Oden_2012}
	A.~Hawkins-Daarud, K.~Zee, J.~Oden, Numerical simulation of a thermodynamically
	consistent four-species tumor growth model, Int. J. Numer. Methods Biomed.
	Eng. 8 (2012) 3--24.
	
	\bibitem{JTB_Wise_2008}
	S.~Wise, J.~Lowengrub, H.~Frieboes, V.~Cristini, Three-dimensional multispecies
	nonlinear tumor growth {I}: {M}odel and numerical method, J. Theor. Biol. 253
	(2008) 524--543.
	
	\bibitem{CMAME_Huang_2024}
	Q.~Huang, Z.~Qiao, H.~Yang, Maximum bound principle and non-negativity
	preserving {ETD} schemes for a phase field model of prostate cancer growth
	with treatment, Comput. Methods Appl. Mech. Engrg. 426 (2024) 116981.
	
	\bibitem{SISC_Tang_2019}
	T.~Tang, H.~Yu, T.~Zhou, On energy dissipation theory and numerical stability
	for time-fractional phase-field equations, SIAM J. Sci. Comput. 41 (2019)
	A3757--A3778.
	
	\bibitem{JCP_Wang_2017}
	Z.~Li, H.~Wang, D.~Yang, A space-time fractional phase-field model with tunable
	sharpness and decay behavior and its efficient numerical simulation, J.
	Comput. Phys. 347 (2017) 20--38.
	
	\bibitem{JDE_Akagi_2016}
	G.~Akagi, G.~Schimperna, A.~Segatti, Fractional {Cahn--Hilliard}, {Allen--Cahn}
	and porous medium equations, J. Differ. Equ. 261 (2016) 2935--2985.
	
	\bibitem{SINUM_Ainthworth_2017}
	M.~Ainthworth, Z.~Mao, Analysis and approximation of a fractional
	{Cahn--Hilliard} equation, SIAM J. Numer. Anal. 55 (2017) 1689--1718.
	
	\bibitem{SINUM_Du_2019}
	Q.~Du, L.~Ju, X.~Li, Z.~Qiao, Maximum principle preserving exponential time
	differencing schemes for the nonlocal {Allen--Cahn} equation, SIAM J. Numer.
	Anal. 57 (2019) 875--898.
	
	\bibitem{PASMA_Inc_2018}
	M.~Inc, A.~Yusuf, A.~Aliyu, D.~Baleanu, Time-fractional {Cahn--Allen} and
	time-fractional {Klein--Gordon} equations: lie symmetry analysis, explicit
	solutions and convergence analysis, Physica A Stat. Mech. Appl. 493 (2018)
	94--106.
	
	\bibitem{CMA_Wang_2019}
	H.~Liu, A.~Cheng, H.~Wang, J.~Zhao, Time-fractional {A}llen--{C}ahn and
	{C}ahn--{H}illiard phase-field models and their numerical investigation,
	Comput. Math. Appl. 76 (2019) 1876--1892.
	
	\bibitem{CMAME_Xu_2016}
	F.~Song, C.~Xu, C.~Karniadakis, A fractional phase-field model for two-phase
	flows with tunable sharpness: {A}lgorithms and simulations, Comput. Methods
	Appl. Mech. Engrg. 305 (2016) 376--404.
	
	\bibitem{CPC_Chen_2019}
	L.~Chen, J.~Zhao, W.~Cao, H.~Wang, J.~Zhang, An accurate and efficient
	algorithm for the time-fractional molecular beam epitaxy model with slope
	selection, Comput. Phys. Commun. 245 (2019) 106842.
	
	\bibitem{SISC_Hou_2021}
	D.~Hou, C.~Xu, Highly efficient and energy dissipative schemes for the time
	fractional {A}llen--{C}ahn equation, SIAM J. Sci. Comput. 43 (2021)
	A3305--A3327.
	
	\bibitem{PRL_Golding_2006}
	I.~Golding, E.~Cox, Physical nature of bacterial cytoplasm, Phys. Rev. Lett. 96
	(2006) 098102.
	
	\bibitem{Nature_Kirchner_2000}
	J.~Kirchner, X.~Feng, C.~Neal, Fractal stream chemistry and its implications
	for contaminant transport in catchments, Nature 403 (2000) 524.
	
	\bibitem{PSS_Nigmatullin_1986}
	R.~Nigmatullin, The realization of the generalized transfer equation in a
	medium with fractal geometry, Phys. Status Solidi. 133 (1986) 425--430.
	
	\bibitem{MS_Sharma_2015}
	A.~Sharma, S.~Namsani, J.~Singh, Molecular simulation of shale gas adsorption
	and diffusion in inorganic nanopores, Mol. Simul. 41 (2015) 414--422.
	
	\bibitem{JCP_Zhokn_2017}
	A.~Zhokh, P.~Strizhak, Non-{F}ickian diffusion of methanol in mesoporous media:
	geometrical restrictions or adsorption-induced?, J. Chem. Phys. 146 (2017)
	124704.
	
	\bibitem{PRL_Chepizhko_2013}
	O.~Chepizhko, F.~Peruani, Diffusion, subdiffusion, and trapping of active
	particles in heterogeneous media, Phys. Rev. Lett. 111 (2013) 160604.
	
	\bibitem{PR_Metzler_2000}
	R.~Metzler, J.~Klafter, The random walk’s guide to anomalous diffusion: a
	fractional dynamics approach, Phys. Rep. 339 (2000) 1--77.
	
	\bibitem{JMAA_Wang_2019}
	H.~Wang, X.~Zheng, Wellposedness and regularity of the variable-order
	time-fractional diffusion equations, J. Math. Anal. Appl. 475 (2019)
	1778--1802.
	
	\bibitem{SIREV_Du_2021}
	Q.~Du, L.~Ju, X.~Li, Z.~Qiao, Maximum bound principles for a class of
	semilinear parabolic equations and exponential time-differencing schemes,
	SIAM Rev. 63 (2021) 317--359.
	
	\bibitem{JSC_Du_2020}
	Q.~Du, J.~Yang, Z.~Zhou, Time-fractional {A}llen--{C}ahn equations: Analysis
	and numerical methods, J. Sci. Comput. 85 (2020) 42.
	
	\bibitem{JCP_Liao_2020}
	H.~Liao, T.~Tang, T.~Zhou, A second-order and nonuniform time-stepping
	maximum-principle preserving scheme for time-fractional {A}llen--{C}ahn
	equations, J. Comput. Phys. 414 (2020) 109473.
	
	\bibitem{SISC_Liao_2021}
	H.~Liao, T.~Tang, T.~Zhou, An energy stable and maximum bound preserving scheme
	with variable time steps for time fractional {A}llen--{C}ahn equation, SIAM
	J. Sci. Comput. 43 (2021) A3503--A3526.
	
	\bibitem{JSC_Liao_2024}
	H.~Liao, X.~Zhu, H.~Sun, Asymptotically compatible energy and dissipation law
	of the nonuniform ${L}2$-$1_{\sigma}$ scheme for time fractional
	{A}llen--{C}ahn model, J. Sci. Comput. 99 (2024) 46.
	
	\bibitem{SISC_2011_Qiao}
	Z.~Qiao, Z.~Zhang, T.~Tang, An adaptive time-stepping strategy for the
	molecular beam epitaxy models, SIAM J. Sci. Comput. 33 (2011) 1395--1414.
	
	\bibitem{JCP_Gomez_2011}
	H.~Gomez, T.~Hughes, Provably unconditionally stable, second-order
	time-accurate, mixed variational methods for phase-field models, J. Comput.
	Phys. 230 (2011) 5310--5327.
	
	\bibitem{JCP_Wodo_2011}
	O.~Wodo, B.~Ganapathysubramanian, Computationally efficient solution to the
	{C}ahn--{H}illiard equation: {A}daptive implicit time schemes, mesh
	sensitivity analysis and the 3{D} isoperimetric problem, J. Comput. Phys. 230
	(2011) 6037--6060.
	
	\bibitem{IMA_Jin_2016}
	B.~Jin, R.~Lazarov, Z.~Zhou, An analysis of the {L1}, scheme for the
	subdiffusion equation with nonsmooth data, IMA J. Numer. Anal. 36 (2016)
	197--221.
	
	\bibitem{SINUM_Martin_2017}
	M.~Stynes, E.~O’Riordan, J.~Gracia, Error analysis of a finite difference
	method on graded meshes for a time-fractional diffusion equation, SIAM J.
	Numer. Anal. 55 (2017) 1057--1079.
	
	\bibitem{SINUM_Liao_2018}
	H.~Liao, D.~Li, J.~Zhang, Sharp error estimate of the nonuniform ${L}1$ formula
	for linear reaction-subdiffusion equations, SIAM J. Numer. Anal. 56 (2018)
	1112--1133.
	
	\bibitem{JCP_Xu_2007}
	Y.~Lin, C.~Xu, Finite difference/spectral approximations for the
	time-fractional diffusion equation, J. Comput. Phys. 225 (2007) 1533--1552.
	
	\bibitem{SINUM_Xu_2006}
	C.~Xu, T.~Tang, Stability analysis of large time-stepping methods for epitaxial
	growth models, SIAM J. Numer. Anal. 44 (2006) 1759--1779.
	
	\bibitem{ACM_Ji_2020}
	B.~Ji, H.~Liao, L.~Zhang, Simple maximum principle preserving time-stepping
	methods for time-fractional {Allen--Cahn} equation, Adv. Comput. Math. 46
	(2020) 37.
	
	\bibitem{JSC_Huang_2023}
	G.~Zhang, C.~Huang, A.~Alikhanov, B.~Yin, A high-order discrete energy decay
	and maximum-principle preserving scheme for time fractional {Allen--Cahn}
	equation, J. Sci. Comput. 96 (2023) 39.
	
	\bibitem{JCP_Shen_2018}
	J.~Shen, J.~Xu, J.~Yang, The scalar auxiliary variable ({SAV}) approach for
	gradient flows, J. Comput. Phys. 353 (2018) 407--416.
	
	\bibitem{SINUM_Shen_2018}
	J.~Shen, J.~Xu, Convergence and error analysis for the scalar auxiliary
	variable ({SAV}) schemes to gradient flows, SIAM J. Numer. Anal. 56 (2018)
	2895--2912.
	
	\bibitem{SINUM_Qiao_2022}
	L.~Ju, X.~Li, Z.~Qiao, Generalized {SAV}-exponential integrator schemes for
	{A}llen--{C}ahn type gradient flows, SIAM J. Numer. Anal. 60 (2022)
	1905--1931.
	
	\bibitem{SISC_Zhao_2024}
	R.~Qi, X.~Zhao, A unified design of energy stable schemes with variable steps
	for fractional gradient flows and nonlinear integro-differential equations,
	SIAM J. Sci. Comput. 46 (2024) A130--A155.
	
	\bibitem{JCP_Quan_2022}
	C.~Quan, B.~Wang, Energy stable {L2} schemes for time-fractional phase-field
	equations, J. Comput. Phys. 458 (2022) 111085.
	
	\bibitem{SISC_Ji_2020}
	B.~Ji, H.~Liao, Y.~Gong, L.~Zhang, Adaptive second-order {C}rank--{N}icolson
	time-stepping schemes for time-fractional molecular beam epitaxial growth
	models, SIAM J. Sci. Comput. 42 (2020) B738--B760.
	
	\bibitem{JCP_Shen_2022}
	Y.~Zhang, J.~Shen, A generalized {SAV} approach with relaxation for dissipative
	systems, J. Comput. Phys. 464 (2022) 111311.
	
	\bibitem{JCP_Jiang_2022}
	M.~Jiang, Z.~Zhang, J.~Zhao, Improving the accuracy and consistency of the
	scalar auxiliary variable ({SAV}) method with relaxation, J. Comput. Phys.
	456 (2022) 110954.
	
	\bibitem{JSC_Liu_2024}
	Z.~Liu, Y.~Zhang, X.~Liao, A novel energy-optimized technique of {SAV}-based
	({EOP-SAV}) approaches for dissipative systems, J. Sci. Comput. 101 (2024)
	38.
	
	\bibitem{ANM_2025_Zhang}
	B.~Zhang, C.~Zhou, H.~Fu, A novel efficient generalized energy-optimized
	exponential {SAV} scheme with variable-step {BDF}k method for gradient flows,
	Appl. Numer. Math. 210 (2025) 39--63.
	
	\bibitem{JSC_Qiao_2022}
	L.~Ju, X.~Li, Z.~Qiao, Stabilized exponential-{SAV} schemes preserving energy
	dissipation law and maximum bound principle for the {Allen--Cahn} type
	equations, J. Sci. Comput. 92 (2022) 66.
	
	\bibitem{MOC_Hou_2023}
	D.~Hou, L.~Ju, Z.~Qiao, A linear second-order maximum bound
	principle-preserving {BDF} scheme for the {Allen--Cahn} equation with a
	general mobility, Math. Comp. 92 (2023) 2515--2542.
	
	\bibitem{SCM_Liao_2024}
	H.~Liao, T.~Tang, T.~Zhou, Positive definiteness of real quadratic forms
	resulting from the variable-step ${L}1$-type approximations of convolution
	operators, Sci. China Math. 67 (2024) 237--252.
	
	\bibitem{CiCP_Liao_2021}
	H.~Liao, W.~Mclean, J.~Zhang, A second-order scheme with nonuniform time steps
	for a linear reaction-subdiffusion problem, Commun. Comput. Phys. 30 (2021)
	567--601.
	
	\bibitem{JCM_Tang_2016}
	T.~Tang, J.~Yang, Implicit-explicit scheme for the {A}llen--{C}ahn equation
	preserves the maximum principle, J. Comput. Math. 34 (2016) 471--481.
	
	\bibitem{JSC_Feng_2005}
	X.~Feng, H.~Wu, A posteriori error estimates and an adaptive finite element
	method for the {A}llen--{C}ahn equation and the mean curvature flow, J. Sci.
	Comput. 24 (2005) 2.
	
	\bibitem{JCP_Provatas_1999}
	N.~Provatas, N.~Goldenfeld, J.~Dantzig, Adaptive mesh refinement computation of
	solidification microstructures using dynamic data structures, J. Comput.
	Phys. 148 (1999) 265--290.
	
	\bibitem{JCP_Mackenzie_2002}
	J.~Mackenzie, M.~Robertson, A moving mesh method for the solution of the
	one-dimensional phase-field equations, J. Comput. Phys. 181 (2002) 526--544.
	
	\bibitem{JCP_Feng_2006}
	W.~Feng, P.~Yu, S.~Hu, Z.~Liu, Q.~Du, L.~Chen, Spectral implementation of an
	adaptive moving mesh method for phase-field equations, J. Comput. Phys. 220
	(2006) 498--510.
	
	\bibitem{SISC_Fernando_2022}
	M.~Fernando, H.~Sundar, Scalable local timestepping on octree grids, SIAM J.
	Sci. Comput. 44 (2022) C156--C183.
	
	\bibitem{SISC_Mehlin_2015}
	M.~Grote, M.~Mehlin, T.~Mitkova, Runge-{K}utta-based explicit local
	time-stepping methods for wave propagation, SIAM J. Sci. Comput. 37 (2015)
	A747--A775.
	
	\bibitem{JCP_Alikhanov_2015}
	A.~Alikhanov, A new difference scheme for the time fractional diffusion
	equation, J. Comput. Phys. 280 (2015) 424--438.
	
	\bibitem{SINUM_Mustapha_2014}
	K.~Mustapha, B.~Abdallah, K.~Furati, A discontinuous {P}etrov-{G}alerkin method
	for time-fractional diffusion equations, SIAM J. Numer. Anal. 52 (2014)
	2512--2529.
	
\end{thebibliography}

\end{document}